\documentclass{amsart}          

\usepackage{amsmath,amsthm}     
\usepackage{amssymb,bbm}            
\usepackage{euscript}           
\usepackage{array,enumerate,calc,graphicx,xcolor}
\usepackage[matrix,arrow,curve,frame]{xy}    
\usepackage[colorlinks]{hyperref}
\usepackage{stmaryrd}
\usepackage{tikz-cd}

\usepackage[textwidth=3cm, textsize=small, colorinlistoftodos]{todonotes} 

\makeatletter
\newcommand*\bigcdot{\mathpalette\bigcdot@{.7}}
\newcommand*\bigcdot@[2]{\mathbin{\vcenter{\hbox{\scalebox{#2}{$\m@th#1\bullet$}}}}}
\makeatother

\usepackage{tikz}
\usetikzlibrary{matrix,arrows,decorations.pathmorphing}

\xymatrixcolsep{1.9pc}                          
\xymatrixrowsep{1.9pc}
\newdir{ >}{{}*!/-5pt/\dir{>}}                  

\newtheorem{thm}[subsection]{Theorem}
\newtheorem{defn}[subsection]{Definition}
\newtheorem{prop}[subsection]{Proposition}

\newtheorem{cor}[subsection]{Corollary}
\newtheorem{lemma}[subsection]{Lemma}

\theoremstyle{definition}  
\newtheorem{example}[subsection]{Example}

\newtheorem{remark}[subsection]{Remark}

  {\end{list}}

\newcommand{\dfn}{\textbf} 

\newcommand{\mdfn}[1]{\dfn{\mathversion{bold}#1}} 

\newcommand{\Smash}             {\wedge}

\newcommand{\Wedge}             {\vee}
\DeclareRobustCommand{\bigWedge}{\bigvee}
\newcommand{\tens}              {\otimes}               

\newcommand{\iso}               {\cong}

\newcommand{\cat}{\EuScript}    

\newcommand{\cD}{{\cat D}}
\newcommand{\cE}{{\cat E}}

\newcommand{\Mod}{\text{Mod}}

\newcommand{\field}[1]  {\mathbb #1} 

\newcommand{\HH}         {\field H}
\newcommand{\R}         {\field R}
\newcommand{\M}         {\field M}

\newcommand{\Z}         {\field Z}
\newcommand{\C}         {\field C}
\newcommand{\Q}         {\field Q}

\DeclareMathOperator*{\GL}{GL}

\DeclareMathOperator*{\im}{Im}

\DeclareMathOperator{\Hom}{Hom}

\DeclareMathOperator{\End}{End}

\newcommand{\ra}{\rightarrow}                   
\newcommand{\lra}{\longrightarrow}              
\newcommand{\lla}{\longleftarrow}               
\newcommand{\llra}[1]{\stackrel{#1}{\lra}}      

\newcommand{\inc}{\hookrightarrow}              

\newcommand{\blank}{-}                          
\newcommand{\id}{id}                            
\newcommand{\und}{\underline}

\newcommand{\period}    {{\makebox[0pt][l]{\hspace{2pt} .}}}

\newcommand{\he}{\simeq}
\newcommand{\pt}{pt}

\numberwithin{equation}{section}

\DeclareMathOperator{\Func}{Func}

\newcommand{\MMack}[2]{\xymatrix{
{#1} \ar@(ul,dl)[] \ar@<0.5ex>[r] & {#2}\ar@<0.5ex>[l]}}
\newcommand{\undZ}{\underline{\Z}}

\newcommand{\MMod}{\!-\!\Mod}
\DeclareMathOperator{\coker}{coker}

\newcommand{\mZ}{\underline{\smash{\Z}}}

\DeclareMathOperator{\OC}{OC}
\newcommand{\ray}[1]{\overrightarrow{#1}}
\newcommand{\tH}{\tilde{H}}
\newcommand{\FF}{\mathcal{F}}
\DeclareMathOperator{\stab}{Stab}
\DeclareMathOperator{\gr}{gr}
\DeclareMathOperator{\Rees}{Rees}

\newcommand{\squar}{\circleddash}
\DeclareMathOperator{\Th}{Th}
\newcommand{\ls}{{\lambda\!-\!\sigma}}

\newcommand{\BB}{\mathbb{E}}

\numberwithin{equation}{subsection}

\begin{document}

\title[Bredon cohomology of Euclidean 
configuration spaces]{$RO(G)$-graded Bredon cohomology of Euclidean 
configuration spaces}

\author{Daniel Dugger}
\author{Christy Hazel}

\address{Department of Mathematics\\University of Oregon\\Eugene, OR
  97403}
\email{ddugger@uoregon.edu}

\address{Department of Mathematics\\Grinnell College\\ Grinnell, IA 50112}
\email{hazelchristy@grinnell.edu}

\maketitle

\begin{abstract}
    Let $G$ be a finite group and $V$ be a $G$-representation. We investigate the $RO(G)$-graded Bredon cohomology with constant integral coefficients of the space of ordered configurations in $V$. In the case that $V$ contains a trivial subrepresentation, we show the cohomology is free as a module over the cohomology of a point, and we give a generators-and-relations description of the ring structure. In the case that $V$ does not contain a trivial representation, we give a computation of the module structure that works as long as a certain vanishing condition holds in the Bredon cohomology of a point. We verify this vanishing condition holds in the case that $\dim(V)\geq 3$ and $G$ is any of $C_p$, $C_{p^2}$ ($p$ a prime), or the symmetric group on three letters. 
\end{abstract}

\tableofcontents

\section{Introduction}
\label{se:intro}

Let $G$ be a finite group and let $\und{\Z}$ denote the constant coefficient Mackey functor whose value is $\Z$.  $RO(G)$-graded Bredon cohomology associates to any $G$-space $X$ the graded group  $H^\star(X;\und{\Z})=\bigoplus_{W\in RO(G)} H^{W}(X;\und{\Z})$, which becomes a multigraded ring after fixing an ordered basis of irreducible representations (cf. Section~\ref{se:RO(G)} for discussion of this point). If $V$ is a $G$-representation let $\OC_k(V)$ denote the \dfn{ordered configuration space} of tuples $(x_1,\ldots,x_k)\in V^k$ such that $x_i\neq x_j$ for all $i\neq j$.  The $G$-action on $\OC_k(V)$ is inherited from the action on $V$. Our goal in this paper is to compute the groups $H^\star(\OC_k(V);\und{\Z})$, together with the ring structure when possible.  While we are not able to completely solve this problem in all cases, we have the following results:

\begin{enumerate}[(1)]
\item We compute the additive structure, as a module over $H^\star(\pt;\und{\Z})$, whenever $V\supseteq 1$ (i.e., $V$ contains a copy of the $1$-dimensional trivial representation).  We also give a
generators-and-relations description of the cohomology ring. These methods work for any finite group $G$.

\item For the case where $V\not\supseteq 1$, we give a calculation of the additive structure that depends on certain hypotheses specific to the pair $(G,V)$.  We do not know how widely these hypotheses are satisfied, though we verify them in the cases where $\dim V\geq 3$ and $G$ is any of $C_p$, $C_{p^2}$ ($p$ a prime), and the symmetric group $\Sigma_3$. When $G=C_2$ we are also able to do the calculations for $\dim V= 2$.  For the ring structure, we give an approach that works for $G=C_2$ but may be difficult to carry out in general.
\end{enumerate}

\medskip

To explain the results in more detail, let us first recall the non-equivariant situation.  The spaces $\OC_k(\R^n)$ have been much-studied, classically by Arnold \cite{A} when $n=2$ and more generally by Cohen  \cite[Chapter III, Sections 6--7]{CLM} building off of work from \cite{FN}.  (Some of these classical papers use the notation $F(\R^n,k)$ for the ordered configuration space.)  There are also many modern references, just a few of which are \cite{LS}, \cite{G}, and \cite{S}.
For $n\geq 2$ the integral cohomology ring of $\OC_k(\R^n)$ is the quotient of the polynomial ring
$\Z[\omega_{ij}\,\bigl | \, 1\leq i\neq j\leq k]$ by the three sets of relations
\begin{align*}
&\omega_{ij}=(-1)^n\omega_{ji},\\
&\omega_{ij}^2=0,\\
&\omega_{ij}\omega_{jk}+\omega_{jk}\omega_{ki}+\omega_{ki}\omega_{ij}=0.
\end{align*}
The third of these  will be called the \dfn{Arnold relation}, as it seems to have first appeared in \cite{A}.  Often the first relation is omitted and the generators are just taken to be $\omega_{ij}$ for $i<j$, but having both $\omega_{ij}$ and $\omega_{ji}$ is convenient for the third relation and will also be convenient when we generalize to the equivariant situation. [Note: Readers curious about the $n=1$ case can look at Section~\ref{se:V-G}.]

Since the $\omega_{ij}$ classes are easy to describe, we do this here.  
One has maps $\tilde{\omega}_{ij}\colon \OC_k(\R^n)\ra S^{n-1}$ sending a configuration $\und{x}=(x_1,\dots, x_n)$ to $\frac{x_i-x_j}{|x_i-x_j|}$, and $\omega_{ij}$ is just the pullback of the fundamental class from the cohomology of the sphere: $\omega_{ij}=\tilde{\omega}_{ij}^*(\iota_{n-1})$.  The first two of the three relations then follow immediately.  There are different ways to derive the Arnold relation (the original in \cite{A} using differential forms), but we will give a very geometric/topological approach in Section~\ref{se:gen-rels} below.

The above presentation implies that $H^*(\OC_k(\R^n))$ is nonzero only in degrees that are multiples of $n-1$, and in those degrees it is free abelian.  It is not exactly transparent from the relations, but the rank of $H^{i(n-1)}(\OC_k(\R^n))$ is the Stirling number of the first kind $c(k,k-i)$.  Here $c(a,b)$ is the number of permutations of $a$ that can be expressed as a product of $b$ disjoint cycles (see Section~\ref{se:Stirling} for other descriptions of the Stirling numbers).  The total rank of $H^*(\OC_k(\R^n))$ is therefore $k!$.  In Section~\ref{se:denouement} we will see that both of these facts have interesting explanations in terms of equivariant cohomology, an observation that is due to Proudfoot \cite{P}.

\vspace{0.2in}

Now let us move to Bredon cohomology for $G$-equivariant spaces.  
We write $\M$ for $H^\star(\pt;\und{\Z})$, which is the ground ring for our $RO(G)$-graded theory.  Every orthogonal representation $W$ has a so-called Euler class $a_W\in \M^W$ (see Section~\ref{se:euler}).  These classes are in the center of $\M$ and satisfy $a_{W_1\oplus W_2}=a_{W_1}\cdot a_{W_2}$ and  $a_1=0$.  Note that this implies $a_W=0$ whenever $W\supseteq 1$.   

Before stating our results let us recall that all representations are assumed to be orthogonal, and if $W_1\subseteq W_2$ then $W_2-W_1$ denotes the orthogonal complement.  If $W\supseteq 1$ then when we write $W-1$ we assume that a specific $1$-dimensional trivial subrepresentation has been chosen in $W$.

Our first main result is as follows:

\newtheorem*{th:main1}{Theorem \ref{th:main1}}
\begin{th:main1}
If $V\supseteq 1$ and $\dim V\geq 2$ then $H^\star(\OC_k(V);\und{\Z})$ is the quotient of the free $\M$-algebra generated by classes $\omega_{ij}$ of degree $V-1$, $1\leq i\neq j\leq k$, subject to the following relations:
\begin{align*}
&\omega_{ij}=(-1)^{|V|}\omega_{ji}+a_{V-1}, \\
&\omega_{ij}^2=a_{V-1}\omega_{ij}, \\
&\omega_{ij}\omega_{jk}+\omega_{jk}\omega_{ki}+\omega_{ki}\omega_{ij}=a_{V-1}(\omega_{ij}+\omega_{jk}+\omega_{ki})-a_{V-1}^2. 
\end{align*}
\end{th:main1}

\noindent
Observe that if $V\supseteq 2$ then the Euler classes $a_{V-1}$ all vanish and the above relations reduce to the non-equivariant form.  \medskip

Despite the similarity of the above result with the non-equivariant version, our proof is notably different in style.  In the non-equivariant setting the map $\OC_{k+1}(\R^n)\ra \OC_k(\R^n)$ that forgets the last point in the configuration is a fibration whose fiber is (up to homotopy) a wedge of spheres, and so one can use the Serre spectral sequence to inductively do the computations.  In the equivariant setting we are hampered by the fact that the Serre spectral sequence is a much less usable tool, with even basic computations seeming to require an extensive knowledge of cohomology with local coefficients.  Instead of going this route we study the ``motive'' $H\und{\Z}\Smash \OC_k(V)$ and build this up inductively via a collection of cofiber sequences.  This only works because the boundary maps in these cofiber sequences turn out to always vanish, by somewhat of a miracle.  

For the case where $V\not\supseteq 1$ we can use essentially the same techniques, but here we get less lucky and the ``miracle'' does not come for free.  We are able to prove that the boundary maps vanish in some familiar cases, but so far not in general.  To state our result here, let $S(V)$ be the unit sphere inside the representation $V$.  These spheres are key to the calculations because the $\tilde{\omega}_{ij}$ maps take the form $\OC_k(V)\ra S(V)$.  The issue we run into is that not much is known about $H^\star(S(V);\mZ)$ in the case $V\not\supseteq 1$, as this object is intricately related to $H^\star(\pt;\mZ)$, which also can be mysterious.  

\newtheorem*{th:main-2}{Theorem \ref{th:main-2}}

\begin{th:main-2}
Let $G$ be a finite group and suppose that $V$ is a $G$-representation such that $\dim V\geq 2$ and $V^G=0$.  Additionally, assume for all $\ell\in \Z$ that in the sequence
\[ H^{\ell V-\ell}(\pt)\llra{\cdot a_V} H^{(\ell+1)V-\ell}(pt)
\llra{\cdot a_V}
H^{(\ell+2)V-\ell}(\pt)
\]
the first map is surjective and the second map is injective.
[These conditions are equivalent to $H^{\ell V-(\ell-1)}(S(V))=0$ for all $\ell\in \Z$.]  
Then there is a splitting of $H\mZ$-modules
\[ H\mZ\Smash \OC_k(V)_+\he \bigvee_{j=0}^{k-1}(H\mZ\Smash \Sigma^{j(V-1)}S(V)_+)^{a(k,j)}
\]
where $a(k,j)$ is an alternating sum of Stirling numbers given in Definition~\ref{aseq} in Section~\ref{se:homology}.  
\end{th:main-2}

This theorem gives an additive splitting for $H^\star(\OC_k(V);\mZ)$ in terms of shifted copies of $H^\star(S(V);\mZ)$, but the latter remains a black box.  The hypotheses of Theorem~\ref{th:main-2} are an interesting property of the ground ring $H^\star(\pt;\mZ)$ that seems to be worthy of further investigation.  As test cases we verify that these hypotheses hold when $\dim V\geq 3$ and $G$ is any of $C_p$, $C_{p^2}$, and $\Sigma_3$---see Appendix~\ref{se:verify} for details.  The case $\dim V=2$ is an anomaly and the hypotheses are almost never satisfied here when $\ell=-2$; again, details are in Appendix~\ref{se:verify}.  Despite this, in the special case $G=C_2$ we verify that Theorem~\ref{th:main-2} still holds when $\dim V=2$.  See Proposition~\ref{pr:V=2}.

We are likewise unable to give a simple presentation of the 
multiplicative structure on $H^\star(\OC_k(V);\mZ)$ in the case $V\not\supseteq 1$, at least for general $G$.  In Section~\ref{se:ring-special-case}
we discuss some of the issues involved and give a solution when $G=C_2$. 
\medskip

Our primary motivation in this paper was the study of Bredon cohomology, with configuration spaces being a convenient test area for computations.
But the results also offer an interesting perspective on some 
non-equivariant phenomena.  The relations defining the singular cohomology rings of the spaces $\OC_k(\R^n)$ only depend on the parity of $n$, but because the degrees of the generators depend on $n$ it is not obvious how to give a direct comparison between the rings. For example, it is of course not true that there is a map of spaces $\OC_k(\R^n)\ra \OC_k(\R^{n+2})$ that induces an isomorphism on singular cohomology, or a map $\OC_k(\R^n)\ra \OC_k(\R^{n+1})$ that induces an isomorphism with mod $2$ coefficients.  

However, the use of Bredon cohomology {\it does\/} lead to some direct connections here.  For example, take $G=C_2$ and $V=\R^n\oplus \R_-$, where $\R^n$ has the trivial action and $\R_-$ has the sign action.  The inclusion $i\colon \OC_k(\R^n)\inc \OC_k(V)$ {\it does\/} induce an interesting map on (Bredon) cohomology, and the diagram
\[ \xymatrixcolsep{2pc}\xymatrix{
&& H^\star(\OC_k(V);\mZ) \ar[dl]_{i^*}\ar[d]^\psi \\
H^*_{sing}(\OC_k(\R^n))\tens \M \ar[r]^-\iso & 
 H^\star(\OC_k(\R^n);\mZ) & H^*_{sing}(\OC_k(\R^{n+1});\Z)
}
\]
where $\psi$ is the ``forgetful map'' from Bredon to singular cohomology
allows {\it some\/} information to be passed between $H^*_{sing}(\OC_k(\R^n))$ and $H^*_{sing}(\OC_k(\R^{n+1}))$.  One can play a similar game when $V$ has two or more copies of $\R_-$ added on instead of one.  This phenomenon had already been noticed in the context of Borel equivariant cohomology in the work of Proudfoot and his collaborators \cite{P, MPY, DPW}, and we include discussion of the Bredon case not because it leads to any groundbreaking insights but just as a demonstration of the inner-workings of our computations. See Section~\ref{se:comparing} for discussion of this topic. 

\begin{remark}
For other recent work on Bredon homology and configuration spaces, see \cite{BQV}.  That paper focuses on Bredon homology of {\it unordered} configuration spaces, and so the results are in a somewhat different direction than the ones here.
\end{remark}

\subsection{Notational conventions} Throughout this paper $G$ denotes a finite group and $V$, $W$ denote finite-dimensional, orthogonal $G$-representations. The dimension is written $|V|$ or $\dim V$, depending on context.  If $W$ is an orthogonal subrepresentation of $V$, then we write $V-W$ for the orthogonal complement. In the case that $V$ contains a one-dimensional trivial representation, we will write $V-1$ to indicate the orthogonal complement of some choice of trivial subrepresentation. For a representation $V$ we write $S(V)$ for the unit sphere in $V$, $D(V)$ for the closed unit disk in $V$, and $S^V$ for the one-point compactification of $V$. We will sometimes use the fact that $D(V)/S(V)\iso S^V$, for example via the isomorphism $x\mapsto \tan(\frac{\pi}{2}|x|)\cdot x$ which is canonical in $V$.  

To the representation $V$ we assign numerical invariants $d(V)$ and $e(V)$ that appear throughout the paper.  The definitions are in Sections~\ref{se:euler} and \ref{se:orient}.

We use $\star$ to denote $RO(G)$-gradings and $\ast$ to denote integer gradings. 

\subsection{Acknowledgments} We are grateful to Nick Proudfoot for conversations regarding the work in \cite{DPW, Mo, MPY, P}. The second author is also thankful to Mike Hill for some helpful conversations about Bredon cohomology.

\section{Background}
\label{se:background}

In this section we review background information about Bredon cohomology, and also establish some fundamental results about both the Bredon cohomology of a point and the Bredon cohomology of spheres.  

\medskip

\subsection{\mdfn{$RO(G)$}-graded Bredon cohomology}
\label{se:RO(G)}

For foundational material on $RO(G)$-graded Bredon cohomology we refer the reader to \cite{M}.  We follow the now-common practice of using $*$ to denote integer gradings and $\star$ to denote $RO(G)$-gradings, so that the notation  $H^\star(X;\mZ)$ indicates $\star\in RO(G)$.

We will use that $H^W(\blank;\und{\Z})$ is represented by the space $AG(S^W)$, the free abelian group on $S^W$ (suitably topologized) with the basepoint $\infty$ as the zero element.   This fact is due to dos Santos \cite{dS}, building on earlier work of Lima-Filho \cite{LF}.  
The assignment $W\mapsto AG(S^W)$ defines an equivariant ring spectrum denoted $H\mZ$.  

Defining $RO(G)$-graded cohomology rings requires a certain amount of care, and it is common practice to sweep some of the subtleties under the rug.  But since signs will be important for us, we need to give a brief synopsis.  Our overall approach follows that of \cite{D2} but with a few modifications.  

Fix once and for all an ordered collection of irreducible representations $I_1,\ldots,I_r$ giving a basis for $RO(G)$.  For each $j$ fix a dual of $S^{I_j}$ in the equivariant stable homotopy category and denote it $S^{-I_j}$.  For $n_1,\ldots,n_r\geq 0$ define 
\[ S^{\pm n_1I_1\pm \ldots\pm n_rI_r}=(S^{\pm I_1})^{\Smash(n_1)}\Smash (S^{\pm I_2})^{\Smash(n_2)}\Smash \cdots \Smash (S^{\pm I_r})^{\Smash(n_r)}.
\]
For $\alpha\in RO(G)$ we write $S^\alpha$ for $S^{m_1I_1+\cdots+m_rI_r}$ where the $m_i$ are the unique integers for which $\alpha=\sum m_iI_i$.
Define
\[ H^{\alpha}(X;\mZ)=[S^{-\alpha}\Smash X_+,H\mZ] \quad
\text{and} \quad H^\star(X;\mZ)=\oplus_{\alpha\in RO(G)} H^\alpha(X;\mZ).\]

For a representation $V$ write $\und{V}$ for the corresponding element of $RO(G)$.  Note that $S^V$ and $S^{\und{V}}$ are homeomorphic, but not canonically.  Likewise, the group $H^V(X;\mZ)=[S^{-V}\Smash X_+,H\mZ]$ is isomorphic to $H^{\und{V}}(X;\mZ)$, but again not canonically.  If $\und{V}=\sum n_jI_j$, define a \dfn{rigidification} of $V$ to be an isomorphism $V\llra{\iso} \bigoplus_j I_j^{\oplus n_j}$.  Such a rigidification determines a homeomorphism $S^V\llra{\iso} S^{\und{V}}$.  By a \dfn{rigid representation} we mean a representation equipped with a chosen rigidification.

To define the multiplication on $H^\star(X;\mZ)$ we need to choose identifications $\phi_{\alpha,\beta}\colon S^{-\alpha}\Smash S^{-\beta}\llra{\iso} S^{-(\alpha+\beta)}$ for all $\alpha,\beta\in RO(G)$. Then for $x\in H^\alpha(X)$ and $y\in H^\beta(X)$ we define $xy$ to be
the composite
\[ \xymatrix{
S^{-(\alpha+\beta)}\Smash X_+ \ar[d]^{\phi^{-1}_{\alpha,\beta}\Smash \Delta} \\
S^{-\alpha}\Smash S^{-\beta}\Smash X_+\Smash X_+ \ar[r]^{1\Smash t\Smash 1} & S^{-\alpha}\Smash X_+\Smash S^{-\beta}\Smash X_+ \ar[r]^-{x\Smash y} &H\mZ\Smash H\mZ\ar[r]^-\mu & H\mZ.
}
\]
We choose the isomorphisms $\phi$ essentially as described in \cite{D2} but with one variation.  Said briefly, they are obtained by the following rules:
\begin{itemize}
\item Commute any $S^{\pm I_j}$ past $S^{\pm I_k}$ for $j\neq k$;
\item Allow any $S^{I_j}$ to annihilate an $S^{-I_j}$ that is next to it, via the duality maps;
\item Every time we commute $S^{\pm I_j}$ past $S^{\pm I_k}$ for $j\neq k$ we multiply by the sign $(-1)^{|I_j|\cdot |I_k|}$.
\end{itemize}
The third rule was {\it not\/} used in \cite{D2}, but putting it in has the effect of making the connection to non-equivariant topology cleaner (more on this in a moment).  The inclusion of these signs was a topic of \cite{DDIO}. For the resulting product to be associative one needs the $\phi$ maps to satisfy a certain coherence condition (see \cite{D2}), but that is true with these choices.  For $x\in H^\alpha(X)$ and $y\in H^\beta(X)$ one also finds the
skew-commutativity rule
\begin{equation}
\label{eq:skew-commute}
xy=yx\cdot (-1)^{|\alpha|\cdot |\beta|} 
\end{equation}
where if $\alpha=\sum n_jI_j$ then $|\alpha|=\sum n_j |I_j|$.

\begin{remark}
If one uses $H^V$ rather than $H^{\und{V}}$ then the need for the $\phi$-maps disappears, because one has the canonical isomorphism $S^V\Smash S^W\iso S^{V\oplus W}$.  To extend this to virtual representations one then needs for the notation $H^{V-W}$ to depend on both $V$ and $W$ and not just on the difference in $RO(G)$; such a group would be better written as $H^{(V,W)}$.  One then obtains a theory that is not exactly ``$RO(G)$-graded'' but where the indexing is on pairs of representations.  In this setting products like $xy$ and $yx$ live in different groups---e.g. $H^{V_1\oplus V_2}$ rather than $H^{V_2\oplus V_1}$---and so are not just related by a sign.  Rather, they are related by a certain twist isomorphism induced by the twist isomorphism on vector spaces.   This kind of ``fancy'' grading is indeed what people often mean when they say ``$RO(G)$-graded''.  On the one hand, it is an appealing Gordian-knot style approach to the situation.  On the other hand, in practical computations one often wants to work with a concrete $\Z^r$-graded ring with formulas such as (\ref{eq:skew-commute}), rather than a behemoth indexed by all representations.

Here is an example of how these issues play out in practice.  In Proposition~\ref{pr:H-S^W} below we will describe the cohomology ring of $S^W$ for any representation $W$.  To give the answer in the usual language of algebra, e.g. via generators and relations, one is pushed into the context of honest $RO(G)$-gradings.  The downside of that context is that there is no canonical generator in $\tH^{\und{V}}(S^V)$.  If we instead use the ``fancy'' $RO(G)$-grading then we have available the canonical element in $\tH^V(S^V)$, but the complexity of the grading makes  it awkward to give a complete algebraic description of the answer.  As a compromise we could assume that $V$ is a rigid representation, which gives a canonical generator in $\tH^{\und{V}}(S^V)$, but the extra assumption of rigidity can itself feel awkward and unsatisfying.

We refer the reader to \cite[Remark 4.11]{L} and \cite[Chapter XIII]{M} for further discussion of these issues.  The approach we will take
is something like the following: merge the perspectives of both the honest and fancy $RO(G)$-gradings whenever convenient, but keep in the backs of our minds that when using the former one might need to add assumptions about rigidity and when using the latter one might need to add in the effects of certain isomorphisms.
\end{remark}

As a final remark on this topic we recall the forgetful map $\psi$ from equivariant homotopy to non-equivariant homotopy.  At the basic level of spaces this sends a $G$-space $X$ to the underlying space without the $G$-action, inducing the evident map  $\psi\colon [X,Y]_G\ra [X,Y]$ on homotopy classes.  Note that $\psi(S^V)\iso S^{|V|}$, but not canonically.  Even $\psi(S^{\und{V}})\iso S^{|V|}$ is not canonical.  However, if we fix once and for all an orientation on each $I_j$ then we at least get homeomorphisms $S^{I_j}\iso S^{|I_j|}$ that are canonical up to homotopy, and therefore resulting homotopy equivalences $\psi(S^{\und{V}})\he S^{|V|}$ (again canonical up to homotopy).  One then obtains  induced maps $\psi\colon H^{\alpha}(X;\mZ)\ra H^{|\alpha|}_{sing}(X;\Z)$.  Unlike in \cite{D2}, these $\psi$'s assemble to  give a ring homomorphism---this is because of the signs that were included in the $\phi$-maps above, and is the justification for incorporating them.

\subsection{Computational tools}
We will need the following two properties of $H^\star(\blank;\mZ)$.  

\begin{prop}[Quotient Lemma]\label{quotientlem} Let $X$ be a $G$-CW complex. Then $\tilde{H}^{n}(X;\undZ)\cong \tilde{H}^{n}_{sing}(X/G;\Z)$.
\end{prop}

The above is a standard property, following immediately from $H^n(X;\mZ)=[X,AG(S^n)]_G=[X/G,AG(S^n)]$ since the $G$-action on $AG(S^n)$ is trivial.  

\begin{prop}
\label{pr:trivial-action}
If $X$ is a space with trivial $G$-action and $H^*_{sing}(X)$ is free abelian, then $H^\star(X;\mZ)\iso H^*_{sing}(X)\tens \M$.
\end{prop}

\begin{proof}
This follows immediately from the usual cellular spectral sequence, which can be regarded as a spectral sequence of $\M$-modules.  All differentials vanish except $d_1$.   
\end{proof}

\subsection{Euler classes and the cohomology of a point} 
\label{se:euler}

A representation $W$ has an associated ``Euler class'' $a_W\in \M^W$, and when $W$ is $G$-oriented there is an ``orientation class'' $u_W\in \M^{W-|W|}$.  These classes are described in \cite[Section 3]{HHR1} but most likely date earlier and are standard constructions.  We review them here and establish some basic properties that seem not to be in the literature.  

Let $a_W\colon S^0\to S^W$ be the map that  sends the basepoint to $\infty$ and the nonbasepoint to $0$.   
This gives an element in $\pi_{-W}(S^0)$.
If we compose with the canonical inclusion $S^W\inc AG(S^W)$ then we also get a cohomology class in 
$\tH^W(S^0;\mZ)\iso H^W(pt;\underline{\Z})$. In a slight abuse of notation we will denote both the homotopy class and the cohomology class by 
$a_W$.  In either setting, $a_W$ is commonly called the \dfn{Euler class} of $W$.

If $1\subseteq W$ then the map $S^0\to S^W$ is equivariantly null and so $a_W=0$.  The following result generalizes this by precisely determining the order of $a_W$ in all other cases:

\begin{prop}
\label{pr:a-group}
The group $H^W(\pt;\mZ)$ is generated by $a_W$ and has order equal to $\gcd\{ \#(G/H)\,\bigl |\, H\leq G, \ W^H\neq 0\}$.
\end{prop}

In particular, note that if $1\subseteq W$ then the gcd is $1$ and so $H^W(\pt;\mZ)=0$.  

\begin{proof}[Proof of  Proposition~\ref{pr:a-group}]
We first show $a_W$ generates the group. We use the isomorphisms
\[ H^W(\pt;\mZ)\iso [S^0,AG(S^W)]_G\iso \pi_0(AG(S^W)^G).
\]
Elements of the object on the right 
are represented by $G$-equivariant finite formal sums $\sum n_i[x_i]$ with $x_i\in S^W$ and $n_i\in \Z$, with the understanding that the term $[\infty]$ represents the zero element and so can be dropped from any formal sum. The homotopy that contracts $S^W-\{\infty\}$ to $0$ then shows that every formal sum can be deformed to be a multiple of $[0]$. The element $a_W$ is the formal sum $1[0]$ and so our cohomology group is cyclic and $a_W$ is a generator.

We can also use this isomorphism to see why $a_W$ is annihilated by the gcd. Suppose that $H\leq G$ and $W^H\neq 0$.  Let $x\in W^H-\{0\}$, and set $\alpha=\sum_{gH\in G/H} [gx]$.  This is a $G$-equivariant formal sum that doesn't involve the term $[0]$.  The homotopy that contracts all  points in $S^W-\{0\}$ to $\infty$ shows that $\alpha \sim 0$, whereas the homotopy that contracts all points in $S^W-\{\infty\}$ to $0$ shows that $\alpha \sim \#(G/H)[0]$.  So we find that $0=\#(G/H)\cdot a_W$.  Since this holds for all possible $H$, $a_W$ is annihilated by the gcd from the statement of the proposition.  

To show the order is exactly equal to the gcd, we use the isomorphisms 
\[H^W(\pt;\mZ)\iso H_{-W}(\pt;\mZ)\iso \tH_0(S^W;\mZ).\] To compute this group, we build a $G$-CW complex $Y$ with a weak equivalence $Y\to X=S^W$ and then use the cellular chain complex  of $Y$.  Start by putting in two fixed $0$-cells---the points $0$ and 
$\infty$---which gives a $0$-skeleton $Y_0$ and an inclusion $Y_0\to X$ such that $\pi_0(Y_0^H)\ra \pi_0(X^H)$ surjective for every $H$.  Next, for every $H\leq G$ such that $W^H\neq 0$ pick an element $v\in W^H-0$ and add a $1$-cell $e_v$ to $Y_0$ of type $G/H$ that connects the two $0$-cells. Then add a second $1$-cell called $e_{-v}$ of the same type and with the same boundary.  This creates a $1$-skeleton $Y_1$ together with a map $Y_1\ra X$ that sends each $e_{\pm v}$ to the ray from $0$ to $\infty$ that passes through $\pm v$.
The induced maps are such that $\pi_0(Y_1^H)\ra \pi_0(X^H)$ is an isomorphism and $\pi_1(Y_1^H)\ra \pi_1(X^H)$ is a surjection, for every $H$ (for the latter, remember that each $X^H$ is a sphere and so has $\pi_1$ equal to either $0$ or $\Z$ depending on the dimension).  Continue by adding $2$-cells and higher; the exact details are irrelevant to the $H_0$ calculation.  The reduced Bredon chain complex for $Y$ is then
\[ \cdots \lra \bigoplus_{H\leq G, W^H\neq 0} \Z^2  \lra \Z \lra 0
\]
with the generators of the summands corresponding to $H\leq G$ each being sent to $\#(G/H)$ in the target.  So $\tH_0(S^W)$ is as claimed. 
\end{proof}

For convenience of future use  set $\cD(W)=\{\#(G/H)\,\bigl |\, H\leq G, W^H\neq 0\}$ and $d(W)=\gcd \cD(W)$.

\begin{remark}
The following example shows that the gcd in Proposition~\ref{pr:a-group} cannot be replaced by a min.  When $G=\Sigma_3$, let $\sigma$ be the one-dimensional sign representation and let $\lambda$ be the irreducible representation on $\R^2=\C$ that permutes the cube roots of unity.  Take $W=\sigma\oplus \lambda$.  Then $\cD(W)=\{2,3,6\}$ and so $d(W)=1$.
\end{remark}

Here is another useful description of $d(W)$:

\begin{prop}
\[ d(W)=\gcd \bigl \{
\#(G/H)\,\bigl | \,\text{$H\leq G$ 
and $G/H$\! embeds into $W-\{0\}$} 
\bigr \}.
\]
\end{prop}

\begin{proof}
Let $\cE(W)$ be the set from the statement of the proposition.  Clearly $\cE(W)\subseteq \cD(W)$.  But if $H\leq G$ and $W^H\neq 0$, pick $x\in W^H-\{0\}$ and set $J=\stab(x)$.  Then $H\subseteq J$ and $G/J$ embeds into $W-\{0\}$.  Since $\#(G/J)\cdot \#(J/H)=\#(G/H)$, this proves that every element of $\cD(W)$ is a multiple of an element of $\cE(W)$.  Since $\cE(W)\subseteq \cD(W)$, the two sets will have the same gcd.  
\end{proof}

\begin{example}
Let the generator of the cyclic group $C_n$  act on $W=\R^2$ by rotation through $\frac{2\pi}{n}$ radians.  Then $d(W)=n$, and so this example shows that the order of $H^W(\pt;\mZ)$ can be arbitrarily large.

More generally, if $G$ is arbitrary and the action on $W-\{0\}$ is free, then $d(W)=\#G$. 
\end{example}

In contrast to the above example, we have the following result for odd-dimensional representations:

\begin{prop}
\label{pr:a_odd}
If $\dim W$ is odd then $d(W)$ is either $1$ or $2$.  Equivalently, for any representation $W$ the integer $1-(-1)^{|W|}$ annihilates $a_W$.
\end{prop}

It is possible to give a proof of this result that uses only Smith Theory and elementary algebra---we include this in Appendix~\ref{se:app} below.  
There is also a surprisingly simple approach that instead uses only Bredon cohomology---see Remark~\ref{re:a_odd}.  The Bredon proof is contingent on getting certain signs correct, and so we have also included the Smith Theory approach as a reality-check.

\begin{cor}
\label{co:a-commute}
For any $G$-space $X$, any $x\in H^\star(X;\mZ)$, and any $W$,  we have the strict commutativity $a_Wx=xa_W$.
\end{cor}

\begin{proof}
It suffices to prove this when $x$ is homogeneous, say of degree $\alpha\in RO(G)$. 
Then skew-commutativity says that $a_Wx=(-1)^{|\alpha|\cdot |W|}xa_W$.  If $|W|$ is even the sign is $+1$, but if $|W|$ is odd then 
by Proposition~\ref{pr:a_odd} $a_W=-a_W$ and so the sign can be interpreted as $+1$ even if it isn't.  
\end{proof}

As one final observation on this topic we record the following:

\begin{prop}
\label{pr:euler-mult}
For any representations $W_1$ and $W_2$,
the multiplication pairing 
\[ H^{W_1}(\pt;\mZ)\tens H^{W_2}(\pt;\mZ)\ra H^{W_1\oplus W_2}(\pt;\mZ)
\]is an isomorphism.  
\end{prop}

\begin{proof}
Recall the isomorphism $H^W(\pt;\mZ)\iso \pi_0 \bigl ( AG(S^W)^G\bigr)$, with elements of the last group represented by formal sums.    
In terms of these formal sums, the pairing takes a sum $\sum m_i[x_i]$ on $S^{W_1}$ and a sum $\sum n_j[y_j]$ on $S^{W_2}$ and sends this to the formal sum $\sum_{i,j} m_in_j [(x_i,y_j)]$, with the ordered pair suitably interpreted if either coordinate is $\infty$.  From this description one sees immediately that $a_{W_1}\tens a_{W_2}$ is sent to $a_{W_1\oplus W_2}$, and so the pairing is a surjection.  It then suffices to check that the orders of the domain and codomain are equal.  But since $(W_1\oplus W_2)^H=W_1^H \oplus W_2^H$ it follows at once that $d(W_1\oplus W_2)=\gcd\{d(W_1),d(W_2)\}$, and this is precisely the order of $\Z/d(W_1)\tens \Z/d(W_2)$ (note that we are using Proposition~\ref{pr:a-group} in several places here).
\end{proof}

By the above work we understand the groups $H^\alpha(\pt;\mZ)$ when $\alpha$ is a positive element of $RO(G)$---i.e., $\alpha$ is an actual representation.  We will occasionally also need to know about the case where $\alpha$ is a {\it negative} element of $RO(G)$, but that case is much easier:

\begin{prop}
\label{pr:negative-W}
If $W$ is a nonzero $G$-representation then $H^{-W}(\pt;\mZ)=0$.  
Moreoever, if $W\neq 1$ then $H^{1-W}(\pt;\mZ)=0$.  
\end{prop}

\begin{proof}
For the first part we use that $H^{-W}(\pt;\mZ)\iso \tH^{-W}(S^0;\mZ)\iso \tH^0(S^W;\mZ)\iso \tH^0_{sing}(S^W/G)=0$.  The third isomorphism is by the Quotient Lemma (\ref{quotientlem}), and the last equality holds because $S^W$ is connected and therefore $S^W/G$ is as well.

Similarly, use that $S^W$ is the unreduced suspension of $S(W)$ to get that
\[ H^{1-W}(\pt)=\tH^1(S^W)\iso \tH^1(\Sigma_u S(W))=\tH^1_{sing}(\Sigma_u(S(W)/G))=\tH^0_{sing}(S(W)/G).
\]
As long as $W\neq 1$ the space $S(W)/G$ will be connected and so the above group vanishes.
\end{proof}

\subsection{The orientation classes}
\label{se:orient}

We say that a representation $V$ is \dfn{orientable} if each element of $G$ acts on $V$ with positive determinant.  This is equivalent to saying that the induced $G$-action on $H_{d}^{sing}(S^V)$ is trivial where $d=\dim V$, and also to the analogous statement for cohomology.   An \dfn{orientation} of $V$ is a choice of generator for $H_{d}^{sing}(S^V)$.  

When $V$ is orientable the forgetful map $\psi\colon H_d(S^V)\ra H_d^{sing}(S^V)$ is an isomorphism.  This is explained in \cite[Example 3.9]{HHR1}.   A choice of orientation for $V$ therefore determines an element $\mathfrak{o}_V\in H_d(S^V)$.   We can obtain a corresponding element in cohomology via the isomorphisms in the diagram

\[ \xymatrixcolsep{0.7pc}\xymatrix{
H_d(S^V) \ar[d]_\psi\ar[r]^-\iso & [S^d,S^V\Smash H\mZ] \ar[d]\ar[r]^-\iso & [S^{0},S^{V-d}\Smash H\mZ] \ar[r]^-\iso\ar[d] & H^{V-d}(\pt) \ar[d]\ar[dr]^\psi\\
H_d^{sing}(S^V) \ar@{=}[r] & [G_+\!\Smash\! S^d,S^{V}\Smash H\mZ] \ar[r]^-\iso & [G_+,S^{V-d}\Smash H\mZ] \ar[r]^-\iso & H^{V-d}(G) \ar@{=}[r] & H^{0}_{sing}(\pt).
}
\]
Here the vertical maps are induced by $G_+\ra S^0$.  The element in $H^{V-d}(\pt)$ obtained by pushing $\mathfrak{o}_V$ across the top row is denoted $u_V$.  Note that $\psi(u_V)=1$.

\begin{remark}
Warning: The $\psi$ at the right of the above diagram is not well-defined unless $V$ is oriented, since it requires an identification $S^V\iso S^d$ in the classical stable homotopy category.  So the orientation is coming up both in the choice of generator and in the definition of $\psi$, which is why the choices ``cancel out'' to give $\psi(u_V)=1$.  The orientation is also used in the identification $H^{V-d}(G)=H^0_{sing}(\pt)$.  
\end{remark}

Let $G_V$ be the isotropy group of $V$, and define $e(V)=\#(G/G_V)$.  This measures the size of a generic orbit in $V$; alternatively, it is the degree of the branched cover $S^V\ra S^V/G$.  Note that $e(V\oplus W)$ is always a common multiple of $e(V)$ and $e(W)$.   As an example to see how the $e$-invariant compares to other invariants we have considered, let $G=\Sigma_3$, let $\sigma$ be the sign representation, and let $\lambda$ be the irreducible representation on $\R^2$ as the symmetries of an equilateral triangle.  We have
\[ \cD(\sigma)=\{2,6\}, \quad d(\sigma)=2, \quad e(\sigma)=2;
\]
\[ \cD(\lambda)=\{3,6\}, \quad d(\lambda)=3, \quad e(\lambda)=6;
\]
\[ \cD(\sigma\oplus\lambda)=\{2,3,6\}, \quad d(\sigma\oplus\lambda)=1,\quad e(\sigma\oplus\lambda)=6.\]
It is always true that $e(V)\in \cD(V)$, and so $d(V)|e(V)$.  

If $V$ is oriented then $H^{d-V}(\pt)\iso \Z$ and is generated by an element denoted $\frac{e(V)}{u_V}$, which has the evident property that
\[ u_V\cdot \frac{e(V)}{u_V}=e(V)\cdot 1 \in H^0(\pt).
\]
This follows from using the isomorphisms $H^{d-V}(\pt)\iso \tH^d(S^V)\iso \tH^d_{sing}(S^V/G)$ and then proving that the map on degree $d$ singular cohomology induced by the projection $S^V\ra S^V/G$ is $\Z\ra \Z$ sending a generator to $e(V)$ times a generator.  Alternatively, the class $\frac{e(V)}{u_V}$ arises as the transfer of a certain class $e_V\in H^\star(G/G_V)$; see \cite[Definition 9.9.7 and Lemma 9.9.10]{HHR4}.  We will not need the class $e_V$ in what follows.

\subsection{The cohomology of representation spheres}
We next turn to the cohomology of spheres $S^W$.  The inclusion $S^W\inc AG(S^W)$ gives a canonical element $\iota_W\in \tH^W(S^W;\mZ)$, and the $RO(G)$-graded suspension theorem implies that additively $H^\star(S^W;\mZ)$ is the free $\M$-module generated by $1$ and $\iota_W$.  It only remains to specify the product structure:

\begin{prop}
\label{pr:H-S^W}
$H^\star(S^W;\underline{\Z})$ is the free $\M$-algebra generated by $\iota_W$ subject to the relation $\iota_W^2=a_W\iota_W$.
\end{prop}

Again, observe that if $1\subseteq W$ then $a_W=0$ and we obtain the familiar relation $\iota_W^2=0$ from non-equivariant topology.  Also note that ``$R$-algebra'' typically means a ring homomorphism $R\ra S$ where $R$ acts centrally, and in our context ``centrally'' means in the skew-commutative sense of (\ref{eq:skew-commute}).  But also recall that $a_W\iota_W=\iota_Wa_W$ by Corollary~\ref{co:a-commute}.

\begin{proof}[Proof of Proposition~\ref{pr:H-S^W}]
Write $\iota=\iota_W$ and $K_W=AG(S^W)$.
The product $\iota^2$ is represented by the following composite:
\[ \xymatrix{
S^W\ar[r]^-\Delta & S^W\Smash S^W \ar[r]^-{\iota\Smash \iota}  & K_W\Smash K_W\ar[r]^-\mu & K_{2W}.
}
\]
The key observation is that $\Delta$ is homotopic to the map
\[ \xymatrix{
S^W \ar@{=}[r] & S^W\Smash S^0 \ar[r]^-{\id\Smash a_W} & S^W\Smash S^W
}\]
and also to the map $a_W\Smash \id$.  Since the composite $\iota\circ a_W$ is the cohomology class $a_W\in \M$, the claim about $\Delta$ immediately yields $\iota^2=a_W \iota=\iota a_W$.  

To see the homotopy, use $D(W)/S(W)$ as our model for $S^W$.  The homotopy is $H\colon D(W)/S(W)\times I \ra D(W)/S(W)\Smash D(W)/S(W)$
given by $H(x,t)=(x,tx)$ (or, in the second case, by $H(x,t)=(tx,x)$).
\end{proof}

We will also need to know how the antipodal map $A\colon S^W\to S^W$ acts on cohomology.  This is another instance where the Euler class plays a role.  Note that $A$ is not a pointed map, and so does not preserve the  reduced cohomology.  Instead the following is true:

\begin{prop} 
\label{pr:antipodal}
For $A\colon S^W\to S^W$ the antipodal map we have the formula
$A^*(\iota_W)=(-1)^{|W|+1}\iota_W+a_W$. 
\end{prop}

\begin{proof}
Let $j_0,j_\infty\colon \pt \ra S^W$ be the maps sending the basepoint to $0$ or $\infty$, respectively.  Then $\iota_W\circ j_\infty$ is the inclusion of the basepoint, so $j_\infty^*(\iota_W)=0$.  On the other hand, $j_0^*(\iota_W)=a_W$.  The proof will use these formulas together with $A\circ j_0=j_\infty$ and $A\circ j_\infty=j_0$.

We know that $H^W(S^W)=\M^W\oplus \tH^W(S^W)$, with the first summand generated by $a_W$ and the second by $\iota_W$.  Since $A^*(\iota_W)\in H^W(S^W)$ we can therefore write $A^*(\iota_W)=ka_W+m\iota_W$ for some $k,m\in \Z$ (really $k$ is only defined modulo $d(W)$).  
Applying the forgetful map $\psi$ into nonequivariant singular cohomology, we have $\psi(\iota_W)=\iota$, which is a generator of $H^{|W|}_{sing}(S^{|W|})$, and $\psi(a_W)=0$.  This shows that $m=(-1)^{|W|+1}$, the usual nonequivariant degree of the antipodal map on $S^{|W|}$.  Next consider the equation
\[ 0 = j_\infty^*(\iota_W)=j_0^*A^*(\iota_W)=kj_0^*(a_W)+mj_0^*(\iota_W)=(k+m)a_W.\]
So $ka_W=-ma_W$ and we obtain the formula
\begin{equation} 
\label{eq:A-eq}
A^*(\iota_W)=-ma_W+m\iota_W=(-1)^{|W|+1}\iota_W+(-1)^{|W|}a_W
\end{equation}
(we have switched the order of the terms in the last equality). This is almost what we want--we just need to show $(-1)^{|W|}a_W=a_W$.
To do so, we can also consider the formula
\begin{equation}
\label{eq:j-a-formula}
a_W=j_0^*(\iota_W)=j_\infty^*(A^*\iota_W)=(-1)^{|W|+1}j_\infty^*(\iota_W)+(-1)^{|W|} j_\infty^*(a_W) = (-1)^{|W|}a_W.
\end{equation}
So we may rewrite (\ref{eq:A-eq}) as $A^*(\iota_W)=(-1)^{|W|+1}\iota_W+a_W$.
\end{proof}

\begin{remark}
\label{re:a_odd}
Notice that (\ref{eq:j-a-formula}) shows that $1-(-1)^{|W|}$ annihilates $a_W$.  This proves Proposition~\ref{pr:a_odd}.
\end{remark}

One can also consider the cohomology of the spheres $S(W)$.  If $W\supseteq 1$ then $S(W)\iso S^{W-1}$, but otherwise $S(W)$ has no fixed point and so cannot be the 1-point compactification of any representation.  The inclusion of $S(W)$ into $D(W)$ gives a cofiber sequence of pointed spaces $S(W)_+ \ra D(W)_+ \ra S^W$, but since $D(W)$ is contractible we can also write it in the form
\begin{equation}
\label{eq:a-cofiber}
S(W)_+ \lra S^0 \llra{a_W} S^W.
\end{equation}
The resulting long exact sequence in reduced cohomology takes the form
\[ \xymatrixcolsep{1.5pc}\xymatrix{ 
\cdots & \M^{\alpha+1}\ar[l] & \M^{\alpha+1-W}\ar[l]_{a_W} & H^\alpha(S(W)) \ar[l] & \M^\alpha \ar[l] & \M^{\alpha-W} \ar[l]_{a_W} & \cdots \ar[l] 
}\]
So the groups $H^\star(S(W))$ are extensions of $\ker a_W$ and $\coker a_W$.  Determining these groups seems to be difficult in general, but the following is one useful fact:

\begin{prop}
If $V\supseteq W$ then $H^V(S(W))=0$.    
\end{prop}

\begin{proof}
The piece of the relevant long exact sequence  is
    \[ \xymatrixcolsep{1.5pc}\xymatrix{ 
\cdots & \M^{V+1}\ar[l] & \M^{V+1-W}\ar[l]_{a_W} & H^V(S(W)) \ar[l] & \M^V \ar[l] & \M^{V-W} \ar[l]_{a_W} & \cdots \ar[l] 
}\]
But $\M^V$ is generated by $a_V$, and $a_V=a_W\cdot a_{V-W}$.  So the map $\M^{V-W}\ra \M^V$ is surjective.  On the other side, since $V+1-W=(V-W)+1$ is a representation containing $1$ we know from Proposition~\ref{pr:a-group} that $\M^{V-W+1}=0$.  So $H^V(S(W))=0$ by the long exact sequence.
\end{proof}

We will be particularly interested in the group $H^{W-1}(S(W))$, and for this we have
\[
\xymatrix{
\M^W  & \M^0 \ar[l]_{a_W}  & H^{W-1}(S(W)) \ar[l] & \M^{W-1} \ar[l] & \M^{-1}\ar[l] \ar@{=}[d] \\
\Z/d(W)\ar[u]^\iso & \Z\ar@{->>}[l] \ar[u]^\iso &&&0.
}
\]
So we get a short exact sequence $0\lla \Z \lla H^{W-1}(S(W)) \lla \M^{W-1} \lla 0$, which is split but not naturally.

If $W\supseteq 1$ then $\M^W=0$, and we define a \dfn{fundamental class} for $S(W)$ to be an element of $H^{W-1}(S(W))$ that maps to a generator in $\M^0$.  If $W\supseteq 2$ then $\M^{W-1}=0$ and there are only two such classes, but in general there will be $2\cdot d(W-1)$ of them.
Note that applying $\psi$ to a fundamental class gives a fundamental class in singular cohomology.  

In the case $W\not\supseteq 1$ then the kernel of $\M^0\ra \M^W$ is still a copy of $\Z$, but it is generated by $d(W)\cdot 1$.  We define a \dfn{semi-fundamental class} to be an element of $H^{W-1}(S(W))$ that maps to a generator of this kernel.  Applying $\psi$ to a semi-fundamental class gives $\pm d(W)\cdot 1$ in singular cohomology.

\section{Cohomology relations on configuration spaces}
\label{se:gen-rels}

In this section we return to the configuration spaces $\OC_k(V)$, where $V$ is an orthogonal $G$-representation.  Note that $\OC_k(V)^{H}\iso \OC_k(V^H)$ for all subgroups
$H\subseteq G$.  A configuration of points $(x_1,\ldots,x_k)$ will
often be denoted $\und{x}$ for short.  

For each $i,j\leq k$ recall the map $\tilde{\omega}_{ij}\colon \OC_k(V)\ra S(V)$ given by
\[ \tilde{\omega}_{ij}(\und{x})=\dfrac{x_i-x_j}{||x_i-x_j||}.
\]
Observe that $\tilde{\omega}_{ij}$ and $\tilde{\omega}_{ji}$ differ by post-composition with the antipodal map on $S(V)$.  

For the rest of this section we assume $V\supseteq 1$.  In this case $S(V)$ is isomorphic to $S^{V-1}$.  So $\tH^{V-1}(S(V))\iso \Z$.  We fix once and for all a choice of generator $\iota_{V-1}\in \tH^{V-1}(S(V))$ (we will sometimes drop the subscript).  Define
$\omega_{ij}=\tilde{\omega}_{ij}^*(\iota_{V-1})$. We prove the following properties of the classes $\omega_{ij}$:

\begin{prop}
\label{pr:relations-1}
$\omega_{ij}^2=a_{V-1}\omega_{ij}=\omega_{ij}a_{V-1}$ and $\omega_{ji}=(-1)^{|V|}\omega_{ij}+a_{V-1}$.
\end{prop}

\begin{proof}
The first relation follows at once from Proposition~\ref{pr:H-S^W}, and then the second from Corollary~\ref{co:a-commute}.   For the third we use the formula $\tilde{\omega}_{ji}=A\circ\tilde{\omega}_{ij}$ where $A$ is the antipodal map on $S(V)$.  Then
\[ \omega_{ji}=\tilde{\omega}_{ji}^*(\iota)=\tilde{\omega}_{ij}^*(A^*\iota)= \tilde{\omega}_{ij}^*((-1)^{|V|}\iota+a_{V-1})=(-1)^{|V|}\omega_{ij}+a_{V-1}
\]
where in the third equality we used Proposition~\ref{pr:antipodal}.
\end{proof}

\begin{cor}
\label{co:omega*omega}
$\omega_{ij}\cdot \omega_{ji}=0$.
\end{cor}

\begin{proof}
We just compute
\begin{align*} \omega_{ij}\omega_{ji}=\omega_{ij}\bigl ( (-1)^{|V|}\omega_{ij}+a_{V-1}\bigr )&=(-1)^{|V|}\omega_{ij}^2+\omega_{ij}a_{V-1}\\
&=(-1)^{|V|}\omega_{ij}a_{V-1}+\omega_{ij}a_{V-1}\\
&=(1-(-1)^{|V|-1})\cdot \omega_{ij}a_{V-1}.
\end{align*}
Now use that $1-(-1)^{|V|-1}$ annihilates $a_{V-1}$, by Proposition~\ref{pr:a_odd}.
\end{proof}

There is an alternative proof of the above identity that is more geometric---see Remark~\ref{re:secondproof} for details.

We next work towards establishing the Arnold relation.  The following seems to be the most ``geometric'' form of the relation:

\begin{prop}
\label{pr:arnold-1}
For any triple of indices $i$, $j$, $k$ the relation
\[ (\omega_{ij}-\omega_{ik}) \cdot (\omega_{ji}-\omega_{jk})=0
\]
holds in $H^\star(\OC_k(V);\mZ)$.
\end{prop}

Before giving the proof let us observe that the relation can be rewritten as follows: 

\begin{cor}
\label{co:arnold-2}
For any triple $i$, $j$, $k$ the relation
\[ \omega_{ij}\omega_{jk}+\omega_{jk}\omega_{ki}+\omega_{ki}\omega_{ij}=a_{V-1} (\omega_{ij}+\omega_{jk}+\omega_{ki})-a_{V-1}^2
\]
holds in $H^\star(\OC_k(V);\mZ)$.
\end{cor}

\begin{proof}
Multiply out the relation from Proposition~\ref{pr:arnold-1} using distributivity and use that $\omega_{ij}\omega_{ji}=0$ from Corollary~\ref{co:omega*omega}.  This yields
\begin{align*}
    0&=-\omega_{ij}\omega_{jk}-\omega_{ik}\omega_{ji}+\omega_{ik}\omega_{jk} \\
    &= -\omega_{ij}\omega_{jk}-((-1)^{|V|}\omega_{ki}+a_{V-1})((-1)^{|V|}\omega_{ij}+a_{V-1}) +(-1)^{|V|-1}\omega_{jk}\omega_{ik}\\
    &= -\omega_{ij}\omega_{jk}-\omega_{ki}\omega_{ij}-a_{V-1}((-1)^{|V|}\omega_{ij}+(-1)^{|V|}\omega_{ki}+a_{V-1})+(-1)^{|V|-1}\omega_{jk}\omega_{ik}\\
    &= -\omega_{ij}\omega_{jk}-\omega_{ki}\omega_{ij}+a_{V-1}(\omega_{ij}+\omega_{ki}-a_{V-1})+(-1)^{|V|-1}\omega_{jk}\omega_{ik}.
\end{align*}
We have used Proposition~\ref{pr:relations-1} in the second equality, and in the fourth we have used that $a_{V-1}=(-1)^{|V|-1}a_{V-1}$ by Proposition~\ref{pr:a_odd}.  We are also using in several places that $a_{V-1}$ is central (Corollary~\ref{co:a-commute}).
Next we use Proposition~\ref{pr:relations-1} again to replace $\omega_{ik}$ with $(-1)^{|V|}\omega_{ki}+a_{V-1}$, and this gives us
\begin{align*}
0=
-\omega_{ij}\omega_{jk}-\omega_{ki}\omega_{ij}-\omega_{jk}\omega_{ki}+a_{V-1}(\omega_{ij}+\omega_{ki}+\omega_{jk}-a_{V-1}).
\end{align*}

\end{proof}

\begin{remark}
The natural $\Sigma_3$-action on $\OC_3(V)$ induces an action on the cohomology ring $H^\star(\OC_3(V);\mZ)$.
In the nonequivariant setting, the Arnold relation is readily checked to be the unique quadratic relation that is invariant under the $\Sigma_3$-action (up to multiplication by scalars, of course).  The Serre spectral sequence for the fibration $\OC_3(V)\ra \OC_2(V)$ (forget the last point in the configuration) shows there must be a nontrivial quadratic relation, and then equivariance says there is only one possibility.   This is a way to derive the Arnold relation without doing any geometry at all.  

However, in our equivariant setting the Arnold relation from Corollary~\ref{co:arnold-2} is {\it not\/} the unique quadratic relation that is equivariant: removing the constant term $a_{V-1}^2$ yields another relation with the same equivariance properties. 
\end{remark}

\begin{proof}[Proof of Proposition~\ref{pr:arnold-1}]
For $a,b\in V$ with $a\neq b$ write $u(a,b)$ for the unit vector pointing from $b$ towards $a$, i.e. $u(a,b)=\frac{a-b}{||a-b||}$.  For an ordered configuration $\und{x}=(x_1,\ldots,x_k)$ write $x_{ij}=u(x_i,x_j)$.  With this notation the cohomology class $\omega_{ij}$ may be regarded as the map $\OC_k(V)\ra AG(S(V))$ sending $\und{x}$ to $[x_{ij}]$.

For any indices $i$, $j$, $k$ there is a map $\OC_k(V)\ra \OC_3(V)$ that sends $\und{x}$ to the ordered triple $(x_i,x_j,x_k)$.  The class $\omega_{12}$ on $\OC_3(V)$ pulls back to $\omega_{ij}$ and so forth, so it suffices to prove the desired relation on $\OC_3(V)$ with $i=1$, $j=2$, and $k=3$.

Given an affine hyperplane $H$ in $V$, define a ``side'' of $H$ to be the union of $H$ with one of the two components of $V-H$.  So $H$ divides $V$ into two ``sides'', which are closed subsets whose intersection is $H$.
We will be concerned with two points being on the same side of $H$ or on opposite sides of $H$---but note that these notions are not negations of each other since one of our two points could be in $H$.

Given a configuration $\und{x}$, let $H_{rs}$ be the affine hyperplane in $V$ that passes through $x_r$ and is perpendicular to $x_r-x_s$.
Define $U_+\subseteq \OC_3(V)$ to be the subset of those configurations $\und{x}$ where $x_2$ and $x_3$ lie on the same side of $H_{12}$.  Likewise, define $U_-\subseteq \OC_3(V)$ to be the subset of those configurations $\und{x}$ where $x_2$ and $x_3$ lie on opposite sides of $H_{12}$.  Clearly $U_+\cup U_-=\OC_3(V)$.

We will prove that
\begin{itemize}
    \item 
$\omega_{12}-\omega_{13}$ is zero when restricted to  $U_+$, and therefore lifts to a class $\alpha\in H^\star(\OC_3(V),U_+)$
\item $\omega_{21}-\omega_{23}$ is zero when restricted to $U_-$ and therefore lifts to a class $\beta\in H^\star(\OC_3(V),U_-)$.
\end{itemize}
From these it follows that $(\omega_{12}-\omega_{13})(\omega_{21}-\omega_{23})$ lifts to the product $\alpha\beta\in H^\star(\OC_3(V),U_+\cup U_-)$, and is therefore zero.

The class $\omega_{12}-\omega_{13}$ is the map $\OC_k(V)\ra AG(S(V))$ that sends a configuration $\und{x}$ to $[x_{12}]-[x_{13}]$.  But if we restrict to $U_+$ then the line segment connecting $x_2$ to $x_3$ does not pass through $x_1$, and so we have the homotopy $H\colon U_+\times I \ra AG(S(V))$ given by
\[ H(\und{x},t)=[x_{12}]-[u(x_1,tx_2+(1-t)x_3)].
\]
Then $H_0=\omega_{12}-\omega_{13}$ and $H_1$ is the zero class. 

The proof for $\omega_{21}-\omega_{23}$ restricted to $U_-$ is very similar, but this time we move $x_3$ to $x_1$ via the straight line and use that it does not pass through $x_2$.
\end{proof}

\begin{remark}
Say that ``$ABC$ forms a ray'' if $C$ is on $\ray{AB}$ and is outside the interval $\overline{AB}$.
The above argument can also be written with $U_+$ the set of configurations $\und{x}$  such that $x_2x_1x_3$ is not a ray, and $U_-$
the set of configurations $x$ such that $x_1x_2x_3$ is not a ray.  
\end{remark}

\begin{remark}
\label{re:secondproof}
The above style of argument also gives an amusing alternative proof of Corollary~\ref{co:omega*omega}.  Fix a trivial subrepresentation $1\subseteq V$ and let $H$ be the orthogonal complement.  Let $U_+$ and $U_-$ be the subspaces of configurations $\und{x}$ where $x_{ij}$ is on one of the two sides of $H$.  We leave the remaining details to the reader.
\end{remark}

\subsection{Review of Stirling numbers}
\label{se:Stirling}

The unsigned Stirling number of the first kind $c(a,b)$ can be described in any of the following
ways:
\begin{enumerate}[(1)]
\item The number of permutations of $a$ that can be expressed as a
product of $b$ disjoint cycles (a fixed point counts as a cycle here).
\item The number of ways of placing $a$ distinguishable points on $b$
indistinguishable circles, with no circle left empty.
\item The absolute value of the coefficient of $x^b$ in
$x(x-1)(x-2)\cdots(x-a+1)$.
\item The sum of all $(a-b)$-fold products of distinct elements from
the set $\{1,2,\ldots,a-1\}$.  
\item The number of ways of selecting $a-b$ elements in the upper triangular part of an $a\times
a$ matrix so that no two are in the same column.
\item The number of permutations of $a$ that can be written as a
product of $a-b$ transpositions and no fewer.  
\end{enumerate}
Note that $c(a,a)=1$ and $c(a,a-1)=\binom{a}{2}$, and that $\sum_b
c(a,b)=a!$.  
Moreover, the Stirling numbers satisfy the recursion relation
$c(a+1,b)=a c(a,b)+c(a,b-1)$ which can be regarded as a variant of Pascal's identity
for binomial coefficients.  

A standard result says that the Stirling numbers are connected to the cohomology of configuration
spaces:

\begin{prop}
\label{pr:OC-classical}
The cohomology groups $H^*_{sing}(\OC_k(\R^n))$ are nonzero only in
degrees that are multiples of $n-1$, and 
$H^{i(n-1)}_{sing}(\OC_k(\R^n))\iso \Z^{c(k,k-i)}$.  Consequently, the
total rank of $H^*_{sing}(\OC_k(\R^n))$ is $k!$. 
\end{prop}

\begin{proof}
The classical computation is that $H^*_{sing}(\OC_k(\R^n))$ is the
quotient of the exterior algebra $\Lambda_\Z(\omega_{ij})_{1\leq
i<j\leq k}$ modulo the Arnold relations.  The Arnold relations allow
us to rewrite any product $\omega_{ij}\omega_{kl}$ in which the
maximal index that appears is repeated---for example,
$\omega_{24}\omega_{34}$---as a sum of monomials that do not have that
property.  Consequently, an additive basis is given by monomials of
the form
\begin{equation}
\label{eq:basis}
 \omega_{1j_{11}}\cdots \omega_{1j_{1r_1}}\cdot \omega_{2j_{21}}\cdots
\omega_{2j_{r_2}} \cdots 
\end{equation}
with the following properties:
\begin{itemize}
\item $\omega_{ij}$ appears only when $i<j$, and appears at most once,
\item the set of indices $\{j_{a*}\}$ does not intersect the set of
indices $\{j_{b*}\}$, for all choices of $a$ and $b$.  
\end{itemize}
If we imagine a $k\times k$ matrix in which every row reads
$[1,2,3,\ldots,k]$, then the choice of $j_{a*}$ indices can be
obtained by circling entries of the $a$th row in the upper triangular
part of the matrix.  The second condition above is the condition that
no two circled entries are in the same column, so the number of such
monomials coincides with description (5) of the Stirling numbers.
\end{proof}

The equivariant analog of the above result is as follows:

\begin{prop}
\label{pr:OC-algebra}
Let $R$ be the quotient of the free skew-commutative $\M$-algebra generated by classes $\omega_{ij}$ for $1\leq i\neq j\leq k$ subject to the relations
\begin{itemize} 
\item $\omega_{ij}=(-1)^{|V|}\omega_{ji}+a_{V-1}$,
\item $\omega_{ij}^2=a_{V-1}\omega_{ij}$,
\item $\omega_{ij}\omega_{jk}+\omega_{jk}\omega_{ki}+\omega_{ki}\omega_{ij}=a_{V-1} (\omega_{ij}+\omega_{jk}+\omega_{ki})-a_{V-1}^2$.
\end{itemize}
Then as a left $\M$-module $R$ is free with a basis consisting of $c(k,k-i)$ generators in degree $i(V-1)$, for all $i$.
\end{prop}

\begin{proof}
The proof is essentially the same as the non-equivariant case.  The basis consists of the same monomials in the $\omega_{ij}$ described in the proof of Proposition~\ref{pr:OC-classical}.
\end{proof}

\section{The homology of configuration spaces}
\label{se:homology}

In this section we study the (co)homology of ordered configurations of $k$ points in a $G$-representation $V$. Our computation will proceed roughly as follows. Observe
\[\OC_k(V)=V^k-\bigcup_{1\leq i < j \leq k} D_{i,j}\]
where $D_{i,j}$ is the subspace of $V^k$ where the $i$th and $j$th component are equal. We can inductively remove these subspaces, noting at each stage the resulting space (after smashing with $H\undZ$) is weakly equivalent to a wedge of certain suspensions of the unit sphere $S(V)$. The inductive step will rely on the vanishing of particular maps from the unit sphere, as seen in Theorem~\ref{th:splitting} below.

\subsection{\mdfn{$V$}-arrangements and a motivic splitting}

In this section we fix a finite-dimensional $G$-representation $V$. We  prove a splitting theorem for certain complements in $V^k$ of unions of subrepresentations.   The configuration space $\OC_k(V)$ will be an example of such a space, but we have need for the extra generality. We start by introducing some necessary conditions on the subrepresentations that are removed.

\begin{defn}
Fix $k\geq 1$ and let $H_1,\ldots,H_m\subseteq V^k$ be a collection of sub-representations.  We say this collection is a \mdfn{$V$-arrangement} if for every $r\geq 1$ and every set of indices $1\leq i_1,\ldots,i_r\leq m$, there is an isomorphism of representations $H_{i_1}\cap \cdots \cap H_{i_r}\iso V^s$ for some $s\geq k-r$.  
\end{defn}

Note that in a $V$-arrangement each $H_i$ will be either all of $V^k$ or else isomorphic to $V^{k-1}$.  The collection of all $D_{i,j}=\{\und{x}\in V^k \,|\, x_i=x_j\}$ is an example of a $V$-arrangement.   The empty collection of subspaces ($m=0$) is also an example of a $V$-arrangement.   Any subcollection of a $V$-arrangement is another $V$-arrangement. We also have need of the following intersection property:

\begin{prop}
\label{pr:V-arrange}
If $H_1,\ldots,H_m\subseteq V^k$ is a $V$-arrangement, then the collection $H_1\cap H_m,\ldots,H_{m-1}\cap H_m \subseteq H_m$ is also a $V$-arrangement.
\end{prop}

\begin{proof}
There are two cases, depending on whether $H_m=V^k$ or $H_m\iso V^{k-1}$.  The first is trivial.  In the second case, an $r$-fold intersection of the $H_i\cap H_m$ is an $(r+1)$-fold intersection of the original $H_i$, and so is isomorphic to $V^s$ for some $s\geq k-(r+1)=(k-1)-r$.  So the necessary condition is satisfied.
\end{proof}

Now we come to the core result.  It involves an awkward vanishing condition which will be investigated in more detail in Section~\ref{se:vanishing-criterion}.  

\begin{thm}
\label{th:splitting}
Let $k\geq 1$, let $m\geq 1$, and let $H_1,\ldots,H_m\subseteq V^k$ be a $V$-arrangement. 
 Set $X=V^k-(H_1\cup \cdots \cup H_m)$.  
Suppose that for all $\ell\in \Z$ we have 
\[ \Bigl [S(V)_+,H\mZ\Smash \Sigma^{\ell(V-1)+V}\bigl( S(V)_+\bigr ) \Bigr]=0
\]
(note that this condition depends only on $V$).
Then $H\mZ\Smash X_+$ is equivalent (as an $H\mZ$-module) to a wedge of summands of the form $H\mZ\Smash \Sigma^{j(V-1)}\bigl(S(V)_+\bigr)$ for $j=0,1,\ldots,k-1$.  
\end{thm}

\begin{proof}
We proceed by induction on $m$. When $m=1$ there are two cases: $H_1=V^k$ and $H_1\iso V^{k-1}$.  In the former $X=\emptyset$ and $H\mZ\Smash X_+ = *$, which is the empty wedge of the given summands.  In the latter case we have $X=V^k-H_1\iso V^{k-1}\times (V-0)\he S(V)$, and so $H\mZ\Smash X_+\he H\mZ\Smash S(V)_+$.  This completes the base case.

For the inductive step we fix $m\geq 2$ and assume the desired splitting holds for all $V$-arrangements of size $m-1$ or smaller. Let
 $H_1,\ldots,H_m\subseteq V^k$ be a $V$-arrangement.  If $H_m=V^k$ then $X=\emptyset$ and the result is trivial, so we may assume $H_m\neq V^k$.
Set $Y=V^k-(H_1\cup \dots \cup H_{m-1})$ and observe that $H\undZ \Smash Y_+$ decomposes as desired by the induction hypothesis. Let 
\[Z=H_m-((H_1\cap H_m) \cup \dots \cup (H_{m-1} \cap H_m)).\]
We can also apply the induction hypothesis to obtain a similar splitting for $H\mZ\Smash Z_+$, by Proposition~\ref{pr:V-arrange}.  Note that the splitting for $H\mZ\Smash Z_+$ will only have summands $H\mZ\Smash \Sigma^{j(V-1)}\bigl (S(V)_+\bigr )$ for $j\leq k-2$, since $H_m
\iso V^{k-1}$.

Note that $Z$ is a closed submanifold of $Y$.
Let $N_YZ$ denote the normal bundle of $Z$ in $Y$.  This bundle is trivial: it is the pullback
of the normal bundle of $H_m$ in $V^k$ under the inclusion $Z\hookrightarrow H_m$, and $N_{V^k}(H_m)$ is trivial because $H_m$ is a linear subspace of $V^k$. The fibers are isomorphic to the orthogonal complement of $H_m$ in $V^k$, which (up to isomorphism) is $V$.

Using the Pontryagin-Thom collapse map outside a tubular neighborhood of $Z$, we have a homotopy cofiber sequence
\[
Y-Z \hookrightarrow Y \to \Th(N_Y Z)
\]
where the rightmost term is the Thom space.  Note that $Z$ is not compact and so the tubular neighborhood theorem does not immediately apply, but it nevertheless holds in the case of a linear subspace like this.  
Since $N_Y Z$ is trivial we have $\Th(N_YZ)\iso \Sigma^V Z_+$.  

Smashing our cofiber sequence with $H\mZ$ yields the 
corresponding homotopy cofiber sequence of $H\mZ$-modules
\[ H\mZ\Smash (Y-Z)_+ \lra H\mZ\Smash Y_+ \lra H\mZ\Smash \Sigma^V (Z_+).
\]
Our plan is to show the right map is null, which will give us the equivalence
\[ H\mZ\Smash (Y-Z)_+\he (H\mZ\Smash Y_+) \Wedge \Sigma^{-1}(H\mZ\Smash \Sigma^V (Z_+)).
\]
This will complete the proof because $Y-Z=X$ and we have splittings for the two terms on the right by the induction hypothesis.  

The map we need to analyze is
\begin{equation}\label{eq:nullmap}
\begin{tikzcd}[sep=2mm]
H\undZ \Smash Y_+ \arrow[r] \arrow[d, dash,"\sim" style={rotate=90, inner sep=.5mm, anchor=north}] &H\undZ \Smash \text{Th}(N_YZ) \arrow[d, dash,"\sim" style={rotate=90, inner sep=.5mm, anchor=north}]\\
\bigvee_{j}H\undZ\Smash \Sigma^{j(V-1)}\bigl(S(V)_+\bigr) & H\undZ \Smash \Sigma^{V}(Z_+) \arrow[d, dash,"\sim" style={rotate=90, inner sep=.5mm, anchor=north}] \\
 & \Sigma^V \left(\bigvee_{i}H\undZ\Smash \Sigma^{i(V-1)}\bigl (S(V)_+\bigr)\right)
\end{tikzcd}
\end{equation}
where $0\leq j\leq k-1$ and $0\leq i\leq k-2$.  Note that in the wedges there might be multiple summands corresponding to each value of $j$ or $i$.  This map is null if and only if each component---corresponding to choosing individual summands in the domain and codomain---is null.
But these components lie in the groups
\begin{align*}
    &[H\undZ \Smash S^{j(V-1)}\Smash S(V)_+, \Sigma^V H \undZ \Smash S^{i(V-1)}\Smash S(V)_+]_{H\undZ} \\
    &\cong [S^{j(V-1)}\Smash S(V)_+, \Sigma^V H \undZ \Smash S^{i(V-1)}\Smash S(V)_+]\\
    &\cong [S(V)_+, H \undZ \Smash S^{(i-j)(V-1)+V}\Smash S(V)_+]
\end{align*}
and these groups vanish by assumption.  This 
completes the proof.
\end{proof}

\subsection{The vanishing criterion}
\label{se:vanishing-criterion}
We now investigate the vanishing hypothesis that appeared in Theorem~\ref{th:splitting}.  We will see that this follows from a more concrete condition that only involves the Bredon cohomology of a point, and that this is satisfied in many cases of interest. 

\begin{defn}
\label{de:vanishing}
Let $V$ be an orthogonal $G$-representation. We say $V$ satisfies the {\bf{vanishing requirement}} for $G$ if for all $\ell\in \Z$
\begin{itemize}
    \item $H^{\ell V-\ell}(\pt)\llra{\cdot a_V} H^{(\ell+1)V-\ell}(pt)$ is surjective, and
    \item $H^{(\ell+1) V-\ell}(\pt)\llra{\cdot a_V} H^{(\ell+2)V-\ell}(\pt)$ is injective.
\end{itemize}
\end{defn}
The terminology ``vanishing requirement'' comes from the following:

\begin{prop}\label{pr:vanish1}
    A representation $V$ satisfies the vanishing requirement if and only if $H^{\ell(V-1)+V}(S(V))=0$ for all $\ell \in \Z$.
\end{prop}

\begin{proof}
This follows readily from the cofiber sequence $S(V)_+\to S^0\llra{a_V}S^V$.
\end{proof}

The vanishing requirement implies the hypothesis of Theorem~\ref{th:splitting}:

\begin{prop}\label{groupvanish1}
    Suppose $V$ satisfies the vanishing requirement. Then
    \[[S(V)_+, H\undZ \Smash \Sigma^{\ell (V-1)+V}\bigl(S(V)_+\bigr)]=0\]
for all $\ell\in \Z$.
\end{prop}
\begin{proof}
    We use the usual cofiber sequence $S(V)_+\longrightarrow S^0\overset{a_V}{\longrightarrow} S^V$ of (\ref{eq:a-cofiber}) and then smash with $H\undZ \Smash S^{\ell (V -1) +V} $ to get the cofiber sequence
    \[
    H\undZ \Smash \Sigma^{\ell (V -1) +V} \bigl(S(V)_+\bigr)\longrightarrow H\undZ \Smash S^{\ell (V -1) +V}\overset{a_V}{\longrightarrow} H\undZ \Smash S^{\ell (V -1) +2V}.
    \]
    Applying $[S(V)_+,-]$ then yields a long exact sequence
    \[
    H^{\ell (V -1) +2V-1}(S(V))\to \bigl [S(V)_+, H\undZ \Smash \Sigma^{\ell (V -1) +V} \bigl( S(V)_+\bigr) \bigr ]\longrightarrow H^{\ell (V -1) +V}(S(V)).
    \]
    Since $V$ satisfies the vanishing condition, the left and right groups are zero, and thus the middle group must be zero.
\end{proof}

Which representations satisfy the vanishing requirement?  Here is the first class of examples:

\begin{prop}
\label{pr:vanishing-case1}
If $V\supseteq 1$ and $V\neq 2$ then $V$ satisfies the vanishing requirement.
\end{prop}

\begin{proof}
Since $V\supseteq 1$ we have $a_V=0$, so the vanishing requirement is equivalent to the statement that $H^{(\ell+1)V-\ell}(\pt)=0$ for all $\ell\in \Z$.  If we set $W=V-1$ then this says $H^{1+(\ell+1)W}(\pt)=0$ for all $\ell\in \Z$.  When $\ell\geq -1$ this is by Proposition~\ref{pr:a-group}.  When $\ell\leq -2$ this is by Proposition~\ref{pr:negative-W}, using that $W\neq 1$.  
\end{proof}

Verifying the vanishing requirement for representations not containing $1$ seems to be more difficult,  as it requires computing significant portions of the ring $\M$. Such computations are sparse in the literature and have only been done for a few families of groups. In the examples that we do know, the vanishing requirement is always satisfied as long as $\dim(V)\geq 3$. We summarize this here:

\begin{prop}
\label{pr:VR-examples}
Let $p$ be a prime.  
Let $G$ be one of the groups $C_p$, $C_{p^2}$, or $\Sigma_3$.
Then any orthogonal $G$-representation $V$ with $\dim(V)\geq 3$ satisfies the vanishing requirement.
\end{prop}

\begin{proof}
In each setting the proof amounts to a case-by-case analysis of the detailed computations of $\M$ available in the literature.  This is somewhat lengthy, though the main aspect that is ``hard'' is organizing the known facts about $\M$. We go through the cases in detail in Appendix~\ref{se:verify}.
\end{proof}

\subsection{The case when \mdfn{$V\supseteq 1$}} 
We have already proven that such $V$ satisfy the vanishing requirement, so we obtain the following:

\begin{prop}\label{homologyfree} If $V\supsetneq 1$ then $H\undZ \Smash\OC_k(V)_+$ is weakly equivalent, as an $H\mZ$-module, to a wedge sum of copies of $H\undZ \Smash S^{j(V-1)}$ where $0\leq j \leq k$.
\end{prop}

\begin{proof}
Recall that $\OC_k(V)$ can be constructed inductively by removing the $\binom{k}{2}$  subspaces $D_{i,j}=\{\und{x}\in V^k \,|\, x_i=x_j\}$. The subspaces $D_{i,j}$ form a $V$-arrangement. Thus for $V\neq 2$ the result follows from Theorem~\ref{th:splitting} and Proposition~\ref{pr:vanishing-case1} after noting $S(V)\simeq S^{V-1}$ and $H\mZ\Smash S(V)_+\he H\mZ \Wedge (H\mZ\Smash S^{V-1})$.

The case $V=2$ must be handled separately, but follows trivially from the fact that the forgetful map $\Psi\colon H\mZ\MMod\ra H\Z\MMod$, when restricted to the thick subcategory generated by trivial suspensions of $H\mZ$, is a triangulated equivalence.  
\end{proof}

\begin{thm} 
\label{th:smash-descript}
Let $V$ be an orthogonal $G$-representation such that $V\supsetneq 1$. Then 
\[
H\undZ \Smash\OC_k(V)_+ \simeq \bigvee_{j=0}^k (H\undZ \Smash S^{j(V-1)})^{c(k,k-j)}
\]
as $H\mZ$-modules,
where $c(k,k-j)$ are the unsigned Stirling numbers of the first kind (see Section~\ref{se:Stirling}).
\end{thm}
\begin{proof} By Proposition~\ref{homologyfree} and Theorem~\ref{th:splitting} we know for some nonnegative integers $b_j$ that
\[H\undZ \Smash\OC_k(V)_+\simeq  \bigvee_{j=0}^k (H\undZ \Smash S^{j(V-1)})^{b_j}.\]
We just need to show $b_j=c(k,k-j)$. Let $n=\dim(V)$. The forgetful
functor
$\Psi\colon H\undZ\MMod \to H\Z\MMod$ preserves wedge sums and sends $H\mZ\Smash X$ to $H\Z\Smash X$ (where in the latter case we have forgotten the $G$-action on $X$),
so applying this to $H\undZ \Smash\OC_k(V)$ gives an underlying equivalence 
\[H\Z \Smash\OC_k(\R^n)_+\simeq  \bigvee_{j=0}^k (H\Z \Smash S^{j(n-1)})^{b_j}. \]
But we know the homology of $\OC_k(\R^n)$ from Proposition~\ref{pr:OC-classical}, and thus it must be that $b_j=c(k,k-j)$.
\end{proof}

\begin{cor}
\label{co:main-1} If $V\supsetneq 1$ then
$H^\star(\OC_k(V);\mZ)$ is a free $\M$-module generated by elements in degrees $j(V-1)$ for $0\leq j\leq k$, with $c(k,k-j)$ elements in degree $j(V-1)$.
\end{cor}

\begin{proof}
Immediate from Theorem~\ref{th:smash-descript}, using that $H^\star(X;\mZ)\iso [H\mZ\Smash X_+,\Sigma^\star H\mZ]_{H\mZ}$.
\end{proof}

\subsection{The case when \mdfn{$V\not\supseteq 1$}}
For this case we only have results when $V$ satisfies the vanishing requirement of Definition~\ref{de:vanishing}.  When that is satisfied,
Theorem~\ref{th:splitting} gives us an additive splitting 
for the cohomology of $\OC_k(V)$ in terms of the cohomology of suspensions of $S(V)_+$.  To calculate the multiplicity of the different pieces we can again appeal to the forgetful functor to non-equivariant topology.  The extra basepoint in $S(V)_+$ leads to a slightly different form of results compared to the $V\supseteq 1$ case. In particular, we get an alternating sum of Stirling numbers in the answer.

\begin{defn}\label{aseq} Define $a(k,j)$ to be the following alternating sum of Stirling numbers of the first kind:
\begin{align*}
a(k,j)&= c(k,k-j)-c(k,k-(j-1))+c(k,k-(j-2))-...+(-1)^{j}c(k,k)\\
&=\sum_{i=0}^{j}(-1)^ic(k,k-(j-i)).
\end{align*}
\end{defn}

\begin{thm}
\label{th:main-2}
Let $G$ be a finite group and suppose $V$ is a $G$-representation that satisfies the vanishing requirement. Then
\begin{align*}
H\undZ \Smash \OC_k(V)_+& \simeq \bigvee_{j=0}^{k-1} \Bigl (H\undZ \Smash \Sigma^{j(V-1)}\bigl(S(V)_+\bigr)\Bigr )^{a(k,j)}\end{align*}
where $a(k,j)$ is the alternating sum defined in Definition~\ref{aseq}.
\end{thm}
\begin{proof} The ordered configuration space $\OC_k(V)$ is the complement in $V^k$ of the $V$-arrangement $D_{i,j}=\{\und{x}\in V^k \,|\, x_i=x_j\}$. Thus, as long as $V$ satisfies the vanishing requirement, Theorem~\ref{th:splitting} gives that
\[
H\undZ \Smash \OC_k(V)_+ \simeq \bigvee_{j=0}^{k-1} \Bigl (H\undZ \Smash \Sigma^{j(V-1)}\bigl( S(V)_+\bigr) \Bigr )^{b_j}
\]
for some $b_j$. We just need to show $b_j=a(k,j)$.

Let $n=\dim(V)$ and apply the forgetful functor $\Psi\colon H\undZ\MMod
\to H\Z\MMod$ to get an underlying equivalence 
\begin{align*}
H\Z \Smash\OC_k(\R^n)_+
&\simeq \bigvee_{j=0}^{k-1} \Bigl (H\Z \Smash \Sigma^{j(n-1)} \bigl( S^{n-1}_+\bigr) \Bigr)^{b_j}\\
&\simeq \bigvee_{j=0}^{k-1} (H\Z \Smash (S^{j(n-1)}\vee S^{(j+1)(n-1)}))^{b_j}\\
&\simeq H\Z^{b_0} \vee \bigvee_{j=1}^{k-1} (H\Z \Smash S^{j(n-1)})^{b_j+b_{j-1}}.
\end{align*}
Using the classical computation given in Proposition~\ref{pr:OC-classical} we see $b_0=c(k,k)$ and $b_{j-1}+b_j=c(k,k-j)$ for $j>0$.
We can then inductively solve to get \[b_j=c(k,k-j)-c(k,k-(j-1))+c(k,k-(j-2))+\dots\]
which is exactly the sequence $a(k,j)$ defined above. 
\end{proof}

\begin{cor}
\label{co:main-2}
In the setting of Theorem~\ref{th:main-2} there is a decomposition of $\M$-modules
\[ H^\star(\OC_k(V);\mZ)\iso \bigoplus_{j=0}^{k-1} \Sigma^{j(V-1)}H^\star(S(V);\mZ)^{\oplus a(k,j)}.
\]
\end{cor}

\begin{proof}
Immediate.
\end{proof}

\subsection{The case of $G=C_2$}
We demonstrate our results with a complete discussion of the spaces $\OC_k(V)$ for all nontrivial $C_2$-representations $V$.  
Let $\sigma$ denote the sign representation on $\R^1$, and note all finite-dimensional orthogonal representations satisfy $V\cong p1\oplus q\sigma$ for some $p,q\geq 0$.  

For the case $p\geq 1$ Corollary~\ref{co:main-1} gives that $H^\star(\OC_k(V);\mZ)$ is a free $\M$-module and provides a precise description of the degrees and multiplicities of the generators.
So we instead focus on $p=0$, where $V=q\sigma$.

If $q\geq 3$ then $V$ satisfies the vanishing criterion by Proposition~\ref{pr:VR-examples}, and so Corollary~\ref{co:main-2} applies.  For this to be useful we need to know $H^\star(S(V);\mZ)$, but this is well-known.  If we set
\[ \BB_q=\M[u_\sigma^{-1}]/(a_\sigma^q)=\Z[u_\sigma^{\pm 1},a_\sigma]/(2a_\sigma,a_\sigma^q)
\]
then we have
\[ H^\star(S(q\sigma);\mZ)=\BB_q[\iota]/(a_\sigma \iota,\iota^2)
\]
where $\iota$ is a semi-fundamental class of degree $q\sigma-1$.  
Figure~\ref{fig:S(2sigma)} shows pictures for $q=2$ and $q=3$. We have used motivic indexing for the picture, so that the index $p+q\sigma$ is drawn in spot $(p+q,q)$ on the grid; this explains the labelling of the axes.  Squares denote copies of $\Z$ and dots denote copies of $\Z/2$.  Diagonal lines represent multiplication by $a_\sigma$, e.g. the line connecting $u$ to $a_\sigma u$.
Note that the forgetful map satisfies $\psi(u)=1$, $\psi(a_\sigma)=0$, and $\psi(\iota)=2$.

\begin{figure}[ht]
\begin{tikzpicture}[scale=0.6]
\draw[step=1cm,gray,very thin] (-4.5,-2.5) grid (4.5,6.5);
\draw[] (-4.5,2) -- (4.5,2) node[below, black] {\small $1$};
\draw[] (0,-2.5) -- (0,6.5) node[left, black] {\small $\sigma\!-\!1$};

\foreach \y in {-1,...,3}
      \draw[] (0.95,2*\y+0.5) node[left,black] { $\square$};
      
\foreach \y in {-1,...,3}
      \draw[] (1.95,2*\y+0.5) node[left,black] { $\square$};

\foreach \y in {-1,...,2}
      \fill (1.55,2*\y+1.5) circle (3pt);

\foreach \y in {-1,...,2}
      \draw[thick] (0.65,2*\y+0.65) -- (1.5,2*\y+1.5);

\draw[] (0.5,4.25) node[left,black] {\small $u$};
\draw[] (0.5,2.25) node[left,black] {\small $1$};
\draw[] (2.45,3.35) node[left,black] {\small $a_\sigma$};
\draw[] (0.5,0.25) node[left,black] {\small $u^{-1}$};
\draw[] (2.35,4.45) node[left,black] {\small $\iota$};
\draw[] (2.45,6.45) node[left,black] {\small $u\iota$};
\draw[] (3.15,2.45) node[left,black] {\small {$u^{-1}\iota$}};

\end{tikzpicture}
\ \ \qquad
\begin{tikzpicture}[scale=0.6]
\draw[step=1cm,gray,very thin] (-4.5,-2.5) grid (4.5,6.5);
\draw[] (-4.5,2) -- (4.5,2) node[below, black] {\small $1$};
\draw[] (0,-2.5) -- (0,6.5) node[left, black] {\small $\sigma\!-\!1$};

\foreach \y in {-1,...,3}
      \draw[] (0.95,2*\y+0.5) node[left,black] { $\square$};
      
\foreach \y in {-1,...,2}
      \draw[] (2.95,2*\y+1.5) node[left,black] { $\square$};

\foreach \y in {-1,...,2}
      \fill (1.55,2*\y+1.5) circle (3pt);

\foreach \y in {-2,...,2}
      \fill (2.55,2*\y+2.5) circle (3pt);

\foreach \y in {-1,...,2}
      \draw[thick] (0.65,2*\y+0.65) -- (1.5,2*\y+1.5);

\foreach \y in {-1,...,2}
      \draw[thick] (1.65,2*\y+1.65) -- (2.5,2*\y+2.5);

\draw[thick] (1.65,-2.35) -- (2.5,-1.5);

\draw[] (0.5,4.25) node[left,black] {\small $u$};
\draw[] (0.5,2.25) node[left,black] {\small $1$};
\draw[] (2.45,3.35) node[left,black] {\small $a_\sigma$};
\draw[] (3.45,4.35) node[left,black] {\small $a_\sigma^2$};
\draw[] (0.5,0.25) node[left,black] {\small $u^{-1}$};
\draw[] (3.35,5.45) node[left,black] {\small $\iota$};
\draw[] (4.15,3.45) node[left,black] {\small {$u^{-1}\iota$}};

\end{tikzpicture}
\caption{The rings $H^\star(S(2\sigma);\mZ)$ and
$H^\star(S(3\sigma);\mZ)$
}
\label{fig:S(2sigma)}
\end{figure}

The representation $V=2\sigma$ does not satisfy the vanishing hypothesis, but we will see below that Corollary~\ref{co:main-2} still holds in this case.  So it is convenient to use this as an example.  The splitting of Corollary~\ref{co:main-2} says that additively we have
\[ H^\star(\OC_3(2\sigma);\mZ)\iso H^\star(S(2\sigma);\mZ)\oplus  \Sigma^{2\sigma-1} H^\star(S(2\sigma);\mZ) \oplus
\Sigma^{2\sigma-1} H^\star(S(2\sigma);\mZ).
\]
The picture for this is in Figure~\ref{fig:H-OC}.
The number $2$'s shown remind us that there are {\it two\/} copies in those spots (e.g. two red dots in the (2,1) box), and the coloring on the second two summands is just to help distinguish them from the first.  The ring structure for this example will be discussed in Section~\ref{se:ring-special-case} below.

\begin{figure}[ht]
\begin{tikzpicture}[scale=1]
\draw[step=1cm,gray,very thin] (-3.5,-0.5) grid (3.5,5.5);
\draw[thick] (-3.5,2) -- (3.5,2) node[below, black] {\small $1$};
\draw[thick] (-1,-0.5) -- (-1,5.5) node[left, black] {\small $\sigma\!-\!1$};

\foreach \y in {0,...,2}
      \draw[] (-0.2,2*\y+0.5) node[left,black] {$\square$};
      
\foreach \y in {0,...,2}
      \draw[] (0.8,2*\y+0.5) node[left,black] {$\square$};

\foreach \y in {-1,...,2}
      \fill (0.55,2*\y+1.5) circle (3pt);

\foreach \y in {0,...,2}
      \draw[thick] (-0.35,2*\y+0.65) -- (0.5,2*\y+1.5);

\draw[] (-0.5,4.25) node[left,black] {\small $u$};
\draw[] (-0.5,2.25) node[left,black] {\small $1$};
\draw[] (-0.1,0.15) node[left,black] {\small $u^{-1}$};
\draw[] (0.45,4.45) node[left,black] {\small $\iota$};

\foreach \y in {0,...,2}
      \draw[] (1.0,2*\y+0.2) node[left,red] {$\square$};
\foreach \y in {0,...,2}
      \draw[] (2.0,2*\y+0.2) node[left,red] {$\square$};
\foreach \y in {-1,...,2}
      \fill[red] (1.65,2*\y+1.3) circle (3pt);

\foreach \y in {0,...,2}
      \draw[thick,red] (0.85,2*\y+0.35) -- (1.6,2*\y+1.2);

\foreach \y in {0,...,2}
 \draw[] (0.9,0.2+2*\y) node[left,red] {\tiny $2$};

\foreach \y in {0,...,2} 
 \draw[] (1.9,0.2+2*\y) node[left,red] {\tiny $2$};

\foreach \y in {0,...,2} 
 \draw[] (2,1.2+2*\y) node[left,red] {\tiny $2$};
 
\end{tikzpicture}
\caption{The additive structure of $H^\star(\OC_3(2\sigma);\mZ)$}
\label{fig:H-OC}
\end{figure}

\noindent

\begin{prop}
\label{pr:V=2}
When $G=C_2$ and $V=2\sigma$ the conclusion of Theorem~\ref{th:main-2} still holds.  
\end{prop}

\begin{proof}
We only give a sketch.  Set $V=2\sigma$ and consider $\OC_k(V)$. Following the method of Theorem~\ref{th:splitting} we will prove the vanishing of all of the maps that come up.  These maps have components living in the groups
\[ X_j=[S(V)_+,H\mZ\Smash S^{j(V-1)+V}\Smash S(V)_+]
\]
for various values of $j\in \Z$, and we will consider the forgetful  map
\[ \psi\colon X_j\ra [S^1_+,H\Z\Smash S^{j+2}\Smash S^1_+]_e=X_j^e
\]
where the target $X_j^e$ is maps in the ordinary stable homotopy category.  
Using the cofiber sequence $S(V)_+\ra S^0\ra S^V$ in the codomain shows that $X_j$ sits in a long exact sequence
\[
\xymatrix{
\cdots \ar[r]^-{\cdot a_{2\sigma}} & H^{2V-1+j(V-1)}(S(V))\ar[r] & X_j \ar[r]& H^{V+j(V-1)}(S(V)) \ar[r]^-{\cdot a_{2\sigma}} \ar[r] & \cdots 
}
\]
The cohomology groups of $S(V)$ are as shown in Figure~\ref{fig:S(2sigma)} and one sees that the two groups sandwiching the $X_j$ are zero except in the cases $j=-1,-2,-3$.  In these cases one gets that $X_j$ is $\Z$, $\Z^2$, and $\Z$ (respectively), and in each case the forgetful map $\psi$ is injective: this is by a diagram chase using that the same is true for the relevant groups in $H^\star(S(V))$ sandwiching it, and because the maps labelled $a_{\sigma}$ and $a_{2\sigma}$ become zero under $\psi$.  So we can prove that our elements of $X_j$ are zero by proving that they are sent to zero under $\psi$.  

Now we do something slightly clever.  Let $W=1\oplus \sigma$.  Considering the same method for building up $\OC_k(W)$, we already know that all of the equivariant maps vanish by Propositions~\ref{groupvanish1} and \ref{pr:vanishing-case1}.  So applying $\psi$ to them yields the zero maps, but these are exactly the same as the  maps we needed to prove were zero.    
\end{proof}

\section{The product structure on cohomology}
\label{se:product}

In this section we study the multiplicative structure on the cohomology ring $H^\star(\OC_k(V);\mZ)$, following our work on the additive structure in the previous section.  In the case $V\supseteq 1$ we can give a presentation of the ring as an $\M$-algebra: this is Theorem~\ref{th:main1} from the introduction, which we prove here.  The proof uses the generators and relations we have produced in Section~\ref{se:gen-rels}, the additive results from Section~\ref{se:homology}, and the known calculations in the non-equivariant setting from \cite{CLM}.

The case $V\not\supseteq 1$ is intrinsically more difficult: we discuss some of the considerations in Section~\ref{se:ring-special-case} below.

\subsection{The case $V\supseteq 1$} 
We now come to one of the main results mentioned in the introduction:

\begin{thm}
\label{th:main1}
If $V\supseteq 1$ then $H^\star(\OC_k(V);\und{\Z})$ is the quotient of the free $\M$-algebra generated by classes $\omega_{ij}$ of degree $V-1$, $1\leq i\neq j\leq k$, subject to the following relations:
\begin{align*}
&\omega_{ij}=(-1)^{|V|}\omega_{ji}+a_{V-1}, \\
&\omega_{ij}^2=a_{V-1}\omega_{ij}, \\
&\omega_{ij}\omega_{jk}+\omega_{jk}\omega_{ki}+\omega_{ki}\omega_{ij}=a_{V-1}(\omega_{ij}+\omega_{jk}+\omega_{ki})-a_{V-1}^2. 
\end{align*}
\end{thm}

\begin{proof}
We start with the equivalence of $H\mZ$-modules
\[ H\mZ\Smash \OC_k(V)_+ \he \bigWedge_i \left(H\mZ \Smash S^{i(V-1)}\right)^{c(k,k-i)}
\]
from Theorem~\ref{th:smash-descript}.  From this we get that
\begin{align*}
H^\star(\OC_k(V);\mZ)=[\OC_k(V)_+,\Sigma^\star H\mZ]&=[H\mZ\Smash\OC_k(V)_+,\Sigma^\star H\mZ]_{H\mZ}\\
&\iso \bigoplus_i \tH^{\star}(S^{i(V-1)})^{\oplus c(k,k-i)}\\
&= \bigoplus_i \Sigma^{i(V-1)}\M^{\oplus c(k,k-i)}.
\end{align*}
Let $R$ be the free skew-commutative $\M$-algebra generated by the classes $\omega_{st}$ subject to the relations listed in the theorem statement. By Proposition~\ref{pr:OC-algebra}, $R$ is a free $\M$-module with the same number of basis elements in the same degrees as $H^\star(\OC_k(V);\mZ)$, and so these are isomorphic free $\M$-modules. Furthermore, the work in Section~\ref{se:gen-rels} gives us a map of $\M$-algebras $f\colon R\ra H^\star(\OC_k(V);\mZ)$.  We just need to show $f$ is an isomorphism.  The difficulty here is that we have two free bases in the picture: one for $R$ consisting of products of $\omega_{st}$ classes, and one for $H^\star(\OC_k(V);\mZ)$ coming from Theorem~\ref{th:smash-descript}, which is a basis we know almost nothing about.  The nontrivial part of the proof involves watching these two bases interact. 

It is enough to prove that $f$ maps a free basis in the domain to a free basis in the target.  
All basis elements are in degrees $j(V-1)$ for $j=0,1,\dots, k-1$, so we can focus on these degrees.  Let $T=\bigoplus_{j\in \Z} \M^{j(V-1)}$, regarded as a $\Z$-graded subring of $\M$.  
If $j<0$ then $\M^{j(V-1)}=0$ by Proposition~\ref{pr:negative-W}, and $T^0=\M^0=\Z$. If $j>0$ then $T^j=\M^{j(V-1)}\iso \Z/d(V-1)$ by Propositions~\ref{pr:a-group} and \ref{pr:euler-mult}.  In fact those results show that $T=\Z[a_{V-1}]/(d(V-1)\cdot a_{V-1})$.  Observe that the $\Z$-graded subrings $R^{*(V-1)}\subseteq R$ and $H^{*(V-1)}(\OC_k(V);\mZ)\subseteq H^\star(\OC_k(V);\mZ)$ are naturally graded $T$-modules, and $f$ is a map of $T$-modules.  

Using the $\M$-module decomposition of $H^{\star}(\OC_k(V);\mZ)$ given above, we have that
\begin{align*}
H^{j(V-1)}(\OC_k(V);\mZ)&\iso \left(\bigoplus_{i=0}^{k-1} \Sigma^{i(V-1)}\M^{\oplus c(k,k-i)}\right)^{j(V-1)}\\
&= \bigoplus_{i=0}^{k-1} \left(\M^{(j-i)(V-1)}\right)^{\oplus c(k,k-i)}\\
&= \bigoplus_{i=0}^{k-1} \left( T^{j-i} \right)^{\oplus c(k,k-i)}.
\end{align*}
Using our description of $T$ we can therefore rewrite the above isomorphism as
\begin{equation}
\label{eq:decomp}
H^{j(V-1)}(\OC_k(V);\mZ)\iso \Z^{c(k,k-j)} \oplus (\Z/d(V-1))^{\oplus (c(k,k-(j+1))+\dots+c(k,1))}.
\end{equation}

Write $f'$ for $f$ restricted to the degrees $*(V-1)$, regarded as a map of graded $T$-modules.  For convenience denote the domain and codomain of $f'$ by $M$ and $N$.  Note that $M$ and $N$ are both free and finitely-generated, with generators in the same degrees, and that $T$ is non-negatively graded.  
The proof from here goes as follows:
\begin{enumerate}[(1)]
\item We prove that $f'$ is an isomorphism if we quotient out by the torsion.
\item Quotienting by the torsion is the same as quotienting by the ideal $I\subseteq T$ of positive elements, so (1) says that $f'\colon M/IM\ra N/IN$ is an isomorphism.
\item Standard commutative algebra then implies that $f'$ is an isomorphism (e.g., prove this by inducting up the degrees).
\end{enumerate}

The only part that needs further justification is (1).  For this, consider the triangle of ring maps
\[ \xymatrix{
R \ar[r]^-f\ar[dr] &H^\star(\OC_k(V);\mZ) \ar[d]^\psi \\
& H^*_{sing}(\OC_k(V);\Z).
}
\]
The diagonal map takes the generators $\omega_{st}$ to the corresponding singular cohomology clasess. We know this map is surjective because we know the $\omega_{st}$ classes generate the singular cohomology. We similarly know $\psi$ is surjective.  Note that all of the torsion summands in (\ref{eq:decomp}) map to zero under $\psi$, since the codomain is torsion-free. In a specific degree $j(V-1)$, our triangle looks like
\begin{equation} 
\label{eq:triangle}
\xymatrix{
\Z^{c(k,k-j)}\oplus \text{torsion} \ar[r]^f\ar@{->>}[dr] & \Z^{c(k,k-j)}\oplus \text{torsion} \ar@{->>}[d]^\psi \\
& \Z^{c(k,k-j)}.
}
\end{equation}
After quotienting by the torsion subgroups the diagonal and vertical maps become isomorphisms, so $f$ becomes an isomorphism as well.
\end{proof}

\subsection{The case $V\not\supseteq 1$}
\label{se:ring-special-case}
The first difficulty we encounter here is that $H^\star(\OC_k(V);\mZ)$ is not a free $\M$-module.  It is a direct sum of shifted copies of $H^\star(S(V);\mZ)$, but we don't actually have a nice description of the cohomology of $S(V)$.  As one example, if $G=C_2$ and $V=\R^n$ with the antipodal action then $H^\star(S(V))$ is not a finitely-generated $\M$-module.  So describing these rings with generators and relations is in some sense the wrong thing to do.  

It is tempting to replace the ground ring $\M$ with $H^\star(S(V);\mZ)$ and make all descriptions relative to that.  This can be done (at least in some cases), but it is somewhat unnatural.  First of all we must choose a ``reference map'' $\OC_k(V)\ra S(V)$.  We have all of the $\omega_{ij}$ for $1\leq i\neq j\leq k$, but choosing one from this list feels arbitrary.  But okay, let us reluctantly choose the $\omega_{12}$ map as our fixed  reference.  

As an extended example let us now focus on $\OC_3(2\sigma)$, whose additive structure is given in Figure~\ref{fig:H-OC} from Section~\ref{se:homology}.  Note that we have not yet chosen specific generators for the red terms in $H^{2\sigma-1}$ (all cohomology groups in this discussion are of the space $\OC_3(2\sigma)$).  We will use the forgetful map $\psi\colon H^{2\sigma-1}\ra H_{sing}^1$ to help with this. 

Observe the Euler class $a_\sigma$ fits into a cofiber sequnce $S^0\overset{a_\sigma}{\longrightarrow} S^{\sigma}\longrightarrow C_{2+}\Smash S^1$. Smashing with $\OC_3(2\sigma)$ then yields the standard forgetful long exact sequence
\[ \xymatrixcolsep{1.5pc}\xymatrix{
\cdots \ar[r] &  H^0_{sing} \ar[r] & H^{\sigma-1} \ar@{=}[d]\ar[r]^{a_\sigma} & H^{2\sigma-1} \ar@{=}[d]\ar[r]^{\psi} & H^1_{sing} \ar[r]\ar@{=}[d] & H^\sigma \ar@{=}[d]\ar[r]& H^{2\sigma} \ar[r]\ar@{=}[d] & \cdots \\
&&0 & \Z^3 & \Z^3 & \Z/2 & \Z^2
}
\]
So $\psi$ is an inclusion where the image has index $2$.  A little work shows that 
\[ \im(\psi)=\langle 
\omega_{12}-\omega_{13},
\omega_{12}-\omega_{23},
2\omega_{12}\rangle.
\]
Note that these generators are not canonical; it would be better to describe $\im(\psi)$ as being spanned by all the differences of $\omega$-classes together with all doubles of $\omega$-classes. However, we have chosen generators that are convenient with respect to our original ``reference frame'' of $\omega_{12}$.   Since $\psi(\iota)=2\omega_{12}$, we can fix the two red generators $A,B\in H^{2\sigma-1}$ by requiring that they map to $\omega_{12}-\omega_{13}$ and $\omega_{12}-\omega_{23}$.

To understand products we need to look at $\psi\colon H^{4\sigma-2}\ra H^2$, and again access this via the forgetful long exact sequence:
\[ \xymatrix{
\cdots \ar[r] &  H^1_{sing} \ar[r] & H^{3\sigma-2} \ar@{=}[d]\ar[r]^{a_\sigma} & H^{4\sigma-2} \ar@{=}[d]\ar[r]^{\psi} & H^2_{sing} \ar[r]\ar@{=}[d] & H^{3\sigma-1} \ar@{=}[d]\ar[r]& \cdots \\
&&\Z/2 & \Z^2 & \Z^2 & (\Z/2)^2
}
\]
Since $\psi$ is injective we can use it to detect products.  Note that $\psi(A)$ and $\psi(B)$ are degree $1$ and therefore square to zero, so $A^2$ and $B^2$ must also vanish.  We also compute
\begin{align*}
\psi(AB)=\psi(A)\psi(B)=(\omega_{12}-\omega_{13})(\omega_{12}-\omega_{23})&=-\omega_{13}\omega_{12}-\omega_{12}\omega_{23}+\omega_{13}\omega_{23}\\
&= -\omega_{31}\omega_{12}-\omega_{12}\omega_{23}-\omega_{23}\omega_{31}\\
&=0.
\end{align*}
So $AB$ must be zero, and likewise for $BA$ by skew-commutativity.

Likewise, $\psi(\iota A)=\psi(\iota)\psi(A)=2\omega_{12}(\omega_{12}-\omega_{13})=-2\omega_{12}\omega_{13}$ and $\psi(\iota B)=-2\omega_{12}\omega_{23}$.  Since $H^2_{sing}$ is generated by $\omega_{12}\omega_{13}$ and $\omega_{12}\omega_{23}$, and the quotient by $\im(\psi)$ is $(\Z/2)^2$, this calculates that $\im(\psi)$ is generated by $2\omega_{12}\omega_{13}$ and $2\omega_{12}\omega_{23}$.  So $\iota A$ and $\iota B$ may be taken to be the generators in this degree.  

Putting everything together, we have established the following:

\[ H^\star(\OC_3(2\sigma);\mZ)\iso \BB_2[\iota,A,B]/(A^2,B^2,AB,a_\sigma \iota,\iota^2) 
\]
where as usual the polynomial ring is interpreted to be skew-commutative and not strictly commutative.    The forgetful map has $\psi(\iota)=2\omega_{12}$, $\psi(A)=\omega_{12}-\omega_{13}$, and $\psi(B)=\omega_{12}-\omega_{23}$.  

The above example generalizes to the following result:

\begin{prop}
Let $G=C_2$ and let $\BB_{r}=\M[u^{-1}]/(a_\sigma^{r})$.  For $n\geq 2$ let $T\subseteq H^*_{sing}(\OC_k(\R^n))$ be the subring generated by the elements $\omega_{ij}-\omega_{kl}$ and $2\omega_{ij}$ (for any $i,j,k,l$).  Then there is an isomorphism
\[ H^\star(\OC_k(n\sigma);\mZ)\iso \BB_{n}\tens T
\]
where the generators of $T$ listed above are regarded as lying in degree $n\sigma-1$ and the ring structure on the tensor product is the skew-commutative one.
\end{prop}

\section{Comparing configuration spaces}
\label{se:comparing}
The ring presented in Theorem~\ref{th:main1} can be regarded as a deformation of the classical singular cohomology of the configuration space $\OC_k(\R^n)$, with the symbol $a_{V-1}$ as the deformation parameter.  The paper \cite{VG} introduced and studied a closely related ring, namely the ring of functions $\Func(\Sigma_k,\Z)$ (actually this is just one example of the rings studied in \cite{VG}, which worked in the more general context of hyperplane arrangements).  Varchenko and Gelfand proved that $\Func(\Sigma_k,\Z)$  can be presented as the quotient of $\Z[e_{ij}\,|\, 1\leq i\neq j\leq k]$ by the following relations:
\begin{itemize}
    \item 
    $e_{ij}=1-e_{ji}$, 
    \item $e_{ij}^2=e_{ij}$,
    \item $e_{ij}e_{jk}+e_{jk}e_{ki}+e_{ki}e_{ij}=e_{ij}+e_{jk}+e_{ki}-1$.
\end{itemize}
The similarity between the above presentation and the one from Theorem~\ref{th:main1} is transparent, but why are these two rings related?  

Essentially the same question has been answered by 
Proudfoot and his collaborators in the papers \cite{P}, \cite{M}, and \cite{DPW} (working in the cases of $\Z/2$-, $\Q$-, and $\Z$-coefficients, respectively).  They explain the relation between three things: the Varchenko-Gelfand ring, the cohomology ring of $\OC_k(\R^n)$, and the Borel-equivariant cohomology of spaces $\OC_k(V)$ for certain choices of group $G$ (either $C_2$ or $S^1$) and representation $V$.  

In this section we will review some of this story and retell it in the context of Bredon cohomology, using the results of our computations.  On the one hand, the Bredon story does not contain any fundamentally new ideas or improve on the Borel story in any significant way (in fact in some ways our version is worse, as we will see).  On the other hand, the Bredon version has a few interesting aspects that demonstrate some features
of the theory.

The heart of the matter involves the inclusions $\OC_k(V)\inc \OC_k(V\oplus W)$, for various representations $V$ and $W$.  In (reduced) singular cohomology these maps are always zero, but in equivariant cohomology they can be nonzero.  The space $\OC_k(\R^1)$ (where $\R^1$ has the trivial action) is homotopy equivalent to $\Sigma_k$, and so $H^*_{sing}(\OC_k(\R^1))$ is just the ring $\Func(\Sigma_k,\Z)$ concentrated in degree $0$.  We will see (following Proudfoot) that studying the inclusions $\OC_k(\R^1)\inc\OC_k(\R^1\oplus W)$ for certain $W$ leads to the desired connection between Theorem~\ref{th:main1} and the Varchenko-Gelfand ring.  

\subsection{The general comparison theorem for representations}
Fix representations $V$ and $W$ where $V\supseteq 1$.
Consider the inclusion $u\colon \OC_k(V)\inc \OC_k(V\oplus W)$.  The cohomology of the codomain has generators $\omega_{ij}(V\oplus W)$ in degree $V+W-1$, whereas for the domain the generators $\omega_{ij}(V)$ lie in degree $V-1$.  Here is the fundamental computation:

\begin{prop}
\label{pr:compare}
    $u^*(\omega_{ij}(V\oplus W))=\omega_{ij}(V)\cdot a_W$.  Consequently, if $W\supseteq 1$ then $u^*$ is the zero map on reduced cohomology.
\end{prop}

\begin{proof}
This is a consequence of the commutative diagram
\[ \xymatrixcolsep{4pc}\xymatrix{
\OC_k(V\oplus W) \ar[r]^-{\tilde{\omega}_{ij}(V\oplus W)} & S(V\oplus W) \\
\OC_k(V) \ar@{ >->}[u] \ar[r]^-{\tilde{\omega}_{ij}(V)} & S(V).\ar@{ >->}[u]
}
\]
Since $V\supseteq 1$ we have $S(V)\iso S^{V-1}$ and $S(V\oplus W)\iso S^{V-1+W}$.  So the result follows formally from Lemma~\ref{le:spheres} below.
\end{proof}

\begin{lemma}
\label{le:spheres}
Let $K$ and $L$ be representations where $K\supseteq 1$.  Then one has
\[ j^*(\iota_{K\oplus L})=\iota_K\cdot a_L
\] where $j\colon S^K\inc S^{K\oplus L}$ is the evident inclusion.
\end{lemma}

\begin{proof}
Use the commutative diagram
\[ \xymatrixcolsep{3.5pc}\xymatrix{
S^K\ar@{ >->}[r]^-j & S^{K\oplus L} \ar[r]^-{\iota_{K\oplus L}} & AG(S^{K\oplus L}) \\
S^K\Smash S^0 \ar@{=}[u]\ar[r]^-{\id\Smash a_L} & S^K\Smash S^L \ar[r]^-{\iota_K\Smash \iota_L} & AG(S^K)\Smash AG(S^L). \ar[u]_-\mu
}
\]
\end{proof}

When $W\not\supseteq 1$ the map $u^*$ from Proposition~\ref{pr:compare} has the potential to be nonzero and to contain useful information.  We will see some examples below.

\subsection{The Varchenko-Gelfand ring}
\label{se:V-G}
Let $H_{ij}\subseteq \R^k$ be the hyperplane defined by $x_i=x_j$.  Then $\OC_k(\R)=\R^k-\bigcup_{i,j} H_{ij}$, and this space is a union of contractible components which are in bijective correspondence with elements of $\Sigma_k$: the component of a configuration $\und{x}$ is indexed by the unique permutation $\sigma$ that puts the elements of $\und{x}$ in ascending order, in the sense that $x_{\sigma^{-1}(1)}<x_{\sigma^{-1}(2)}<\cdots$.  (The use of $\sigma^{-1}$ in the subscripts provides that $x_i<x_j$ implies $\sigma(i)<\sigma(j)$.)  
Varchenko-Gelfand \cite{VG} consider the ring of locally constant integer-valued functions on $\OC_k(\R)$, which by the above bijection is the same as the ring $\Func(\Sigma_k,\Z)$.  It is also $H^0(\OC_k(\R);\Z)$.  

For each $i\neq j$ \cite{VG} defines the \dfn{Heaviside function} $e_{ij}$ to have the value $1$ on all elements of $\R^k$ on the side of $H_{ij}$ where $x_i>x_j$, and to take the value $0$ on the other side of $H_{ij}$.  The relations
\[ e_{ij}^2=e_{ij}, \qquad e_{ij}=1-e_{ji}\]
are self-evident.   If we instead think of these as functions $\Sigma_k\ra \Z$, then for $\sigma\in \Sigma_k$ the above says that $e_{ij}(\sigma)$ is $1$ if $\sigma(i)>\sigma(j)$ and $0$ otherwise.  

Observe that $e_{ij}$ basically coincides with the cohomology class $\omega_{ij}\in H^0(\OC_k(\R))$ that we have previously defined, if one takes $-1\in S(\R)$ as the basepoint so that it represents $0$ in $AG(S(\R))$.  

The analog of the Arnold relation is proven as follows.  Given a permutation $\sigma$, if $\sigma(i)<\sigma(j)$ and $\sigma(j)<\sigma(k)$ then it must be the case that $\sigma(i)<\sigma(k)$.  This shows that the product $e_{ji}e_{kj}e_{ik}$ must be the zero function.  Rewrite this as
\begin{align*}
    0=(1-e_{ij})(1-e_{jk})(1-e_{ki})=1-e_{ij}- & e_{jk}-e_{ki}+e_{ij}e_{jk}+e_{jk}e_{ki}+e_{ki}e_{ij} \\
    &-e_{ij}e_{jk}e_{ki}
\end{align*}
and then use that the final cubic term must be zero by exactly the same reasoning as before (we are grateful to Nick Proudfoot for explaining this to us).

In \cite{VG} it is proven that $\Func(\Sigma_k,\Z)$ is the quotient of the polynomial ring $\Z[e_{ij}]$ by the three types of relations listed above. 
 Here we have produced the map $\Z[e_{ij}]/\!\!\sim \lra \Func(\Sigma_k,\Z)$. One way to justify that this is an isomorphism is 
via the following steps:
\begin{itemize}
\item Check by the same argument as in the proof of Proposition~\ref{pr:OC-algebra} that the domain is a free abelian group whose rank is $k!$, as is true for the codomain.
\item Prove surjectivity as follows: For $\sigma\in \Sigma_k$ let $\delta_\sigma\colon \Sigma_k\ra
\Z$ send $\sigma$ to $1$ and all other permutations to zero.  Then
the $\delta_\sigma$ functions are a $\Z$-basis for $\FF(\Sigma_k,\Z)$.    
For a given $\sigma$ let $P_\sigma$ be the product of $\binom{k}{2}$
factors, where for every $i<j$ we include the factor $e_{ij}$ if
 $\sigma(i)>\sigma(j)$, and the factor $1-e_{ij}$ if $\sigma(i)<\sigma(j)$.  As an example, for the permutation $\sigma=\begin{pmatrix} 1 & 2 & 3 \\ 2 & 3 & 1 \end{pmatrix}$
we
would have
\[P_\sigma=(1-e_{12})e_{13}e_{23}.
\]
Observe that $P_\sigma(\sigma)=1$.
If $\alpha$ is another permutation then $P_\sigma(\alpha)=0$ unless
the relative order of all pairs in $\alpha$ exactly matches the
relative order in $\sigma$---but this can only happen if $\alpha=\sigma$.  So
$P_\sigma=\delta_\sigma$, and this proves surjectivity.
\end{itemize}

Varchenko and Gelfand considered the increasing filtration $F_*$ of the ring $R=\Func(\Sigma_k,\Z)$ where $F_r$ is spanned by monomials in the $e_{ij}$ of degree less than or equal to $r$.  Thus we have
\[ \Z\langle 1\rangle =F_0 \subseteq F_1 \subseteq F_2 \subseteq \cdots \subseteq F_k=R\]
and $F_r\cdot F_s\subseteq F_{r+s}$.  Therefore $\gr R$ is a graded ring, and they observed that $\gr R\iso H^*_{sing}(\OC_k(\R^n))$ with $F_i/F_{i-1}\iso H^{i(n-1)}_{sing}(\OC_k(\R^n))$ (really they did this for $n=2$, though the observation is of course valid for higher $n$ as well).

For a filtered ring as above, the \dfn{Rees ring} $\Rees(R)$ is the subring $\bigoplus_i t^iF_i\subseteq R[t]$.  One has $\Rees(R)/(t)\iso \gr R$ and $\Rees(R)/(t-1)\iso R$.  We depict this in diagram form as
\[ \xymatrix{
& \Rees(R) \ar@{->>}[dl]_{t=1} \ar@{->>}[dr]^{t=0}\\
R && \gr(R).
}
\]
The relations among the $e_{ij}$ classes become the following homogeneous relations that hold in $\Rees(R)$:
\begin{align*}
    & e_{ij}=1-e_{ji} \rightsquigarrow [te_{ij}]=t-[te_{ji}]\\
    & e_{ij}^2=e_{ij} \rightsquigarrow [te_{ij}]^2=t\cdot [te_{ij}]\\
    & e_{ij}e_{jk}+e_{jk}e_{ki}+e_{ki}e_{ij}=e_{ij}+e_{jk}+e_{ki}-1\rightsquigarrow\\
&\qquad\qquad\qquad    [te_{ij}][te_{jk}]+[te_{jk}][te_{ki}]+[te_{ki}][te_{ij}]=t[te_{ij}]+t[te_{jk}]+t[te_{ki}]-t^2.    
\end{align*}
Note that we are writing $[te_{ij}]$ for the evident class in $\Rees(R)_1$ to make it clear that this is not a multiple of $t$ in the Rees ring.

The reader will note that these relations in the Rees ring are almost the same as those in Theorem~\ref{th:main1}---with $t$ replaced by $a_{V-1}$---though there are some sign differences that we will discuss in the next section.

\subsection{The denouement}
\label{se:denouement}
Let $W$ be a representation such that $G$ acts freely on $W-\{0\}$.  Examples to keep in mind are when $G=\Z/2$ and $W=\R$ with the sign action, and when $G=C_n$ and $W=\R^2$ with the generator of $G$ acting as counterclockwise rotation by $\frac{2\pi}{n}$ radians.  Note that if $W$ exists then $d(W)=\#G$, where $d(W)$ is the greatest common divisor defined in Proposition~\ref{pr:a-group}, and so by Proposition~\ref{pr:a_odd} if $\#G>2$ then $W$ must be even-dimensional.

We will consider the inclusion $u\colon \OC_k(\R)\inc \OC_k(\R\oplus W)$, where the $G$-action on $\R$ is the trivial one.  So $u$ is the inclusion of the fixed set.  
Consider the diagram
\[ \xymatrix{
& H^\star(\OC_k(\R\oplus W)) \ar[dr]^\psi\ar[dl]_{u^*} \\
H^\star(\OC_k(\R))  && H^*_{sing}(\OC_k(\R\oplus W)) \\
H^*_{sing}(\OC_k(\R))\tens \M \ar[u]_\iso \ar@{=}[r] & \Func(\Sigma_k,\Z)\tens \M.
}\]
The vertical isomorphism is from Proposition~\ref{pr:trivial-action}. 
The ring map $\psi$ sends the generators $\omega_{ij}(\R\oplus W)$ to the classical generators $\omega_{ij}$, and so $\psi$ is surjective.  The element $e_{ij}\tens 1\in \Func(\Sigma_k,\Z)\tens \M$ maps to $\omega_{ij}(\R)$ in $H^\star(\OC_k(\R))$ by inspection, and we know $u^*(\omega_{ij}(\R\oplus W))=a_W\cdot \omega_{ij}(\R)$ by Proposition~\ref{pr:compare}.

All of the interesting phenomena is in the degrees $nW$ for $n\in \Z$, so we restrict our attention to those.  In fact the groups are zero when $n<0$ (this uses Proposition~\ref{pr:negative-W}), so we focus on $n\geq 0$.  The following table shows the additive generators of $H^\star(\OC_k(\R\oplus W))$ in these degrees:

\vspace{0.1in}
\begin{center}
\begin{tabular}{|c|c|c|c|c|}
\hline
0 & 1 & 2 & 3 & $\cdots$\\
\hline\hline
$1$ & $\omega_{ij}$ & $\omega_{ij}\omega_{mn}$  & $\omega_{ij}\omega_{mn}\omega_{st}$ & $\cdots$\\
& $a_W$ & $a_W\omega_{ij}$ & $a_W \omega_{ij}\omega_{mn}$ &\\
&& $a_W^2$ & $a_W^2\omega_{ij}$ &\\
&&& $a_W^3$ & \\
\hline
\end{tabular}
\end{center}
\vspace{0.1in}
\noindent
Here the $\omega_{ij}$ classes and their products are all torsion-free, whereas all classes with $a_W$ are annihilated by $d(W)$.  If one formally substitutes $\omega_{ij}=a_We_{ij}$ then this starts to look like the Rees ring for $\Func(\Sigma_k,\Z)$ associated to the Varchenko-Gelfand filtration, but there are two differences.  One small but important difference is that the class $a_W$ is torsion.  The other difference is that when $W$ is 
odd-dimensional the relations in $H^{*W}(\OC_k(\R\oplus W))$ relating $\omega_{ij}$ to $\omega_{ji}$ have signs that differ from those in the Rees ring---however, remember that this case only occurs when $G=C_2$ and $d(W)=2$.

We can redraw our diagram as follows:

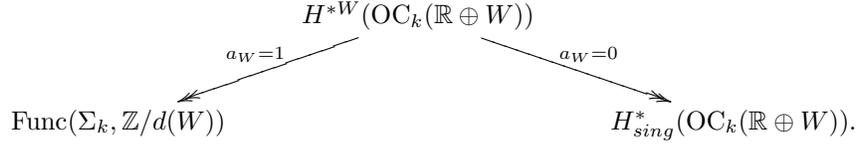
\[ \xymatrix{
& H^{*W}(\OC_k(\R\oplus W)) \ar@{->>}[dl]_{a_W=1}\ar@{->>}[dr]^{a_W=0} \\
\Func(\Sigma_k,\Z/d(W)) && H^*_{sing}(\OC_k(\R\oplus W)).
}
\]
The object on the bottom left is purely combinatorial, whereas the one on the bottom right is topological; we will use the diagram to pass information between them.  
Whether or not the object on top is isomorphic to the Rees ring  is immaterial to the rest of our discussion, but it plays a comparable role.  We can regard this object as a graded module  over $R=\Z[a_W]/(d(W)a_W)$, and as such it is a free module.  Let $b_i$ denote the number of generators of rank $i$, and pretend for the moment that we don't know these numbers.   

The two arrows are the result of applying $(\blank)\tens_R R/(a_W-1)$ and $(\blank)\tens_R R/(a_W)$.  In the latter case, $R/(a_W)$ is a graded $R$-module and so the target is also graded (as we know).  So the $b_i$ are the same as the ranks of the singular cohomology groups.    In the former case, $R/(a_W-1)$ is not graded and we instead only get a filtration defined by letting $F_i$ be the image of the $i$th graded piece.  By inspection this is precisely the Varchenko-Gelfand filtration, and it follows formally that $F_i/F_{i+1}$ is a free $\Z/d(W)$-module of rank $b_i$.  

As a consequence of the above  we deduce that
\begin{itemize}
\item The total rank of $H^*_{sing}(\OC_k(\R\oplus W))$ is the same as the rank of $\Func(\Sigma_k,\Z/d(W))$ as $\Z/d(W)$-module, which is manifestly equal to $k!$.
\item The rank of $H^{i\cdot |W|}_{sing}(\OC_k(\R\oplus W))$ is equal to the rank of $F_i/F_{i+1}$ as a $\Z/d(W)$-module, which is a purely combinatorial object.  From here it is ``only'' a matter of combinatorics to identify this with the appropriate Stirling number.  
\end{itemize}
This is not exactly an independent calculation of $H^*_{sing}(\OC_k(\R\oplus W))$ in that it presupposes knowing the freeness of the upper object in our diagram, but it does give a satisfying explanation of how the machinery of equivariant cohomology provides combinatorial interpretations of the cohomology of $\OC_k(\R\oplus W)$.  

Following Proudfoot and his collaborators, it is tempting to push these methods to incorporate the $\Sigma_k$-actions that exist everywhere.  Certainly $\Sigma_k$ acts on the three objects in our diagram, and the action is preserved by the maps.  It would be nice to conclude, for example, that the total cohomology ring of $\OC_k(\R\oplus W)$ is isomorphic---as a $\Sigma_k$-module---to $\Func(\Sigma_k,\Z)$.  Unfortunately, the fact that $a_W$ is torsion plays against us here.  

There are a couple of approaches to removing the torsion condition on $a_W$.  
One is to look at an inverse limit system where the torsion gets larger and larger: for example, $G=C_n$ and $W=\R^2$ with the rotation-by-$\frac{2\pi}{n}$ representation, and let $n\ra \infty$.  A related---and more satisfying---approach is to take $G=S^1$ with $W=\R^2$ the analogous ``rotation'' representation.  This is the approach taken in \cite{M} and \cite{DPW} for Borel cohomology; it can also be done in the Bredon setting, but it takes us outside the context of finite groups that has been the subject of the present paper.

\appendix
\section{A proof for odd-dimensional representations}
\label{se:app}
In this appendix we give an elementary proof (avoiding Bredon cohomology) of Proposition~\ref{pr:a_odd}.  That is, we show that if $V$ is an odd-dimensional representation of $G$ over $\R$ then the gcd of $\cD(V)=\{\#(G/H)\,|\, V^H\neq 0\}$ is either $1$ or $2$. This implies the Euler class satisfies either $a_V=0$ or $2a_V=0$. We use Smith Theory as one step of the analysis, so the proof is not entirely algebraic.

\begin{proof}[Proof of Proposition~\ref{pr:a_odd}]
First observe that $d(V)=\gcd \cD(V)$ is a divisor of $\#G$ by the definition. Next we prove $d(V)$ must be a power of $2$.

Let $p$ be an odd prime dividing $\#G$ and let $H\leq G$ be a Sylow $p$-subgroup.  Smith theory implies that $\chi(S^V)\equiv \chi(S^{V^H})$ mod $p$. But $\chi(S^V)=0$ since $\dim V$ is odd, so $p$ divides $\chi(S^{V^H})$.  In particular, $V^H\neq 0$ and therefore $\#(G/H)\in \cD(V)$. Note $p$ does not divide $\#(G/H)$, and so $p$ does not divide $\gcd \cD(V)$. We have proven this for all odd primes dividing $\#G$.  Thus, $d(V)$ is a power of $2$. \medskip

We now break into cases based on the parity of $\#G$:

\medskip
\noindent\emph{Case 1:} $\#G$ is odd.

\medskip
Since $d(V)$ is a power of $2$ and $d(V)\mid \# G$, we see $d(V)=1$.

\medskip
\noindent\emph{Case 2a:} $\#G$ is a power of $2$.

\medskip
We will prove the stronger result that either $1\in \cD(V)$ or $2\in \cD(V)$.

Since $\cD(V_1\oplus V_2)=\cD(V_1)\cup \cD(V_2)$, it suffices to prove the result when $V$ is irreducible.  In that case $\Hom_G(V,V)$ is a real division algebra and so is equal to one of $\R$, $\C$, or $\HH$; $V$ is accordingly classified as ``real'', ``complex'', or ``quaternionic''.  In the latter two cases $V$ must be even-dimensional over $\R$, so in fact our $V$ is ``real''.  The complexification $V_\C$ of an irreducible ``real'' representation is also irreducible; this follows from $\End_G(V_\C)\iso \End_G(V)_\C\iso \C$.    

Since $V_\C$ is irreducible we have $\dim_\C V_\C$ divides $\#G$.  But this says $\dim V$ divides $\#G$, so $\dim V$ is a power of $2$.  Since $\dim V$ is also odd, $\dim V=1$.  The representation of $G$ on $V$ is therefore specified by a map $G\ra \GL_1(\R)=\R^*$, or really a map $G\ra \Z/2$ since the image must lie in the torsion elements.  Either this map is trivial or else the kernel is index $2$, which says that either $1\in \cD(V)$ or $2\in \cD(V)$.  So $\gcd \cD(V)$ is $1$ or $2$ here.      

\medskip

\noindent
\emph{Case 2b:} $\#G$ is even.

\medskip

Let $P$ be a Sylow $2$-subgroup of $G$.  By Case 2a, either $V^P\neq 0$ or there is an index 2 subgroup $Q\subseteq P$ such that $V^Q\neq 0$.  In the first case we have that $\#(G/P)\in \cD(V)$, therefore $\cD(V)$ contains an odd number, and since we already showed the gcd is a power of $2$ it must be exactly $1$.  In the second case we have that $\#(G/Q)\in \cD(V)$, i.e. $\cD(V)$ contains twice an odd number.  Since the gcd is a power of $2$, it can only be $1$ or $2$.  This completes the proof.
\end{proof}

\section{Verifying the vanishing requirement}\label{se:verify}
In this section we discuss the matter of checking the vanishing requirement in some specific examples.  Recall that the requirement for $(G,V)$ is that in the sequence
\begin{equation}
\label{eq:VR}H^{\ell V-\ell}(\pt)\llra{a_V} H^{(\ell+1)V-\ell}(\pt)\llra{a_V}
H^{(\ell+2)V-\ell}(\pt)
\end{equation}
the first map is surjective and the second map is injective, for all $\ell\in \Z$.  By Proposition~\ref{pr:vanishing-case1} the requirement always holds if $V^G\neq 0$ and $V\neq 2$, so here we always focus on $V^G=0$.

When $\ell=0$ the conditions are satisfied because we know all three groups by Proposition~\ref{pr:a-group} and can check it directly.  When $\ell=-1$ the middle group is zero, and so the conditions are trivially satisfied.  When $\ell=-2$ the sequence is
\[ H^{2-2V}(\pt) \llra{a_V} H^{2-V}(\pt) \llra{a_V} H^2(\pt)=0.
\]
Since the last group is zero the conditions will be satisfied if and only if the middle group is zero.  But if $V$ is orientable and $\dim V=2$ we know $H^{2-V}(\pt)\neq 0$.  For this reason we will typically only be able to verify the conditions when $\dim V\neq 2$.  

\begin{remark}
\label{re:loc}
If $\alpha\in RO(G)$ we can consider the Mackey functor $\und{H}^\alpha(\pt)$, which is a $\mZ$-module.  So the composite
\[ H^\alpha(\pt)\llra{Res} H^\alpha(G) \llra{Tr} H^\alpha(\pt)
\]
is multiplication by $\#G$.  The middle group is isomorphic to $H^{|\alpha|}_{sing}(\pt)$ and is therefore zero when $|\alpha|\neq 0$.  It follows that when $|\alpha|\neq 0$ the group $H^\alpha(\pt)$ is torsion and annihilated by $\#G$.  The index $(\ell+1)V-\ell$ of the  middle group from (\ref{eq:VR}) can have rank $0$ only when $\ell=-2$ and $\dim V=2$.  Avoiding this case therefore ensures that the middle group is annihilated by $\#G$, and as a consequence inverting primes not dividing $\#G$ does not affect the vanishing requirement.  This is sometimes useful, as we will shortly see.
\end{remark}

\medskip

In the ring $\M=H^\star(\pt;\mZ)$ the portion graded by virtual representations $W+n$ and $-W+k$ for all (non-virtual) representations $W$ and all $n,k\in \Z$ is somewhat more accessible then the entirety of $\M$.  We call this the \textbf{regular portion} of $\M$.  Note that the regular portion is not a subring.  The direct sum of terms $\M^{W+n}$ for $W$ a representation and $n\in \Z$ will be called the ``positive'' part of $\M$, and the direct sum of terms $\M^{-W+n}$ for $n\in \Z$ will be called the ``negative'' part.  Note that these terms only refer to the regular region.  Recall that $\M^{W+n}=0$ for $n> 0$ or $n<-\dim W$, and likewise $\M^{n-W}=0$ for $n\leq 0$ or $n> \dim W$ (cf. Propositions~\ref{pr:a-group} and \ref{pr:negative-W} for parts of this; the other parts follow from suspension isomorphisms and cellular homology).  So the nonvanishing groups are concentrated in positive and negative ``cones''.  

For a particular representation $V$, let $\M(V)$ be the subring of $\M$ consisting of all gradings $aV+b$ for $a,b\in \Z$.  This is contained in the regular portion of $\M$.  Observe that the vanishing hypotheses concern only these subrings, and in particular do not involve the ``irregular'' portion of $\M$.    

Our goal in this appendix is to prove Proposition~\ref{pr:VR-examples}, which says that the vanishing requirement is always satisfied if $\dim V\geq 3$ and $G$ is one of the groups $C_p$, $C_{p^2}$ ($p$ a prime), and $\Sigma_3$.  We will start with the case of the cyclic groups, and for these it will be useful to recall some basic facts about their representation theory.  

Let $C_n$ be the cyclic group with $n$ elements.
Let $\lambda(k)$ denote $\R^2=\C$ with the generator of $C_n$ acting as multiplication by $e^{\frac{2\pi i k}{n}}$.    In addition to $\lambda(k)$ only depending on $k$ modulo $n$, one also has $\lambda(k)\iso \lambda(n-k)$ (using complex conjugation for the isomorphism).  
It is known that $RO(C_n)$ has a basis consisting of $1,\lambda(1),\ldots,\lambda(\lfloor \frac{n-1}{2} \rfloor)$ and---when $n$ is even---the sign representation $\sigma$.  We take this as our ordered basis for the $RO(C_n)$-grading.  

Now we specialize to $n=p^e$, where $p$ is a prime. The following result is a small expansion of Proposition 4.25 of \cite{Z} (though our proof is different), and is also related to Proposition 2.2 of \cite{HHR3}.

\begin{prop}
\label{pr:HZ-equiv}
Let $G=C_{p^e}$ where $p$ is a prime. After inverting the primes other than $p$ there is an equivalence of $H\mZ$-modules fitting in the commutative diagram
\begin{equation}
\label{eq:HZ-equiv}
\xymatrix{
H\mZ\Smash S^{\lambda(k)}\ar[rr]^\he && H\mZ\Smash S^{\lambda(rk)}\\
& H\mZ\Smash S^0 \ar[ul]^-{\id\Smash a_{\lambda(k)}}\ar[ur]_-{\ \ \id\Smash a_{\lambda(rk)}
}}
\end{equation}
whenever $(r,p)=1$.
\end{prop}

\begin{proof}
We will see this by using the equivalence of homotopy theories between $H\mZ$-modules and chain complexes of $\mZ$-modules.
In particular, we will compute cellular chain complexes for $S^{\lambda(k)}$ and $S^{\lambda(rk)}$ and then show the complexes are quasi-isomorphic after localization at $p$. 

Let's first recall how to construct the cellular chain complex of $\mZ$-modules for a given $G$-CW complex. A $G$-CW complex is built by inductively attaching cells of the form $G/H\times D^n$ where $H$ is a subgroup of $G$ and $D^n$ denotes the $n$-disk with the trivial $G$-action. Note if we forget the $G$-action, then attaching a $G/H\times D^n$ equivariant cell is just attaching $\#(G/H)$ cells of dimension $n$. The non-equivariant cellular chain complex thus inherits an action of the group $G$ given by how $G$ permutes the cells. This action turns the non-equivariant cellular chain complex into a complex of $\Z[G]$-modules, where in each degree we have a direct sum of permutation modules $\Z[G/H]$, one for each equivariant cell of the form $G/H \times D^n$. 

We then use this chain complex of $\Z[G]$-modules to get the Mackey functor cellular chain complex by applying the ``fixed point functor'' \[FP\colon \Z[G]\text{-mod}\to \mZ\text{-mod}\] in each degree. For a $\Z[G]$-module $N$ \emph{the fixed point Mackey functor} is defined by $FP(N)\colon G/J\mapsto N^J$. Restriction and transfer maps are given by inclusion and norm (i.e. sum-over-orbits), respectively. One can verify the Mackey functor $FP(\Z[G/H])$ is isomorphic to the free $\mZ$-module generated at the spot $G/H$, so we get a chain complex of free $\mZ$-modules; this is the cellular chain complex of $\mZ$-modules.

Observe the functor $FP$ commutes with localization at $p$. Thus in order to show the cellular chain complex of $\mZ$-modules for $S^{\lambda(k)}$ and $S^{\lambda(rk)}$ are quasi-isomorphic after localization, it  will suffice to show the corresponding complexes of $\Z[G]$-modules are quasi-isomorphic after localization.

We now describe a cell structure for $S^{\lambda(k)}$. Fix a generator $t\in C_{p^e}$. Identify $C_{p^e}$ with $\Z/p^e$ by letting $t$ correspond to the coset of $1$. We will view $C_{p^e}$ as containing elements $0, 1, \dots, p^e-1$ when discussing subgroups of $C_{p^e}$ and view $C_{p^e}$ as containing $1, t, \dots t^{p^e-1}$ when considering elements acting on $\lambda(k)$. Let $\zeta=e^{\frac{2\pi i}{p^e}}$.
Consider the ray $\ray{01}$ in $\lambda(k)$; its orbit under the $C_{p^e}$-action consists of the rays through the points $\zeta^{kl}$ for $l\geq 0$, which may be identified with the subgroup of $C_{p^e}$ generated by $k$. Note the smallest nonzero element in this subgroup is $p^j$ where $p^j$ is the largest power of $p$ dividing  $k$.  So the ray from $0$ to $\zeta^{p^j}$ is the first ray in the orbit when one moves counterclockwise from $\ray{01}$, and also $(k)=(p^j)$. 

Write $p^j=mk$ in $C_{p^e}$ for $0\leq m < p^e$, so that $\zeta^{p^j}=t^m\cdot 1$ in the representation $\lambda(k)$.  Note that $(m,p)=1$. A cell decomposition for $S^{\lambda(k)}$ then consists of two fixed $0$-cells (corresponding to $0$ and $\infty$), a $1$-cell of type $C_{p^e}/(p^{e-j})$  (corresponding to our orbit of rays), and a $2$-cell of type $C_{p^e}/(p^{e-j})$ which is attached to the $1$-cell via $1-t^m$.  That is, $S^{\lambda(k)}$ has a cellular chain complex (at the level of $\Z[G]$-modules) of the form
\[ \xymatrix{
0\ar[r] & \Z[C_{p^e}/(p^{e-j})] \ar[r]^{1-t^m} & \Z[C_{p^e}/(p^{e-j})] \ar[r] & \Z \oplus \Z\ar[r] & 0.
}
\]
The condition $(m,p)=1$ readily implies that the homology in degree $2$ is $\Z$, generated by the sum $1+t+t^2+\cdots+t^{(p^{e-j}-1)}$.  Note that $t$ fixes this class.

Let's apply the same reasoning to $\lambda(rk)$ to get a similar complex. Since $(r,p)=1$ the largest power of $p$ dividing $rk$ will again be $p^j$. Thus we have the same $\Z[G]$-modules in the complex for $S^{\lambda(rk)}$. We just replace the map $1-t^m$ by $1-t^{ms}$, where $s$ is a positive integer such that $rs=1$ in $C_{p^e}$ (since then $p^j=mk=(ms)(rk)$).  Now we write down the map of complexes
\[ \xymatrix{
0\ar[r] & \Z[C_{p^e}/(p^{e-j})] \ar[r]^{1-t^{ms}}\ar[d]_{1+t^m+\cdots+t^{m(s-1)}} & \Z[C_{p^e}/(p^{e-j})] \ar[r]\ar[d]_\id & \Z \oplus \Z \ar[r]\ar[d]_\id & 0\\
0 \ar[r] & \Z[C_{p^e}/(p^{e-j})] \ar[r]^{1-t^{m}} & \Z[C_{p^e}/(p^{e-j})] \ar[r] & \Z\oplus\Z \ar[r] & 0 \period
}
\]
On homology the map in degree $2$ induces multiplication by $s$, because $t$ acts as the identity on $H_2$.   But $s$ is a unit in $\Z/p^e$ which implies $(s,p)=1$, so the induced map on homology is an isomorphism after localization at $p$.
\end{proof}

Because of the equivalence from (\ref{eq:HZ-equiv}) we can for most purposes ignore $\lambda(rk)$ in the $RO(G)$-grading for $H\mZ$, at least after inverting primes away from $\# G$.  For example, if $n=9$ then we can get away with only considering the portion of $RO(G)$ spanned by $1$, $\lambda(1)$, and $\lambda(3)$, with the understanding that a group in the $RO(G)$-grading indexed by a representation $V$ having other summands is isomorphic (after localization) to the corresponding group indexed by $V'$ where the unwelcome summands are replaced with either $\lambda(1)$ or $\lambda(3)$.  The vanishing requirement deals with the homotopy fiber of an Euler class $a_V$, and using the triangles (\ref{eq:HZ-equiv}) one sees that the requirement for $V$ is equivalent to the requirement for $V'$.  Note that we are using Remark~\ref{re:loc} to know that it is enough to consider the vanishing requirements after inverting the primes away from $\#G$.  

\subsection{The case \mdfn{$G=C_p$, $p>2$}}
By the above considerations, we can restrict our attention to degrees $a+b\lambda(1)$ for $a,b\in \Z$.  Let us shorten $\lambda(1)$ to just $\lambda$.  Recall that the rank of $a+b\lambda$ is $a+2b$.  

The structure of $H\mZ^\star(\pt)$ can be deduced from the computations and techniques in \cite{L}, \cite{FL}, and \cite{C}, but it is difficult to find it stated precisely in those sources.  Another reference is \cite[Proposition 6.3]{Z}.
The positive part (when $b>0$) is the polynomial ring $\Z[u_{\lambda},a_\lambda]/(pa_\lambda)$.  
The negative part (when $b<0$) has the following elements:
\begin{itemize}
\item Elements $\frac{p}{u_\lambda^k}$ for $k\geq 1$, which generate $\Z$ summands; \bigskip
\item Elements $\Sigma^1( \frac{1}{u_\lambda^k a_\lambda^l} )$ for $k,l\geq 1$, which generate $\Z/p$ summands.  The degree of this class is $1-k(\lambda-2)-l\lambda$ (note in particular the ``$1+$'' part---the $\Sigma^1$ appears in the notation to remind us that the degree of this element is one more than the expected degree of $\frac{1}{u_\lambda^k a_\lambda^l} $).
\end{itemize}
Multiplication by $u_{\lambda}$ and $a_\lambda$ respects the fraction notation, in the sense that
\begin{align*}
& u_\lambda \cdot \tfrac{p}{u_\lambda^k}=\tfrac{p}{u_\lambda^{k-1}}, \quad
a_\lambda \cdot \tfrac{p}{u_\lambda^k}=0,\\
&u_\lambda\cdot \Sigma^1\bigl ( \tfrac{1}{u_\lambda^ka_\lambda^l} \bigr )=\begin{cases}
\Sigma^1( \tfrac{1}{u_\lambda^{k-1}a_\lambda^l}) & \text{if $k\geq 2$,}\\
0 & \text{if $k=1$},
\end{cases}\\
&a_\lambda\cdot \Sigma^1\bigl ( \tfrac{1}{u_\lambda^ka_\lambda^l} \bigr )=\begin{cases}
\Sigma^1( \tfrac{1}{u_\lambda^{k}a_\lambda^{l-1}}) & \text{if $l\geq 2$,}\\
0 & \text{if $l=1$}.
\end{cases}
\end{align*}
Note the slogan ``if the resulting fraction is not one of the allowed ones, then the product is zero''.  This is all we will need of the ring structure, but to complete the picture let us just state that the product of two $\Sigma^1$-classes is always zero.

A key thing to focus on here is the non-existent classes $\frac{1}{u_\lambda^k}$ and $\Sigma^1(\frac{1}{u_\lambda})$.  We have $p$ times the first classes, and we have $a_\lambda$- and $u_\lambda$-divisions (infinitely many times) of the second set of classes.   Note the fact that this second type of classes is $p$-torsion follows from the fact that $p a_\lambda=0$, using the $a_\lambda$-divisibility.   

\begin{remark}
Observe that the positive cone of $\M$ has nonzero groups only in degrees having even rank.  For the negative cone of $\M$, the elements $\frac{p}{u_\lambda^k}$ are in degrees with even rank but all of the other nonzero elements are in degrees with odd rank.  
\end{remark}

If $V$ is a nonzero representation such that $V^G= 0$ then $V=k\lambda$ for some $k\geq 1$.  We want to examine the sequence
\[ H^{\ell k\lambda-\ell}(\pt)\llra{a_\lambda^k} H^{(\ell+1)k\lambda-\ell}(\pt)\llra{a_\lambda^k}
H^{(\ell+2)k\lambda-\ell}(\pt)
\]
for $\ell\in \Z$.  We will additionally assume $k\neq 1$ to avoid the $\dim V=2$ anomaly.  

When $\ell\geq 0$ all of the groups are in the positive cone of $\M$.
Because the positive cone is polynomial it is immediate that the first map will be surjective and the second injective as long as the rank of the middle index is greater than $0$. This is the condition that $2k>\frac{\ell}{\ell+1}$ and always holds for $\ell\geq 0$.

We do not need to consider $\ell=-1$ since that case is always satisfied.  For $\ell=-2$ the condition is satisfied if and only if the middle group in the sequence is zero, i.e. $H^{2-k\lambda}(\pt)=0$.  Since the degree $2-k\lambda$ lies in the negative cone and has even rank, the cohomology group will be nonzero only if $2-k\lambda$ is a negative multiple of $2-\lambda$.  Since $k\neq 1$, the group is therefore zero.  

So we are down to the $\ell\leq -3$ case, which puts all three of the groups in the negative cone of $\M$.  
In the negative cone, multiplication by $a_\lambda^k$ is an isomorphism as long as the index of the target has non-positive rank {\it and\/} is not a negative multiple of $2-\lambda$ (these are the degrees containing the $\frac{p}{u_\lambda^r}$ classes).  
So the first map in our sequence will be surjective when $(\ell+1)2k-\ell\leq 0$ and $2(\ell+1)k\neq \ell$, or equivalently $(\ell+1)2k<\ell$.
Since $\ell\leq -3$ this is equivalent to 
$2k> \frac{\ell}{\ell+1}$, which is always true.
Likewise, the second map in our sequence will be injective provided that
$(\ell+2)2k-\ell\leq 0$, or $2k\geq \frac{\ell}{\ell+2}$.  This is guaranteed since $k\geq 2$ and $\ell\leq -3$.

The upshot of this analysis is that the vanishing requirement is satisfied as long as $k\neq 1$, i.e. $\dim V\geq 3$.

\subsection{The case \mdfn{$C_2$}}
Here $H^\star(\pt;\mZ)$ was probably first computed in an  unpublished paper of Stong.  The additive structure can be deduced from the results presented in \cite[Theorem 2.1]{L}.  Published descriptions of the ring are \cite[Theorem 2.8]{D1} and also \cite[Theorem 9.9.19]{HHR4}. 
The sign representation $\sigma$ is the only nontrivial irreducible in this case.  Its double $2\sigma$ is orientable and so gives rise to a class $u_{2\sigma}$ and a mirror class $\frac{2}{u_{2\sigma}}$.  
One can think of $2\sigma$ as playing the role of $\lambda$ from the $C_p$ case.  

The positive cone of $\M$ is the polynomial ring $\Z[u_{2\sigma},a_\sigma]/(2a_\sigma)$.  
The negative cone has the following elements:
\begin{itemize}
\item Elements $\frac{2}{u_{2\sigma}^k}$ for $k\geq 1$, which generate $\Z$ summands;\bigskip
\item Elements $\Sigma^1( \frac{1}{u_{2\sigma}^k a_\sigma^l} )$ for $k,l\geq 1$, which generate $\Z/2$ summands. (These elements follow the same degree conventions as in the $C_p$ case above.) 
\end{itemize}
The rest of the analysis is identical to the $C_p$ case, with the conclusion that $V=k\sigma$ satisfies the vanishing hypotheses as long as $k\neq 2$ (and in particular, when $\dim V\geq 3$).  

\subsection{The case \mdfn{$C_{p^2}$, $p>2$}}

Here we take our ordered basis to be $1$, $\lambda(1)$, and $\lambda(p)$.  For convenience set $\lambda_i=\lambda(p^i)$, for $i=0,1$, and write $a_i=a_{\lambda_i}$, $u_i=u_{\lambda_i}$.    Note that $d(\lambda_i)=p^{2-i}=e(\lambda_i)$, and that for $r,s>0$ we have $d(r\lambda_0+s\lambda_1)=p$ and $e(r\lambda_0+s\lambda_1)=p^2$.

The computations of $\M$ are taken from \cite[Theorem 6.10]{Z} and summarized below, with the one exception that \cite{Z} failed to mention the $\frac{p}{u_1^s}$ classes.  Note that the following does not give all of $\M$, but only the {\it regular\/} portion.

The positive cone of $\M$  is the quotient $\Z[u_0,u_1,a_0,a_1]/(p^2a_0,pa_1,a_1u_0=pa_0u_1)$.  The last relation is the so-called $au$-relation (or ``gold relation'').  Note that a consequence of this relation is $a_1^2u_0=a_1(pa_0u_1)=0$.  

The negative cone is spanned by the following elements, where again we use the $\Sigma^{1}$ notation to indicate the elements are shifted $+1$ in degree:
\begin{itemize}
\item Elements $\frac{p^2}{u_0^ru_1^s}$, $r>0$, which generate $\Z$-summands;\bigskip
\item Elements $\frac{p}{u_1^s}$, $s>0$, which also generate $\Z$-summands;\bigskip
\item Elements $\frac{1}{a_0^t}\Sigma^1(\frac{1}{u_0^r u_1^s})$, $t>0$, $r>0$, and $s\geq 0$, which generate $\Z/p^2$-summands   (note that the $\frac{1}{a_0^t}$ can be moved inside the $\Sigma^1$ if desired);\bigskip
\item Elements $\frac{1}{a_0^ra_1^s}\Sigma^1 (\frac{1}{u_1^l})$ with $l>0$, $r\geq 0$, $s\geq 0$, $r+s>0$, which being multiples of $a_1$ are $p$-torsion (same note as above).
\end{itemize}
As in other cases, multiplications by $u_i$ and $a_i$ follow the fraction notation and are zero when the fraction is not defined.  

The key thing to notice here is that we have the usual classes $\frac{p^2}{u_0^ku_1^l}$ and $\frac{p}{u_1^l}$ together with infinite $a_0$- and $a_1$-divisors of the $\Sigma^1$ associates of these classes.  
One might have expected elements $\frac{1}{a_0^ra_1^s}\Sigma^1(\frac{1}{u_0^ku_1^\ell})$, but these vanish by the $au$-relation:
\[ 
\tfrac{1}{a_0^ra_1^s}\Sigma^1(\tfrac{1}{u_0^ku_1^\ell})=a_1^2u_0
\tfrac{1}{a_0^ra_1^{s+2}}\Sigma^1(\tfrac{1}{u_0^{k+1}u_1^\ell})=0.
\]

We now explore the vanishing hypotheses for a representation $V=j\lambda_0+k\lambda_1$, where $j,k\geq 0$ and $j+k\geq 2$ to avoid the $\dim V=2$ anomaly.  We need to examine the sequence
\[ H^{\ell (j\lambda_0+k\lambda_1)-\ell}(\pt)\llra{a_0^ja_1^k} H^{(\ell+1)(j\lambda_0+k\lambda_1)-\ell}(\pt)\llra{a_0^ja_1^k}
H^{(\ell+2)(j\lambda_0+k\lambda_1)-\ell}(\pt).
\]

Start with $\ell\geq 0$, so that all three groups are in the positive cone.
The positive cone is spanned by elements $a_0^Pa_1^Qu_0^Ru_1^S$ with $P,Q,R,S\geq 0$, where if $R>0$ then $Q=0$.  Such an element lies in our middle group if
\[ (\ell+1)(j\lambda_0+k\lambda_1)-\ell=P\lambda_0+Q\lambda_1+R(\lambda_0-2)+S(\lambda_1-2),
\]
or equivalently
\begin{align*}
&\ell=2(R+S),\\
&(\ell+1)j=P+R,\\
&(\ell+1)k=Q+S.
\end{align*}
Note that we have only solutions when $\ell$ is even, and so we restrict to that case.

If $R>0$ then $Q=0$ and our solution is $S=(\ell+1)k$, $R=\frac{\ell}{2}-(\ell+1)k$, $P=(\ell+1)(j+k)-\frac{\ell}{2}$.  But since $R\geq 0$ this is only valid if $k=0$, in which case we are looking at the element $a_0^{(\ell+1)j-\frac{\ell}{2}}u_0^{\frac{\ell}{2}}$.

If $R=0$ then we have $P=(\ell+1)j$, $S=\frac{\ell}{2}$, and $Q=(\ell+1)k-\frac{\ell}{2}$.  But as $Q\geq 0$ we must have $k\geq 1$ here, and our single solution is $a_0^{(\ell+1)j}a_1^{(\ell+1)k-\frac{\ell}{2}}u_1^{\frac{\ell}{2}}$.    

So we have two non-overlapping cases ($k=0$ and $k\geq 1$), and in each case the middle group is spanned by a single element.  

In the $k=0$ case we are looking at multiplication by $a_0^j$ and the single generator $a_0^{(\ell+1)j-\frac{\ell}{2}}u_0^{\frac{\ell}{2}}$.  The second map in our sequence is clearly injective, and the first is surjective because 
$(\ell+1)j-\frac{\ell}{2}\geq j$.  

In the $k\geq 1$ case our middle group is spanned by 
$a_0^{(\ell+1)j}a_1^{(\ell+1)k-\frac{\ell}{2}}u_1^{\frac{\ell}{2}}$.  Multiplication by $a_0^ja_1^k$ is injective because there are no $u_0$'s in the monomial, and is surjective because $(\ell+1)j\geq j$ and $(\ell+1)k-\frac{\ell}{2}\geq k$.  This completes the verification when $\ell\geq 0$.

The case $\ell=-1$  is always satisfied, and $\ell=-2$ works precisely when the middle group
$H^{2-(j\lambda_0+k\lambda_1)}(\pt)$ vanishes.  Note that
$2-(j\lambda_0+k\lambda_1)$ has even rank, and the only nonzero elements of the negative cone in even ranks are ones of the form $\frac{??}{u_0^ru_1^s}$ where $??$ is either $p$ or $p^2$.  This element lies in degree $r(2-\lambda_0)+s(2-\lambda_1)$, and the only way this coincides with the degree of our middle group is if $j+k=1$, which is contrary to our  hypothesis.
So $\ell=-2$ is now verified.

Finally we examine $\ell\leq -3$, where all groups in the sequence are in the negative cone. 
In the negative cone we have the elements $\frac{??}{u_0^Ru_1^S}$ in degree $R(2-\lambda_0)+S(2-\lambda_1)$ and the elements $\frac{1}{a_0^Pa_1^Q} \Sigma^1(\frac{1}{u_0^Ru_1^S})$ where $Q=0$ if $R>0$, and here the degree is $1-P\lambda_0-Q\lambda_1+R(2-\lambda_0)+S(2-\lambda_1)$.  Observe that elements of the first type all have rank zero, and so could be an element of our middle group only if
\[ 0=(\ell+1)2(j+k)-\ell;
\]
but this is impossible since we have $\ell\leq -3$ and $j+k\geq 2$.  
To determine which elements of the second type are in our middle group we need to solve
\[ (\ell+1)(j\lambda_0+k\lambda_1)-\ell=1-P\lambda_0-Q\lambda_1+R(2-\lambda_0)+S(2-\lambda_1).
\]
This becomes
\[ -(\ell+1)=2(R+S), \quad -(\ell+1)j=P+R, \quad -(\ell+1)k=Q+S.
\]
There are solutions only when $\ell$ is odd.
If $R>0$ we are forced to have $Q=0$ and therefore $S=-(\ell+1)k$, $R=-\frac{\ell+1}{2}+(\ell+1)k$, $P=-(\ell+1)(j+k)+\frac{\ell+1}{2}$.  But since $R>0$ we must have $k=0$, and so our unique solution is the element 
$\frac{1}{a_0^{-(\ell+1)j+\frac{\ell+1}{2}}} \Sigma^1\Bigl (\frac{1}{u_0^{-\frac{\ell+1}{2}}} \Bigr)$.  

If $R=0$ then we get $S=-\frac{\ell+1}{2}$, $Q=-(\ell+1)k+\frac{\ell+1}{2}$, $P=-(\ell+1)j$.  Since $Q\geq 0$ we must have $k\geq 1$, so this case does not overlap the previous one.  Here our element is
\[ \frac{1}{a_0^{-(\ell+1)j}a_1^{-(\ell+1)k+\frac{\ell+1}{2}}} \Sigma^1 \Biggl ( \frac{1}{u_1^{-\frac{\ell+1}{2}}} \Biggr ).
\]

Now we analyze the two cases separately.  If $k=0$ we are looking at multiplication by $a_0^j$, but our element is infinitely $a_0$-divisible: so the first map in our sequence is surjective.  Likewise, the second map is injective because the power of $a_0$ in the denominator is larger than $j$ (because $-(\ell+1)j+\frac{\ell+1}{2} > j$, since $j\geq 2$ in this case).

When $k\geq 1$ we are looking at an element that is infinitely divisible by both $a_0$ and $a_1$, so the first map is surjective.  And likewise, the second map is injective because the power on $a_0$ in the denominator is strictly larger than $j$ and the power of $a_1$ is at least $k$.

Once again, the final conclusion is that the vanishing hypotheses are satisfied as long as $\dim V\geq 3$.

\subsection{The \mdfn{$C_4$} case}
Here the computations were done in \cite{G}, and then re-done via other methods (with a few corrections) in \cite{Y2}.  Some partial results were given earlier in \cite[Lemma 6.11 and Corollary 6.13]{HHR2}.  See 
\cite[Theorem 13]{Y2} for the computation of $\M$ we describe here.

Let $\sigma$ denote the sign representation and $\lambda$ denote $\R^2$ under counterclockwise rotation by $\frac{\pi}{2}$.  Take $1$, $\lambda$, $\sigma$ for our ordered basis of $RO(C_4)$.  The basic elements making up $\M$ are $u_\lambda$, $u_{2\sigma}$ (both torsion-free), $a_\lambda$ (order 4), and $a_\sigma$ (order 2).  The positive cone is 
\[ \Z[u_\lambda,u_{2\sigma},a_\lambda,a_\sigma]/(4a_\lambda, 2a_\sigma, a_\sigma^2u_\lambda-2u_{2\sigma}a_\lambda)
\]
with the final relation being the so-called ``$au$-relation''.  

For an element $r(1)+s(\lambda)+t(\sigma)\in RO(C_4)$ we call $r$ the \dfn{fixed-point index}, $s$ the \mdfn{$\lambda$-index}, and $t$ the \mdfn{$\sigma$-index}.  Note that the generators $u_\lambda$, $u_{2\sigma}$, $a_\lambda$, and $a_\sigma$ all have even fixed-point index.  If we divide $\M$ up into two pieces according to the parity of the fixed-point index---call these pieces $\M_{ev}$ and $\M_{odd}$---then the action of the four generators will preserve these pieces.  

We describe $\M_{ev}$ and $\M_{odd}$ via two charts.  Figure~\ref{fig:C4-one} shows $\M_{ev}$; it turns out that all of the elements here can be described as rational expressions in the four basic generators.  
Here are instructions on how to read this chart:
\begin{itemize}
\item Each labelled element (in red) generates a copy of $\Z$.  Single black dots represent a $\Z/2$, and a black dot with an open circle around it represents $\Z/4$.  
\item This is a projection of a three-dimensional chart onto a plane, so rays that look like they might intersect line segments possibly don't. 
The diagonal lines and rays are moving outside the plane containing the red classes.  

\item Multiplication by $u_\lambda$ moves up one unit, and multiplication by $u_{2\sigma}$ moves right one unit.  (Note that these operations do not always take labelled generators to other labelled generators---sometimes an extra integer multiple appears).  
\item Multiplication by $a_\lambda$ moves up and to the right, whereas multiplication by $a_\sigma$ moves down and to the right.  Both of these operations are labelled with diagonal lines in the picture.  For example, observe the relation $a_\sigma^2=a_\lambda\cdot \frac{2u_{2\sigma}}{u_\lambda}$, which is a form of the $au$-relation.  Also observe that $a_\lambda$ and $a_\sigma$ both annihilate $\frac{4}{u_\lambda}$, as one can observe from the absence of lines there.  
\item Observe the $au$-relation $a_\sigma^2 u_\lambda=2a_\lambda u_{2\sigma}$.  In particular, note that multiplication by $a_\sigma$ hits {\it twice\/} the generator of the $\Z/4$, as it must since the source is $2$-torsion.  
\item The powers $a_\sigma^k$ for $k\geq 3$ are all infinitely-divisible by $a_\lambda$, as indicated by the downward-left-pointing arrows.  
\item There is one place in the drawn portion of the diagram where
two different phenomena run into each other in the same degree: this is the $\frac{4u_{2\sigma}^2}{u_\lambda^2}$ class and the $\frac{a_\sigma^4}{a_\lambda^2}$ class.  The first class generates a $\Z$ and the second a $\Z/2$, and the group in this degree is $\Z\oplus \Z/2$.  This always happens in the degrees containing the classes $4(\frac{u_{2\sigma}}{u_\lambda})^k\cdot u_{2\sigma}^i$, and in every other degree the group is one of $0$, $\Z/2$, $\Z/4$, and $\Z$.  [Because the degrees of the $\Z\oplus \Z/2$ groups are not in the regular portion, they don't come up in our analysis of the vanishing requirement.]
\item Trace out the connected component of $1$ in the graph, and make sure you understand it (it is a piece of a grid on a plane).  Multiplication by $u_{2\sigma}$ copies this component over and over, moving it to the right in the chart.  ``Division'' by $u_{2\sigma}$ produces a truncated version of the component, which is then copied over and over as one moves to the {\it left\/} in the chart.
\item The fixed point index is preserved by $a_\lambda$ and $a_\sigma$, so it is constant on the connected components of the chart.  Multiplication by $u_{2\sigma}$ and multiplication by $u_\lambda$ each subtract two from the fixed-point index: so the fixed point index drops by two when jumping from a component to the next one to the right (or up).  
\end{itemize}

\begin{figure}[ht]
\begin{tikzpicture}[scale=2,yscale=1.4]
\draw[] (-0.05,0) node[left,red] {\small $1$};
\draw[] (1.05,0.05) node[left,red] {\small $u_{2\sigma}$};
\draw[] (2.05,0.05) node[left,red] {\small $u_{2\sigma}^2$};
\draw[] (-1,0) node[left,red] {\small $\frac{2}{u_{2\sigma}}$};
\draw[] (-2,0) node[left,red] {\small $\frac{2}{u_{2\sigma}^2}$};

\draw[] (0,1.05) node[left,red] {\small $u_\lambda$};
\draw[] (1,1.05) node[left,red] {\small $u_{2\sigma}u_\lambda$};
\draw[] (2.2,1.08) node[left,red] {\small $u_{2\sigma}^2u_\lambda$};
\draw[] (-1,1.05) node[left,red] {\small $\frac{u_\lambda}{u_{2\sigma}}$};
\draw[] (-2,1.05) node[left,red] {\small $\frac{u_\lambda}{u_{2\sigma}^2}$};

\draw[] (0,-1) node[left,red] {\small $\frac{4}{u_\lambda}$};
\draw[] (1,-1) node[left,red] {\small $\frac{2u_{2\sigma}}{u_\lambda}$};
\draw[] (2,-1) node[left,red] {\small $\frac{2u_{2\sigma}^2}{u_\lambda}$};
\draw[] (-1,-1) node[left,red] {\small $\frac{4}{u_{2\sigma}u_\lambda}$};
\draw[] (-2,-1) node[left,red] {\small $\frac{4}{u_{2\sigma}^2u_\lambda}$};

\draw[] (0,-2) node[left,red] {\small $\frac{4}{u_\lambda^2}$};
\draw[] (1,-2) node[left,red] {\small $\frac{4u_{2\sigma}}{u_\lambda^2}$};
\draw[] (2,-2) node[left,red] {\small $\frac{4u_{2\sigma}^2}{u_\lambda^2}$};
\draw[] (-1,-2) node[left,red] {\small $\frac{4}{u_{2\sigma}u_\lambda^2}$};
\draw[] (-2,-2) node[left,red] {\small $\frac{4}{u_{2\sigma}^2u_\lambda^2}$};

\fill (1.05,-0.75) circle (1pt);
\fill (2.05,-0.75) circle (1pt);

\draw[->] (-2.1,0.1) -- (-1.65,0.6);

\foreach \y in {0,...,4}
    \fill (-1.75+\y,0.48) circle (1pt);
\foreach \y in {2,...,4}
    \draw (-1.75+\y,0.48) circle (2pt);

\foreach \y in {0,...,4}
    \fill (-1.95+\y,0.25) circle (1pt);

\foreach \y in {0,...,4}
    \fill (-1.95+\y,1.25) circle (1pt);

\foreach \y in {2,...,4}
     \draw[] (-1.95+\y,0.25) circle (2pt);

\foreach \y in {0,...,4}
     \draw[] (-1.95+\y,1.25) circle (2pt);


\fill (0.4,-0.32) circle (1pt);

\fill (0.6,-0.1) circle (1pt);

\fill (1.4,-0.32) circle (1pt);

\fill (1.6,-0.1) circle (1pt);

\foreach \y in {0,...,4}
    \fill (-1.6+\y,0.68) circle (1pt);
\foreach \y in {0,...,4}
    \fill (-1.41+\y,0.88) circle (1pt);


\draw (-2.6,0.68) -- (-2.01,0.3);
\draw (-2.4,0.88) -- (-1.81,0.5);

\draw[] (-2.22,1.4)--(-0.74,0.46);
\draw[] (-2.1,1.0)--(-0.95,0.26);

\draw[->] (-2.1,1.1) --(-1.8,1.4);
\draw[->] (-1.59,0.69) --(-1.31,1.0);
\draw[->] (-1.1,0.1) --(-0.65,0.6);

\draw[->,blue,thick] (-1.2,1.4)--(0.58,0.24);
\draw[->,blue,thick] (-1.1,1.0)--(0.78,-0.22);
\draw[->,blue,thick] (-0.09,0)--(2.47,-1.66);

\draw[->,blue,thick] (-1.1,1.1) --(-0.8,1.4);
\draw[->,blue,thick] (-0.59,0.69) --(-0.31,1.0);
\draw[->,blue,thick] (-0.1,0.1) --(0.35,0.6);
\draw[->,blue,thick] (0.41,-0.31) --(0.7,0.02);
\draw[->,blue,thick] (0.94,-0.9) --(1.18,-0.58);

\draw[->] (-0.2,1.4)--(1.58,0.24);
\draw[->] (-0.1,1.0)--(1.78,-0.22);
\draw[->] (0.91,0)--(2.37,-0.95);

\draw[->] (-0.1,1.1) --(0.2,1.4);
\draw[->] (0.41,0.69) --(0.69,1.0);
\draw[->] (0.9,0.1) --(1.35,0.6);
\draw[->] (1.41,-0.31) --(1.7,0.02);
\draw[->] (1.94,-0.9) --(2.18,-0.58);

\draw[->] (0.8,1.4)--(2.58,0.24);
\draw[->] (0.9,1.0)--(2.78,-0.22);
\draw[] (1.91,0)--(2.6,-0.45);

\draw[->] (0.9,1.1) --(1.2,1.4);
\draw[->] (1.41,0.69) --(1.69,1.0);
\draw[->] (1.9,0.1) --(2.35,0.6);
\draw[->] (2.41,-0.31) --(2.7,0.02);

\draw[->] (1.8,1.4)--(2.78,0.76);
\draw[->] (1.9,1.0)--(2.78,0.44);

\draw[->] (1.9,1.1) --(2.2,1.4);
\draw[->] (2.41,0.69) --(2.69,1.0);

%
\fill (1.73,-1.18) circle(1pt);
\fill (2.29,-1.55) circle (1pt);
\fill (1.5,-1.45) circle (1pt);
\fill (2.06,-1.82) circle (1pt);

\draw[blue,->,thick] (1.5,-1.45) -- (1.8,-1.1);
\draw[blue,->,thick] (1.5,-1.45)--(1.23,-1.75);

\draw[blue,->,thick] (2.06,-1.82) -- (2.36,-1.47);
\draw[blue,->,thick] (2.06,-1.82)--(1.81,-2.12);
\draw[blue,->,thick] (0.91,-1.08) -- (2.24,-1.92);

\end{tikzpicture}
\caption{The module $\M_{ev}$}
\label{fig:C4-one}
\end{figure}
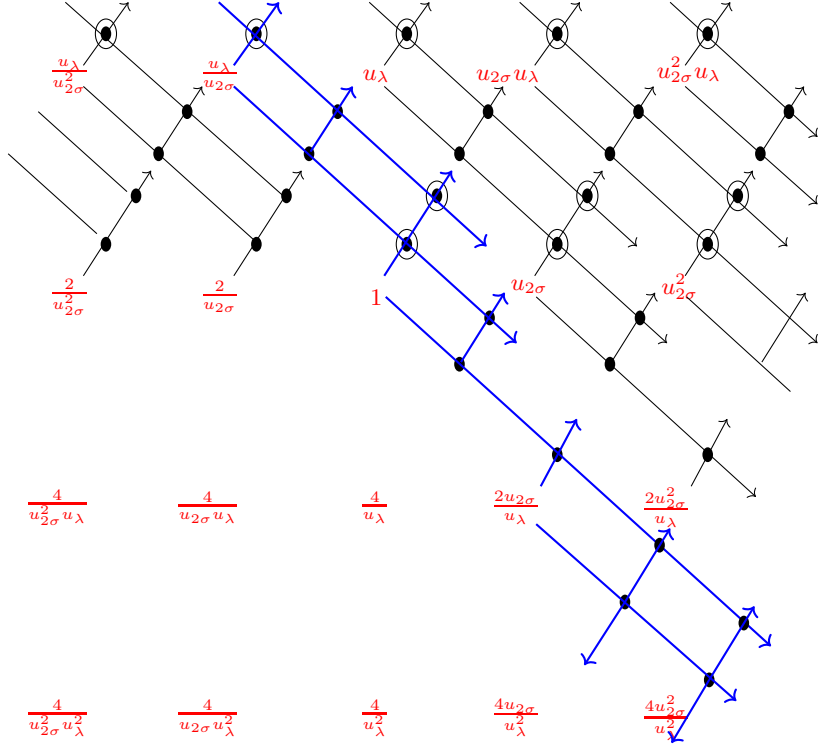

The second chart is Figure~\ref{fig:C4-two}, showing $\M_{odd}$.  These are the `exotic' classes in $\M$.  All of these are infinitely-divisible by $a_\lambda$.  Here are the rules for reading this chart:
\begin{itemize}
\item The basic organization of the chart follows the previous one: multiplication by $u_{2\sigma}$ moves one spot to the right, and multiplication by $u_\lambda$ moves one spot up.  Multiplication by $a_\sigma$ is ``down and to the right'', and multiplication by $a_\lambda$ is ``up and to the right'', though in both cases these ``diagonal'' actions should be regarded as happening in different dimensions than the first two.

\item None of the classes labelled $\Sigma^1(\blank)$ actually exist; rather, these are convenient placeholders to help name other classes.  For example, there are classes $\Sigma^1(\frac{1}{a_\lambda^i u_\lambda})$ for all $i\geq 1$, denoted in the chart by the arrow descending down and to the left from $\Sigma^1(\frac{1}{u_\lambda})$.    
Each of these classes generates a $\Z/4$.  Note that
\[ a_\lambda\cdot \Sigma^1(\tfrac{1}{a_\lambda^2 u_\lambda})=\Sigma^1( \tfrac{1}{a_\lambda u_\lambda})
\]
and similarly for any product where all of the named classes actually exist.  
\item The ``non-existent'' $\Sigma^1$-classes are denoted by $\squar$ in 
 the chart. 
 Some of these are not explicitly named due to space constraints, but are related to the ones shown via division by either $u_\lambda$ or $u_{2\sigma}$.  
\item In terms of degrees, the $\Sigma^1$ adds $1$ to the degree of the class inside the parentheses.  So $\Sigma^1(\frac{1}{u_\lambda})$ (if it existed) would have degree $1+(2-\lambda)=3-\lambda$.
\item The brown numbers written next to certain classes give the fixed point index.  Recall that this index is constant on connected components, and that moving one ``unit'' either down or to the left increases this index by $2$.    
\item The blue classes yield $\Z/4$'s upon division by $a_\lambda$, whereas the pink classes yield $\Z/2$'s.  
\item The class labelled $Z$ is infinitely-divisible by $a_\lambda$, $a_\sigma$, and $u_{2\sigma}$; it is also non-torsion for $a_\lambda$.    All of these classes are drawn slightly out of place in the diagram, as otherwise there are too many classes piling on top of each other.
\item The chart extends infinitely to the left, right, and downwards; but in contrast, multiplication by $u_\lambda$ annihilates everything in the top ``row'' of the diagram.  This includes the classes $\Sigma^1(\frac{1}{u_{2\sigma}^i a_\lambda^j} )$ and all of the classes $\frac{Z}{a_\lambda^i a_\sigma^j u_{2\sigma}^k}$.  
\end{itemize}

\begin{figure}[ht]
\begin{tikzpicture}[scale=2,yscale=1.4]
\draw[] (-0.25,-0.12) node[left,blue] {\tiny $\Sigma^1(1)$};
\draw[] (-1.35,-0.17) node[left,magenta] {$\squar$};
\draw[] (-1,-0.0) node[left,magenta] {\tiny $\Sigma^1 \Bigl ( \frac{1}{u_{2\sigma}}\Bigr )$};
\draw[] (-2.35,-0.17) node[left,magenta] {$\squar$};
\draw[] (-2,-0) node[left,magenta] {\tiny $\Sigma^1 \Bigl ( \frac{1}{u_{2\sigma}^2} \Bigr)$} ;

\draw[] (-0.0,-1.0) node[left,blue] {\tiny $\Sigma^1\Bigl ( \frac{1}{u_\lambda} \Bigr )$};

\draw[] (-0.35,-1.17) node[left,blue] {$\squar$};

\draw[] (-1.35,-1.17) node[left,blue] {$\squar$};
\draw[] (-2.35,-1.17) node[left,blue] {$\squar$};

\draw[] (0.65,-1.17) node[left,magenta] {$\squar$};
\draw[] (1,-1.0) node[left,magenta] {\tiny $\Sigma^1\Bigl ( \frac{u_{2\sigma}}{u_\lambda} \Bigr )$};
\draw[] (1.65,-1.17) node[left,magenta] {$\squar$};
\draw[] (2,-1) node[left,magenta] {\tiny $\Sigma^1\Bigl ( \frac{u_{2\sigma}^2}{u_\lambda} \Bigr )$};

\draw[] (-0.35,-2.17) node[left,blue] {$\squar$};
\draw[] (-1.35,-2.17) node[left,blue] {$\squar$};
\draw[] (-2.35,-2.17) node[left,blue] {$\squar$};
\draw[] (0.65,-2.17) node[left,blue] {$\squar$};
\draw[] (-3.35,-2.17) node[left,blue] {$\squar$};

\draw[] (-3.35,-1.17) node[left,blue] {$\squar$};

\draw[] (1.65,-2.17) node[left,blue] {$\squar$};
\draw[] (2,-2.0) node[left,blue] {\tiny $\Sigma^1 \Bigl ( \frac{u_{2\sigma}^2}{ u_\lambda^2}\Bigr )$};


\draw[->] (-1.5,-0.2) -- (-2,-0.7);
\draw[->] (-2.5,-0.2) -- (-3,-0.7);

\fill (-2.6,-0.3) circle (1pt);
\fill (-2.75,-0.45) circle (1pt);
\fill (-2.9,-0.6) circle (1pt);

\fill (-1.6,-0.3) circle (1pt);
\fill (-1.75,-0.45) circle (1pt);
\fill (-1.9,-0.6) circle (1pt);

\draw[->] (1.5,-1.2) -- (1,-1.7);
\draw[->] (0.5,-1.2) -- (0,-1.7);
\draw[->] (-0.5,-1.2) -- (-1,-1.7);
\draw[->] (-1.5,-1.2) -- (-2,-1.7);
\draw[->] (-2.5,-1.2) -- (-3,-1.7);

\fill (-2.6,-1.3) circle (1pt);
\draw[] (-2.6,-1.3) circle (2pt);
\fill (-2.75,-1.45) circle (1pt);
\draw[] (-2.75,-1.45) circle (2pt);
\fill (-2.9,-1.6) circle (1pt);
\draw[] (-2.9,-1.6) circle (2pt);

\fill (-1.6,-1.3) circle (1pt);
\draw[] (-1.6,-1.3) circle (2pt);
\fill (-1.75,-1.45) circle (1pt);
\draw[] (-1.75,-1.45) circle (2pt);
\fill (-1.9,-1.6) circle (1pt);
\draw[] (-1.9,-1.6) circle (2pt);

\fill (-0.6,-1.3) circle (1pt);
\draw[] (-0.6,-1.3) circle (2pt);
\fill (-0.75,-1.45) circle (1pt);
\draw[] (-0.75,-1.45) circle (2pt);
\fill (-0.9,-1.6) circle (1pt);
\draw[] (-0.9,-1.6) circle (2pt);

\fill (0.4,-1.3) circle (1pt);

\fill (0.25,-1.45) circle (1pt);

\fill (0.1,-1.6) circle (1pt);

\fill (1.4,-1.3) circle (1pt);
\fill (1.25,-1.45) circle (1pt);
\fill (1.1,-1.6) circle (1pt);

\draw[->] (1.5,-2.2) -- (1,-2.7);
\draw[->] (0.5,-2.2) -- (0,-2.7);
\draw[->] (-0.5,-2.2) -- (-1,-2.7);
\draw[->] (-1.5,-2.2) -- (-2,-2.7);
\draw[->] (-2.5,-2.2) -- (-3,-2.7);

\fill (-2.6,-2.3) circle (1pt);
\draw[] (-2.6,-2.3) circle (2pt);
\fill (-2.75,-2.45) circle (1pt);
\draw[] (-2.75,-2.45) circle (2pt);
\fill (-2.9,-2.6) circle (1pt);
\draw[] (-2.9,-2.6) circle (2pt);

\fill (-1.6,-2.3) circle (1pt);
\draw[] (-1.6,-2.3) circle (2pt);
\fill (-1.75,-2.45) circle (1pt);
\draw[] (-1.75,-2.45) circle (2pt);
\fill (-1.9,-2.6) circle (1pt);
\draw[] (-1.9,-2.6) circle (2pt);

\fill (-0.6,-2.3) circle (1pt);
\draw[] (-0.6,-2.3) circle (2pt);
\fill (-0.75,-2.45) circle (1pt);
\draw[] (-0.75,-2.45) circle (2pt);
\fill (-0.9,-2.6) circle (1pt);
\draw[] (-0.9,-2.6) circle (2pt);

\fill (0.4,-2.3) circle (1pt);
\draw[] (0.4,-2.3) circle (2pt);
\fill (0.25,-2.45) circle (1pt);
\draw[] (0.25,-2.45) circle (2pt);
\fill (0.1,-2.6) circle (1pt);
\draw[] (0.1,-2.6) circle (2pt);

\fill (1.4,-2.3) circle (1pt);
\draw[] (1.4,-2.3) circle (2pt);
\fill (1.25,-2.45) circle (1pt);
\draw[] (1.25,-2.45) circle (2pt);
\fill (1.1,-2.6) circle (1pt);
\draw[] (1.1,-2.6) circle (2pt);


\draw[gray] (-0.55,-1.35)--(0,-1.75);

\draw[gray] (-0.7,-1.5)--(0.40,-2.3);
\draw[gray] (-0.85,-1.65)--(0.25,-2.45);
\draw[gray] (-0.3,-2.3)--(0.07,-2.55);

\fill[cyan] (-0.15,-1.9) circle (1pt);
\fill[cyan] (-0.3,-2.05) circle (1pt);
\fill[cyan] (0,-1.75) circle (1pt);
\draw[->,cyan] (0,-1.75)--(-0.39,-2.15);

\draw[gray] (-1.7,-1.5)--(-0.60,-2.3);
\draw[gray] (-1.85,-1.65)--(-0.75,-2.45);
\draw[gray] (-1.3,-2.3)--(-0.93,-2.55);

\fill[cyan] (-1.15,-1.9) circle (1pt);
\fill[cyan] (-1.3,-2.05) circle (1pt);

\draw[gray] (-1.55,-1.35)--(-1,-1.75);
\fill[cyan] (-1,-1.75) circle (1pt);
\draw[->,cyan] (-1,-1.75)--(-1.39,-2.15);

\draw[gray] (-2.7,-1.5)--(-1.60,-2.3);
\draw[gray] (-2.85,-1.65)--(-1.75,-2.45);
\draw[gray] (-2.3,-2.3)--(-1.93,-2.55);

\fill[cyan] (-2.15,-1.9) circle (1pt);
\fill[cyan] (-2.3,-2.05) circle (1pt);
\fill[cyan] (-2,-1.75) circle (1pt);

\draw[gray] (-2.55,-1.35)--(-2,-1.75);
\draw[->,cyan] (-2,-1.75)--(-2.39,-2.15);

\draw[gray] (-2.55,-1.35)--(-3.55,-0.6);
\draw[gray] (-2.7,-1.5)--(-3.58,-0.84);
\draw[gray] (-2.82,-1.65) -- (-3.57,-1.1);
\draw[gray] (-3,-0.75)--(-3.5,-0.4);

\fill[cyan] (-3,-0.75) circle (1pt);
\fill[cyan] (-3.15,-0.9) circle (1pt);
\fill[cyan] (-3.3,-1.05) circle (1pt);
\fill[cyan] (-3.45,-1.2) circle (1pt);
\draw[->,cyan] (-3,-0.75)--(-3.59,-1.35);

\fill[cyan] (-3,-1.75) circle (1pt);
\fill[cyan] (-3.15,-1.9) circle (1pt);
\fill[cyan] (-3.3,-2.05) circle (1pt);
\fill[cyan] (-3.45,-2.2) circle (1pt);
\draw[->,cyan] (-3,-1.75)--(-3.59,-2.35);

\draw[gray] (-2.65,-2.25)--(-3.6,-1.58);
\draw[gray] (-2.82,-2.41)--(-3.6,-1.83);
\draw[gray] (-2.96,-2.57)--(-3.6,-2.08);
\draw[gray] (-3,-1.75)--(-3.6,-1.3);

\draw[gray] (0.45,-1.35)--(1,-1.75);

\draw[gray] (0.3,-1.5)--(1.40,-2.3);
\draw[gray] (0.15,-1.65)--(1.25,-2.45);
\draw[gray] (0.7,-2.3)--(1.07,-2.55);

\fill[cyan] (0.85,-1.9) circle (1pt);
\fill[cyan] (0.7,-2.05) circle (1pt);
\fill[cyan] (1,-1.75) circle (1pt);
\draw[->,cyan] (1,-1.75)--(0.61,-2.15);

\draw[gray] (-2.7,-0.5)--(-1.60,-1.3);
\draw[gray] (-2.85,-0.65)--(-1.75,-1.45);
\draw[gray] (-2.3,-1.3)--(-1.93,-1.55);

\fill[cyan] (-2.15,-0.9) circle (1pt);
\fill[cyan] (-2.3,-1.05) circle (1pt);

\draw[gray] (-2.55,-0.35)--(-2,-0.75);
\fill[cyan] (-2,-0.75) circle (1pt);
\draw[->,cyan] (-2,-0.75)--(-2.39,-1.15);

\draw[gray] (-1.7,-0.5)--(-0.60,-1.3);
\draw[gray] (-1.85,-0.65)--(-0.75,-1.45);
\draw[gray] (-1.3,-1.3)--(-0.93,-1.55);

\fill[cyan] (-1.15,-0.9) circle (1pt);
\fill[cyan] (-1.3,-1.05) circle (1pt);

\draw[gray] (-1.55,-0.35)--(-1,-0.75);
\fill[cyan] (-1,-0.75) circle (1pt);
\draw[->,cyan] (-1,-0.75)--(-1.39,-1.15);

\draw[] (-1.5,0.28) node[right] {\small $Z=\Sigma^1(\frac{1}{a_\sigma u_{2\sigma}})$};
\draw[->] (-1.5,0.28) -- (-1.95,0.285);
\fill (-2,0.3) circle (1pt);
\fill (-1.887,0.367) circle (1pt);
\fill (-1.774,0.434) circle (1pt);

\fill (-2.113,0.223) circle (1pt);
\fill (-2.226,0.156) circle (1pt);

\draw[] (-2.113,0.223) -- (-1.6,-0.3);
\draw[] (-2.226,0.156) -- (-1.76,-0.42);
\draw[] (-2.1,-0.25) -- (-1.92,-0.57);

\draw[->] (-2.3,0.1) -- (-1.7,0.5);
\draw[->] (-1.7,0.5) -- (-2.3,0.1);

\fill (-2.1,0.4) circle (1pt);
\fill (-1.987,0.467) circle (1pt);
\fill (-1.874,0.534) circle (1pt);

\fill (-2.213,0.323) circle (1pt);
\fill (-2.326,0.256) circle (1pt);

\draw[->] (-2.4,0.2) -- (-1.8,0.6);
\draw[->] (-1.8,0.6) -- (-2.4,0.2);

\draw[->] (-2.226,0.156) -- (-2.426,0.356);
\draw[->] (-2.113,0.223) -- (-2.313,0.423);
\draw[->] (-2.0,0.3) -- (-2.2,0.5);
\draw[->] (-1.887,0.367) -- (-2.087,0.567);
\draw[->] (-1.774,0.434) -- (-1.974,0.634);

\fill (-3,0.3) circle (1pt);
\draw[->] (-3.3,0.1) -- (-2.7,0.5);
\draw[->] (-2.7,0.5) -- (-3.3,0.1);

\fill (-2.887,0.367) circle (1pt);
\fill (-2.774,0.434) circle (1pt);

\fill (-3.113,0.223) circle (1pt);
\fill (-3.226,0.156) circle (1pt);

\draw[] (-3.113,0.223) -- (-2.6,-0.3);
\draw[] (-3.226,0.156) -- (-2.76,-0.42);
\draw[] (-3.1,-0.25) -- (-2.92,-0.57);

\draw[->] (-2.3,0.1) -- (-1.7,0.5);
\draw[->] (-1.7,0.5) -- (-2.3,0.1);

\fill (-3.1,0.4) circle (1pt);
\fill (-2.987,0.467) circle (1pt);
\fill (-2.874,0.534) circle (1pt);

\fill (-3.213,0.323) circle (1pt);
\fill (-3.326,0.256) circle (1pt);

\draw[->] (-3.4,0.2) -- (-2.8,0.6);
\draw[->] (-2.8,0.6) -- (-3.4,0.2);

\draw[->] (-3.226,0.156) -- (-3.426,0.356);
\draw[->] (-3.113,0.223) -- (-3.313,0.423);
\draw[->] (-3.0,0.3) -- (-3.2,0.5);
\draw[->] (-2.887,0.367) -- (-3.087,0.567);
\draw[->] (-2.774,0.434) -- (-2.974,0.634);

\draw[gray] (-2.6,-2.3)--(-2.1,-2.62);
\draw[gray] (-2.7,-2.48)--(-2.3,-2.75);
\draw[gray] (-2.85,-2.63)--(-2.6,-2.8);
\draw[gray] (-1.6,-2.3)--(-1.1,-2.62);
\draw[gray] (-1.7,-2.48)--(-1.3,-2.75);
\draw[gray] (-1.85,-2.63)--(-1.6,-2.8);
\draw[gray] (-0.6,-2.3)--(-0.1,-2.62);
\draw[gray] (-0.7,-2.48)--(-0.3,-2.75);
\draw[gray] (-0.85,-2.63)--(-0.6,-2.8);
\draw[gray] (0.4,-2.3)--(0.9,-2.62);
\draw[gray] (0.3,-2.48)--(0.7,-2.75);
\draw[gray] (0.15,-2.63)--(0.4,-2.8);
\draw[gray] (1.4,-2.3)--(1.9,-2.62);
\draw[gray] (1.3,-2.48)--(1.7,-2.75);
\draw[gray] (1.15,-2.63)--(1.4,-2.8);
\draw[gray] (1.4,-1.3)--(1.9,-1.62);
\draw[gray] (1.3,-1.48)--(1.7,-1.75);
\draw[gray] (1.15,-1.63)--(1.4,-1.8);
\draw[] (-1.5,-0.25) node[right,brown] {\tiny $3$};
\draw[] (-2.5,-0.25) node[right,brown] {\tiny $5$};
\draw[] (-0.5,-1.25) node[right,brown] {\tiny $3$};
\draw[] (0.5,-1.25) node[right,brown] {\tiny $1$};
\draw[] (1.5,-1.25) node[right,brown] {\tiny $-1$};

\draw[] (-1.97,0.27)--(-1.52,-0.15);
\draw[] (-2.97,0.27)--(-2.52,-0.15);

\end{tikzpicture}
\caption{The module $\M_{odd}$}
\label{fig:C4-two}
\end{figure}
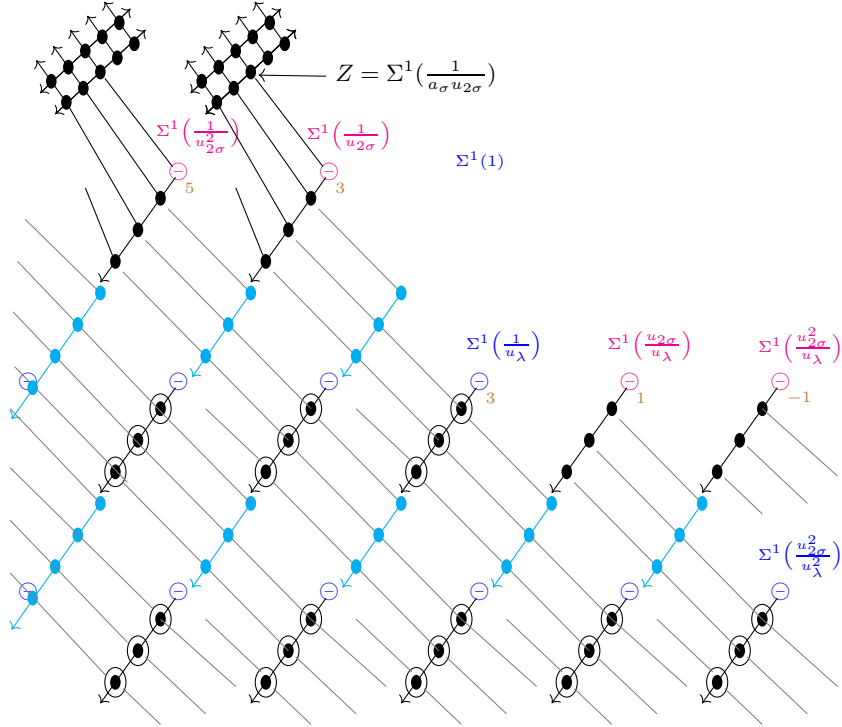

As a $\Z[u_{2\sigma},u_\lambda,a_\lambda,a_\sigma]$-module $\M$ is the direct sum of $\M_{ev}$ and $\M_{odd}$.  So the vanishing condition---which is a property of multiplication by $a_\lambda^ia_\sigma^j$---can be verified by thinking about the two pieces separately.  Observe the following properties of these pieces:
\begin{enumerate}[(1)]
\item For degrees in $RO(C_4)$ where both the $\lambda$-index and $\sigma$-index are negative, $\M_{ev}$ is zero EXCEPT for the classes $\frac{??}{u_{2\sigma}^ru_\lambda^s}$, which are all annihilated by $a_\sigma$ and occur in degrees having rank zero.   
\item $\M_{odd}$ is nonzero only in degrees where at least one of the $\lambda$-index and $\sigma$-index is negative.  
\end{enumerate}

Let $V=i\lambda+j\sigma$, where $i,j\geq 0$.  We assume $2i+j\neq 2$ to avoid the $\dim V=2$ anomaly. Recall that the vanishing requirement is that for each $\ell\in \Z$ the two maps
\[ \xymatrixcolsep{3.5pc}\xymatrix{? \ar[r]^-{a_\lambda^ia_\sigma^j} & H^{(\ell+1)(i\lambda+j\sigma)-\ell}(\pt) \ar[r]^-{a_\lambda^ia_\sigma^j} & ??
}
\]
are surjective and injective, respectively.   We have not written explicit groups in the spots with question marks because the degrees of those group can easily be determined from the other information and are largely irrelevant.  Note that when $\ell$ is odd the middle group is in $\M_{odd}$, and when $\ell$ is even it is in $\M_{ev}$.  Properties (1) and (2) above then imply that
\[ \text{the middle group is zero in the cases} \begin{cases}
\text{$\ell$ odd and positive},\\
\text{$\ell$ even and negative.}
\end{cases}
\]
So in these cases the vanishing requirement is trivially satisfied.

We analyze the remaining cases separately.

\smallskip

\noindent
\underline{Case 1: } $\ell$ is even and $\ell\geq 0$.

\smallskip
Write $\ell=2k$, with $k\geq 0$.  Our middle group is then in degree $(2k+1)(i\lambda+j\sigma)-2k$.  We will locate this degree on the chart for $\M_{ev}$ and then use the chart to analyze the injectivity/surjectivity of the maps.  The analysis is slightly different in the $j=0$ and $j>0$ cases.

Suppose $j=0$, so that $i>0$.  The degree in question is then
\[ (2k+1)i\lambda-2k = k(\lambda-2)+[ (2k+1)i-k]\lambda = k(\lambda-2)+[i+k(2i-1)]\lambda.
\]
On the chart, we get into this degree by starting with $u_\lambda^k$ and then multiplying by $a_\lambda^{i+k(2i-1)}$, noting that the exponent of $a_\lambda$ is at least $i$.  This puts us on the ray of $\Z/4$'s emanating from $u_\lambda^k$, and we need to consider the multiplication-by-$a_\lambda^i$ maps going in and out of this degree.  The map going out is clearly an isomorphism (hence injective), and the map coming in is a surjection because we had $i+k(2i-1)\geq i$ (that is, we have multiplied by at least $i$ of the $a_\lambda$ classes in the course of moving to our spot on the chart).  

Now consider the case $j>0$.  Here the degree is
\[ (2k+1)(i\lambda+j\sigma)-2k=k(2\sigma-2)+(2k+1)i\lambda + [j+2k(j-1)]\sigma,
\]
and we get into this degree by starting
with the class $u_{2\sigma}^k$ and then multiplying by $a_\lambda^{(2k+1)i}$ and $a_\sigma^{j+2k(j-1)}$.  Note that the first exponent is at least $i$ and the second is at least $j$ (hence positive).    This puts us at one of the $\Z/2$'s in the diagram, and we need to consider the multiplication-by-$a_\lambda^ia_\sigma^j$ maps going into and out of it.  The map going out is clearly an isomorphism, hence injective.  The map coming in is surjective because our exponents satisfied $(2k+1)i\geq i$ and $j+2k(j-1)\geq j$.

\smallskip

\noindent
\underline{Case 2: } $\ell$ is odd and $\ell \leq -1$.

Note that $\ell=-1$ is trivial since the middle group is $H^1(\pt)=0$, so we may assume $\ell=-2k-1$ where $k\geq 1$.  Then the middle group lies in degree $(2k+1)-2k(i\lambda+j\sigma)$.  We will again locate this group on the chart of $\M_{odd}$ and then analyze the incoming and outgoing multiplication maps.  

Let use first consider the $j=0$ situation, so that $i>1$ and
we are looking at degree $(2k+1)-2ki\lambda$.  To find this degree on the chart, first find the class $\Sigma^1(\frac{1}{u_\lambda^k})$, which has degree $(2k+1)-k\lambda$, and then divide by $a_\lambda^{2ki-k}$ (note that $2ki-k\geq 1$ since $i\geq 1$ and $k\geq 1$).  This is one of the $\Z/4$'s in the diagram, and we want to look at multiplication by $a_\lambda^i$ going in and out of this group.  Going in is obviously an isomorphism, hence surjective.  Going out will also be an isomorphism as long as we don't go too far so that we are in the range where the groups become zero.  But we have divided by $2ki-k$ of the classes $a_\lambda$, and now we are multipliying by $i$ of them, so we will be fine as long as $2ki-k>i$.  This is readily checked to be true under our hypotheses that $k\geq 1$, $i>1$.  

Next assume $j>0$.  Locate the class $\Sigma^1(\frac{1}{u_{2\sigma}^k})$ on the chart, which has degree $(2k+1)-2k\sigma$.  Divide by $a_\lambda^{2ki}$, and then further divide by $a_\sigma^{2kj-2k}$ (note that $2kj-2k\geq 0$).  If $j=1$ this last step does nothing, whereas if $j>1$ it puts us up into the regions at the top of the chart (the ones containing the classes $\frac{1}{u_{2\sigma}^r}Z$).  Regardless, we find a $\Z/2$ in this degree and need to analyze the maps $a_\lambda^ia_\sigma^j$ going in and out of this spot.  It is apparent at once that the incoming map is surjective---i.e., we can take our class and divide by $a_\lambda$ and $a_\sigma$ any number of times.  For the outgoing map, we need to make sure our class is not annihilated by $a_\lambda^i a_\sigma^j$.  But note that multiplying by $a_\lambda^{2ki-1}$ gives something nonzero by construction, and one readily checks that either $i=0$ or $2ki-1\geq i$; so multiplying by $a_\lambda^i$ does not annihilate our class.  Likewise, multiplication by $a_\sigma^{2kj-2k+1}$ is definitely nonzero (we divided by $a_\sigma^{2kj-2k}$, but the classes $\frac{1}{a_\lambda^r}\Sigma^1(\frac{1}{u_{2\sigma}^s})$ all admit at least one nonzero $a_\sigma$ multiplication).  One again checks that $2kj-2k+1\geq j$, so we are okay.  

\smallskip

\subsection{The case \mdfn{$G=\Sigma_3$}}
For this group the calculations are taken from \cite{Y1}.
For our ordered basis of $RO(\Sigma_3)$ we take $1$, $\lambda$, $\sigma$, where $\sigma$ is the sign representation and $\lambda$ is the action on $\R^2$ via the symmetries of an equilateral triangle.  The Euler classes $a_\lambda$ and $a_\sigma$ have orders $3$ and $2$, respectively. Note that the product $a_\lambda a_\sigma$ must then be zero, which simplifies some things as we will have far fewer classes to manage.

The representations $2\sigma$ and $\lambda+\sigma$ are orientable, so we have orientation classes $u_{2\sigma}$ and $u_{\lambda+\sigma}$.  It turns out that the class $\frac{u_{\lambda+\sigma}}{u_{2\sigma}}$ also exists; we denote this as $u_{\lambda-\sigma}$.  Note that we do not need to consider the class $u_{2\lambda}$, even though $2\lambda$ is orientable, because it arises as the product $u_{2\lambda}=u_{\lambda-\sigma}^2u_{2\sigma}$.  

As a $\Z[u_{2\sigma},u_\ls,a_\sigma,a_\lambda]$-module it is again true that $\M$ splits into two pieces $\M_0$ and $\M_1$, but those two pieces are no longer characterized by the parity of the fixed point rank.  We will only describe these pieces empirically.  The chart for $\M_0$ is shown in Figure~\ref{fig:Sigma_3-one}.
To read this chart just note that:
\begin{itemize}
\item Triangles represent $\Z/3$'s and dots represent $\Z/2$'s.  The red classes generate copies of $\Z$.
\item Multiplication by $a_\lambda$ moves up and to the right, whereas multiplication by $a_\sigma$ moves down and to the right; both are drawn with blue lines. As in previous charts, this is a projection of a multi-dimensional figure onto a plane and so blue rays that look like they intersect usually don't.  
\item The fixed-point index is constant on connected components of the chart, and is indicated in gray square brackets.  Moving down in the diagram increases this index by $1$, and moving left increases it by $2$.
\item A ``knight's move'' of right-one/down-two preserves the fixed-point index, and there are instances where two seemingly-different components run into each other.  For example, $a_\sigma^4 u_{\ls}^2$ and $a_\lambda^2 u_{2\sigma}$ are both in degree $-2+2\lambda+2\sigma$.  The first class generates a $\Z/2$ and the second a $\Z/3$, so these combine to give a $\Z/6$ in that degree.  This phenomenon only happens in the upper right quadrant of the chart.
\item The dashed lines separate the chart into four quadrants.  Observe that everything in the lower-right has negative $\lambda$-index and positive $\sigma$-index, whereas everything in the upper left has positive $\lambda$-index and negative $\sigma$-index.  Thus, nothing in these two regions lies in the regular portion of $\M$. The other two quadrants have a mixed collection of indices, from this particular perspective.
\end{itemize}

\begin{figure}[ht]
\begin{tikzpicture}[scale=2,yscale=1]
\draw[dashed] (-0.5,-2.5)--(-0.5,1.7);
\draw[dashed] (-2.5,-0.6)--(2.5,-0.6);
\draw[] (-0.05,0) node[left,red] {\small $1$};
\draw[] (1.05,0.05) node[left,red] {\small $u_{2\sigma}$};
\draw[] (2.05,0.05) node[left,red] {\small $u_{2\sigma}^2$};
\draw[] (-1,0) node[left,red] {\small $\frac{2}{u_{2\sigma}}$};
\draw[] (-2,0) node[left,red] {\small $\frac{2}{u_{2\sigma}^2}$};

\draw[] (-0.05,-0.15) node[left,gray] {\tiny $[0]$};
\draw[] (1.05,-0.15) node[left,gray] {\tiny $[-2]$};
\draw[] (-0.05,0.85) node[left,gray] {\tiny $[-1]$};
\draw[] (-1.05,0.85) node[left,gray] {\tiny $[1]$};
\draw[] (-1.05,-0.2) node[left,gray] {\tiny $[2]$};
\draw[] (-0.05,-1.2) node[left,gray] {\tiny $[1]$};

\draw[] (0,1.05) node[left,red] {\small $u_{\ls}$};
\draw[] (1,1.05) node[left,red] {\small $u_{2\sigma}u_{\ls}$};
\draw[] (2.2,1.08) node[left,red] {\small $u_{2\sigma}^2u_{\ls}$};
\draw[] (-1,1.05) node[left,red] {\small $\frac{2u_{\ls}}{u_{2\sigma}}$};
\draw[] (-2,1.05) node[left,red] {\small $\frac{2u_{\ls}}{u_{2\sigma}^2}$};

\draw[] (0,-1) node[left,red] {\small $\frac{3}{u_{\ls}}$};
\draw[] (1,-1) node[left,red] {\small $\frac{3u_{2\sigma}}{u_{\ls}}$};
\draw[] (2,-1) node[left,red] {\small $\frac{3u_{2\sigma}^2}{u_{\ls}}$};
\draw[] (-1,-1) node[left,red] {\small $\frac{6}{u_{2\sigma}u_{\ls}}$};
\draw[] (-2,-1) node[left,red] {\small $\frac{6}{u_{2\sigma}^2u_{\ls}}$};

\draw[] (0,-2) node[left,red] {\small $\frac{3}{u_\ls^2}$};
\draw[] (1,-2) node[left,red] {\small $\frac{3u_{2\sigma}}{u_\ls^2}$};
\draw[] (2,-2) node[left,red] {\small $\frac{3u_{2\sigma}^2}{u_\ls^2}$};
\draw[] (-1,-2) node[left,red] {\small $\frac{6}{u_{2\sigma}u_\ls^2}$};
\draw[] (-2,-2) node[left,red] {\small $\frac{6}{u_{2\sigma}^2u_\ls^2}$};


\draw[]  (-0.978,0.23) node {\small $\triangle$};
\draw[]  (-0.87,0.36) node {\small $\triangle$};
\draw[->,blue,thick] (-1.1,0.1) --(-0.75,0.5);

\draw[]  (-1.978,0.23) node {\small $\triangle$};
\draw[]  (-1.87,0.36) node {\small $\triangle$};
\draw[->,blue,thick] (-2.1,0.1) --(-1.75,0.5);

\draw[]  (-0.978,1.23) node {\small $\triangle$};
\draw[]  (-0.87,1.36) node {\small $\triangle$};
\draw[->,blue,thick] (-1.1,1.1) --(-0.75,1.5);

\draw[]  (-1.978,1.23) node {\small $\triangle$};
\draw[]  (-1.87,1.36) node {\small $\triangle$};
\draw[->,blue,thick] (-2.1,1.1) --(-1.75,1.5);

\draw[->,blue,thick] (-0.09,0)--(0.5,-0.38);
\fill (0.15,-0.15) circle (1pt);
\fill (0.37,-0.3) circle (1pt);
\draw[]  (0.022,0.23) node {\small $\triangle$};
\draw[]  (0.13,0.36) node {\small $\triangle$};
\draw[->,blue,thick] (-0.1,0.1) --(0.25,0.5);

\draw[->,blue,thick] (0.91,0)--(1.5,-0.38);
\fill (1.15,-0.15) circle (1pt);
\fill (1.37,-0.3) circle (1pt);
\draw[]  (1.022,0.23) node {\small $\triangle$};
\draw[]  (1.13,0.36) node {\small $\triangle$};
\draw[->,blue,thick] (0.9,0.1) --(1.25,0.5);

\draw[->,blue,thick] (1.91,0)--(2.5,-0.38);
\fill (2.15,-0.15) circle (1pt);
\fill (2.37,-0.3) circle (1pt);
\draw[]  (2.022,0.23) node {\small $\triangle$};
\draw[]  (2.13,0.36) node {\small $\triangle$};
\draw[->,blue,thick] (1.9,0.1) --(2.25,0.5);


\draw[->,blue,thick] (-0.09,1)--(0.5,0.62);
\fill (0.15,0.85) circle (1pt);
\fill (0.37,0.7) circle (1pt);
\draw[]  (0.022,1.23) node {\small $\triangle$};
\draw[]  (0.13,1.36) node {\small $\triangle$};
\draw[->,blue,thick] (-0.1,1.1) --(0.25,1.5);

\draw[->,blue,thick] (0.91,1)--(1.5,0.62);
\fill (1.15,0.85) circle (1pt);
\fill (1.37,0.7) circle (1pt);
\draw[]  (1.022,1.23) node {\small $\triangle$};
\draw[]  (1.13,1.36) node {\small $\triangle$};
\draw[->,blue,thick] (0.9,1.1) --(1.25,1.5);

\draw[->,blue,thick] (1.91,1)--(2.5,0.62);
\fill (2.15,0.85) circle (1pt);
\fill (2.37,0.7) circle (1pt);
\draw[]  (2.022,1.23) node {\small $\triangle$};
\draw[]  (2.13,1.36) node {\small $\triangle$};
\draw[->,blue,thick] (1.9,1.1) --(2.25,1.5);


\draw[->,blue,thick] (-0.09,-1)--(0.5,-1.38);
\fill (0.15,-1.15) circle (1pt);
\fill (0.37,-1.3) circle (1pt);
\draw[->,blue,thick] (0.91,-1)--(1.5,-1.38);
\fill (1.15,-1.15) circle (1pt);
\fill (1.37,-1.3) circle (1pt);
\draw[->,blue,thick] (1.91,-1)--(2.5,-1.38);
\fill (2.15,-1.15) circle (1pt);
\fill (2.37,-1.3) circle (1pt);

\draw[->,blue,thick] (-0.09,-2)--(0.5,-2.38);
\fill (0.15,-2.15) circle (1pt);
\fill (0.37,-2.3) circle (1pt);
\draw[->,blue,thick] (0.91,-2)--(1.5,-2.38);
\fill (1.15,-2.15) circle (1pt);
\fill (1.37,-2.3) circle (1pt);
\draw[->,blue,thick] (1.91,-2)--(2.5,-2.38);
\fill (2.15,-2.15) circle (1pt);
\fill (2.37,-2.3) circle (1pt);
\end{tikzpicture}
\caption{The module $\M_{0}$}
\label{fig:Sigma_3-one}
\end{figure}
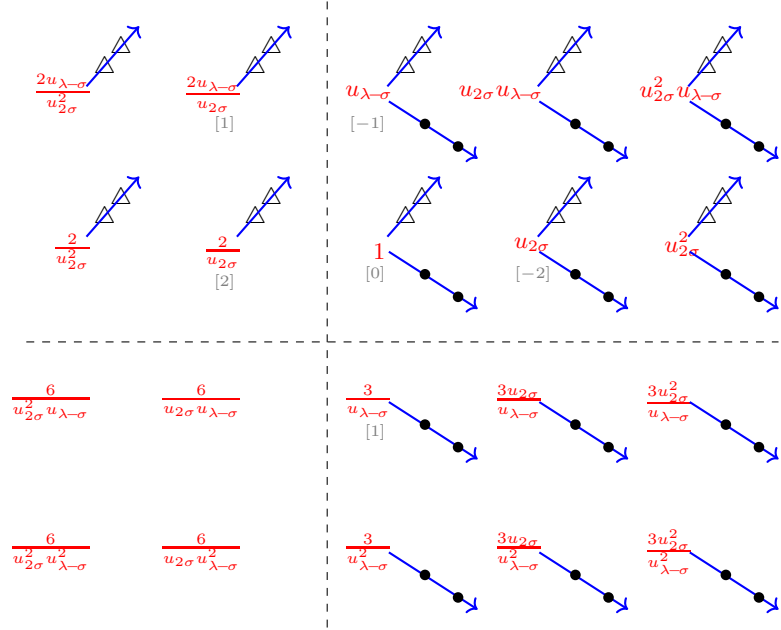

The chart for $\M_1$ is shown in Figure~\ref{fig:Sigma_3-two}.  The main points to keep in mind are
\begin{itemize}
\item As in the previous chart, triangles represent $\Z/3$'s and dots represent $\Z/2$'s.  The $a_\sigma$ multiplications are down-and-to-the-right, and the $a_\lambda$ multiplications are up-and-to-the-right.  
\item The red classes are just placeholders, and don't actually exist in $\M$ (or said differently, they are all zero).  
\item Every nonzero element in this chart is infinitely-divisible by either $a_\lambda$ or $a_\sigma$.
\item The fixed-point index (shown in gray brackets) behaves similarly to Figure~\ref{fig:Sigma_3-one}, and there is again the issue---this time in the lower left quadrant of the diagram---of the $a_\sigma$- and $a_\lambda$-rays sometimes running into each other to create a $\Z/6$.
\item Observe that every class in $\M_1$ has a negative $\lambda$-index or a negative $\sigma$-index.  
\end{itemize}

\begin{figure}[ht]
\begin{tikzpicture}[scale=2]
\draw[] (-0.05,0) node[left,red] {\tiny $\Sigma^1(1)$};
\draw[] (-1,0) node[left,red] {\tiny $\Sigma^1(\frac{1}{u_{2\sigma}})$};
\draw[] (-2,0) node[left,red] {\tiny $\Sigma^1(\frac{1}{u_{2\sigma}^2})$};

\draw[] (-0.08,-0.15) node[left,gray] {\tiny $[1]$};
\draw[] (-0.08,-0.82) node[left,gray] {\tiny $[2]$};
\draw[] (-1.08,-0.17) node[left,gray] {\tiny $[3]$};
\draw[] (-1.08,0.85) node[left,gray] {\tiny $[2]$};

\draw[] (-1,1.05) node[left,red] {\tiny $\Sigma^1(\frac{u_{\ls}}{u_{2\sigma}})$};
\draw[] (-2,1.05) node[left,red] {\tiny $\Sigma^1(\frac{u_{\ls}}{u_{2\sigma}^2})$};

\draw[] (0,-1) node[left,red] {\tiny  $\Sigma^1(\frac{1}{u_{\ls}})$};
\draw[] (1,-1) node[left,red] {\tiny $\Sigma^1( \frac{u_{2\sigma}}{u_{\ls}})$};
\draw[] (2,-1) node[left,red] {\tiny $\Sigma^1(\frac{u_{2\sigma}^2}{u_{\ls}})$};
\draw[] (-1,-1) node[left,red] {\tiny $\Sigma^1(\frac{1}{u_{2\sigma}u_{\ls}})$};
\draw[] (-2,-1) node[left,red] {\tiny $\Sigma^1(\frac{1}{u_{2\sigma}^2u_{\ls}})$};

\draw[] (0,-2) node[left,red] {\tiny $\Sigma^1(\frac{1}{u_\ls^2})$};
\draw[] (1,-2) node[left,red] {\tiny $\Sigma^1(\frac{u_{2\sigma}}{u_\ls^2})$};
\draw[] (2,-2) node[left,red] {\tiny $\Sigma^1(\frac{u_{2\sigma}^2}{u_\ls^2})$};
\draw[] (-1,-2) node[left,red] {\tiny $\Sigma^1(\frac{1}{u_{2\sigma}u_\ls^2})$};
\draw[] (-2,-2) node[left,red] {\tiny $\Sigma^1(\frac{1}{u_{2\sigma}^2u_\ls^2})$};


\draw[->,blue,thick] (-2.39,-1.1)--(-2.68,-1.48);
\draw[] (-2.46,-1.18) node[] {\small $\triangle$};
\draw[] (-2.55,-1.32) node[] {\small $\triangle$};

\draw[->,blue,thick] (-1.39,-1.1)--(-1.68,-1.48);
\draw[] (-1.46,-1.18) node[] {\small $\triangle$};
\draw[] (-1.55,-1.32) node[] {\small $\triangle$};

\draw[->,blue,thick] (-0.39,-1.1)--(-0.68,-1.48);
\draw[] (-0.46,-1.18) node[] {\small $\triangle$};
\draw[] (-0.55,-1.32) node[] {\small $\triangle$};

\draw[->,blue,thick] (0.61,-1.1)--(0.32,-1.48);
\draw[] (0.54,-1.18) node[] {\small $\triangle$};
\draw[] (0.45,-1.32) node[] {\small $\triangle$};

\draw[->,blue,thick] (1.61,-1.1)--(1.32,-1.48);
\draw[] (1.54,-1.18) node[] {\small $\triangle$};
\draw[] (1.45,-1.32) node[] {\small $\triangle$};


\draw[->,blue,thick] (-2.39,-2.1)--(-2.68,-2.48);
\draw[] (-2.46,-2.18) node[] {\small $\triangle$};
\draw[] (-2.55,-2.32) node[] {\small $\triangle$};

\draw[->,blue,thick] (-1.39,-2.1)--(-1.68,-2.48);
\draw[] (-1.46,-2.18) node[] {\small $\triangle$};
\draw[] (-1.55,-2.32) node[] {\small $\triangle$};

\draw[->,blue,thick] (-0.39,-2.1)--(-0.68,-2.48);
\draw[] (-0.46,-2.18) node[] {\small $\triangle$};
\draw[] (-0.55,-2.32) node[] {\small $\triangle$};

\draw[->,blue,thick] (0.61,-2.1)--(0.32,-2.48);
\draw[] (0.54,-2.18) node[] {\small $\triangle$};
\draw[] (0.45,-2.32) node[] {\small $\triangle$};

\draw[->,blue,thick] (1.61,-2.1)--(1.32,-2.48);
\draw[] (1.54,-2.18) node[] {\small $\triangle$};
\draw[] (1.45,-2.32) node[] {\small $\triangle$};


\draw[->,blue,thick] (-1.3,-1.9)--(-1.8,-1.65);
\fill (-1.44,-1.83) circle(1pt);
\fill (-1.635,-1.73) circle(1pt);

\draw[->,blue,thick] (-2.3,-1.9)--(-2.8,-1.65);
\fill (-2.44,-1.83) circle(1pt);
\fill (-2.635,-1.73) circle(1pt);

\draw[->,blue,thick] (-1.3,-0.9)--(-1.8,-0.65);
\fill (-1.44,-0.83) circle(1pt);
\fill (-1.635,-0.73) circle(1pt);

\draw[->,blue,thick] (-2.3,-0.9)--(-2.8,-0.65);
\fill (-2.44,-0.83) circle(1pt);
\fill (-2.635,-0.73) circle(1pt);

\draw[->,blue,thick] (-1.3,0.1)--(-1.8,0.35);
\fill (-1.44,0.17) circle(1pt);
\fill (-1.635,0.27) circle(1pt);

\draw[->,blue,thick] (-2.3,0.1)--(-2.8,0.35);
\fill (-2.44,0.17) circle(1pt);
\fill (-2.635,0.27) circle(1pt);

\draw[->,blue,thick] (-1.3,1.2)--(-1.8,1.45);
\fill (-1.44,1.27) circle(1pt);
\fill (-1.635,1.37) circle(1pt);

\draw[->,blue,thick] (-2.3,1.2)--(-2.8,1.45);
\fill (-2.44,1.27) circle(1pt);
\fill (-2.635,1.37) circle(1pt);

\draw[dashed] (-0.8,-2.5)--(-0.8,1.7);
\draw[dashed] (-2.9,-0.4)--(2.3,-0.4);
\end{tikzpicture}
\caption{The module $\M_{1}$}
\label{fig:Sigma_3-two}
\end{figure}
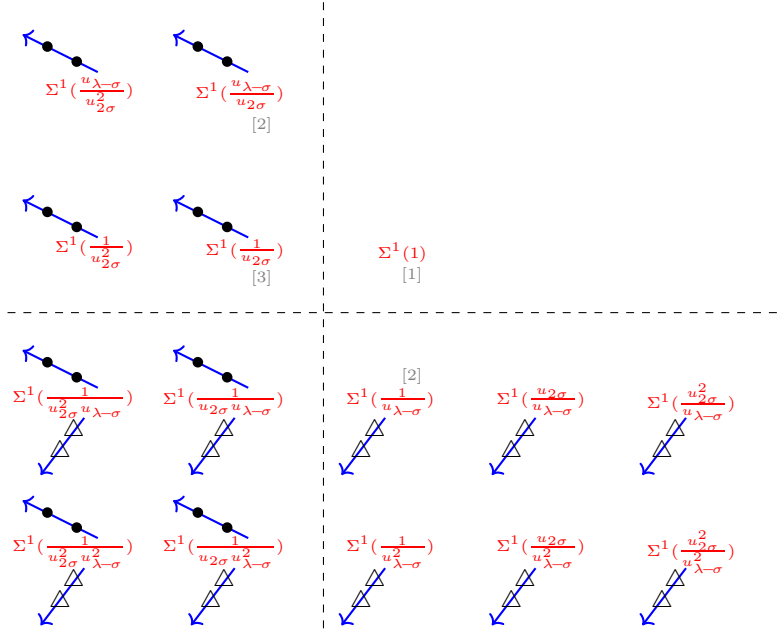

Note that the nonzero elements in the first quadrant of the $\M_0$ chart are sums and multiples of elements of the form 
\begin{enumerate}[(i)]
\item $u_{\lambda-\sigma}^Pu_{2\sigma}^Qa_\sigma^R$, lying in degree
$P\lambda+(2Q+R-P)\sigma-(P+2Q)$, and
\item $u_{\lambda-\sigma}^Pu_{2\sigma}^Qa_\lambda^R$, lying in degree
$(P+R)\lambda +(2Q-P)\sigma-(P+2Q)$.
\end{enumerate}
Since $P,Q,R\geq 0$ we observe that $-(\text{fixed point index})\geq \text{$\lambda$-index}$ in case (i), and $-(\text{fixed point index})\geq \text{$\sigma$-index}$ in case (ii).

Let $V=i\lambda+j\sigma$ with $i,j\geq 0$ and assume $2i+j\neq 2$ to avoid the $\dim V=2$ anomaly.  The vanishing requirement is that for each $\ell\in \Z$ the two maps 
\[ \xymatrixcolsep{3.5pc}\xymatrix{? \ar[r]^-{a_\lambda^ia_\sigma^j} & H^{(\ell+1)(i\lambda+j\sigma)-\ell}(\pt) \ar[r]^-{a_\lambda^ia_\sigma^j} & ??
}
\]
are surjective and injective, respectively.  
When $\ell\geq 0$ the middle group is in the first quadrant of the $\M_0$ chart, and comparing to the previous paragraph we will satisfy the index inequalities only when $i=0$ (putting us in case (i) with $P=0$ and $\ell=2Q$), or when $j=0$ (putting us in case (ii) with $2Q-P=0$ and $\ell=P+2Q=4Q$).  In the first of these cases the group is $\Z/2$ generated by $u_{2\sigma}^{\frac{\ell}{2}} a_\sigma^{(\ell+1)j-\ell}$; injectivity  and surjectivity are clear from the chart once one notes that $(\ell+1)j-\ell \geq j\geq 1$.  In the second of the cases,  $\ell=4Q$ and the group is $\Z/3$ generated by $u_{\ls}^{2Q}u_{2\sigma}^Q a_\lambda^{(4Q+1)i-2Q}$.  Here $(4Q+1)i-2Q\geq i\geq 1$ and so it is clear from the chart that the outgoing map is  injective and the incoming map is surjective.

When $\ell=-1$ the middle group is $H^1(\pt)=0$, so the conditions are satisfied.

Finally we deal with the case $\ell\leq -2$.  A little work shows that $\M_0$ is zero in the degree $(\ell+1)(i\lambda+j\sigma)-\ell$: the idea is that since the $\lambda$-index and $\sigma$-index are both non-positive with at least one negative, a nonzero class could only appear in the lower left quadrant of the chart, or maybe on the bottom line of the upper left quadrant.  These possibilities are quickly ruled out except for the single case $V=2\sigma$, $\ell=-2$: this is the only case where the middle group has a nonzero piece in $\M_0$.  But this case is contrary to our hypothesis that $2i+j\neq 2$.

So we only need consider $\M_1$ now.  Degrees where the $\lambda$-index and $\sigma$-index are both non-positive are all in the lower-left quadrant.  The analysis is very similar to the $\ell>0$ case, in that one quickly finds the middle group to be zero except when either $i$ or $j$ is zero, and in those remaining cases one readily sees that the required conditions are always satisfied.    

The conclusion is that the vanishing requirement holds for every representation with $\dim V\neq 2$.  


\bibliographystyle{amsalpha}

\begin{thebibliography}{JTTW}
\bibitem[A]{A} V. Arnold, {\em The cohomology ring of the colored braid group\/}, Mat. Zametki {\bf 5} (1969), 227--231 (Russian).  English translation by V. Vassiliev in {\em Vladimir I. Arnold---Collected Works, Vol II}, Springer-Verlag, Berlin, 2014.

\bibitem[BQV]{BQV} E. Belmont, J.~D. Quigley, Chase Vogeli, {\em Bredon homological stability for configuration spaces of $G$-manifolds\/}, preprint, 2023, arXiv:2311.02459.

\bibitem[C]{C} J.~L. Caruso, {\em Operations in equivariant $\Z/p$-cohomology\/}, Math. Proc. Camb. Phil. Soc. {\bf 126} (1999), 521--541.  

\bibitem[CLM]{CLM} F.~R. Cohen, T.~J. Lada, J.~P. May, {\em The homology of iterated loop spaces}, Lecture Notes in Math. {\bf 533}, Springer-Verlag, Berlin-New York, 1976.  

\bibitem[DPW]{DPW} G. Dorpalen-Barry, N. Proudfoot, and J. Wang, {\em Equivariant cohomology and conditional oriented matroids\/}, 2022 preprint, arXiv:2208.04855.

\bibitem[D1]{D1} D. Dugger, {\em An Atiyah-Hirzebruch spectral sequence for $KR$-theory\/}, $K$-theory {\bf 35} (2005), no. 3-4, 213--256.

\bibitem[D2]{D2} D. Dugger, {\em Coherence for invertible objects and multigraded homotopy rings\/}, Algebr. Geom. Topol. {\bf 14} (2014), no. 2, 1055--1106.

\bibitem[DDI\O]{DDIO} D. Dugger, B.~I. Dundas, D.~C. Isaksen, P.~A. \O stvaer, {\em The multiplicative structures on motivic homotopy groups\/}, 2022 preprint, to appear in Algebr. Geom. Topol.

\bibitem[DHM]{DHM} D. Dugger, C. Hazel, and C. May, {\em $\und{\Z/\ell}$-modules for the cyclic group $C_2$\/}, J. Pure Appl. Algebra {\bf 228} (2024), no. 3, Paper No. 107473, 48 pp.

\bibitem[FN]{FN} E. Fadell and L. Neuwirth, {\em Configuration spaces\/}, Math. Scand. {\bf 10} (1962), 111--118.

\bibitem[FL]{FL} K.~K. Ferland and L.~G. Lewis, {\em The $RO(G)$-graded equivariant ordinary homology of $G$-cell complexes with even-dimensional cells for $G=\Z/p$\/}, Mem. Amer. Math. Soc. {\bf 167} (2004), no. 794, vii+129pp.  


\bibitem[Ge]{Ge} N. Georgakopoulos, {\em The $RO(C_4)$-integral homology of a point\/}, 2019 preprint, arXiv: 1912.06758.

\bibitem[G]{G} E. Getzler, {\em Resolving mixed Hodge modules on configuration spaces\/}, Duke Math. J. {\bf 96} (1999), no. 1, 175--203.

\bibitem[HHR1]{HHR1} M.~A. Hill, M.~J. Hopkins, D.~C. Ravenel, {\em On the nonexistence of elements of Kervaire invariant one\/}, Ann. of Math. (2) {\bf 184} (2016), no. 1, 1--262.

\bibitem[HHR2]{HHR2} M.~A. Hill, M.~J. Hopkins, D.~C. Ravenel, {\em The slice spectral sequence for the $C_4$ analog of real $K$-theory\/}, Forum Math. {\bf 29} (2017), no. 2, 383--447.

\bibitem[HHR3]{HHR3} M.~A. Hill, M.~J. Hopkins, D.~C. Ravenel, {\em The slice spectral sequence for certain $RO(C_{p^n})$-graded suspensions of $H\mZ$\/}, Bol. Soc. Mat. Mex. (3) {\bf 23} (2017), no. 1, 289--317.

\bibitem[HHR4]{HHR4} M.~A. Hill, M.~J. Hopkins, D.~C. Ravenel, {\em Equivariant stable homotopy theory and the Kervaire invariant problem\/}, New. Math. Monogr. {\bf 40}, Cambridge University Press, Cambridge, 2021.  

\bibitem[LS]{LS} G.~I. Lehrer and L. Solomon, {\em On the action of the symmetric group on the cohomology of the complement of its reflecting hyperplanes\/}, J. Algebra {\bf 104} (1986), no. 2, 410--424.

\bibitem[L]{L} L.~G. Lewis, {\em The $RO(G)$-graded equivariant ordinary cohomology of complex projective spaces with linear $\Z/p$-actions\/},
Algebraic topology and transformation groups (G\"ottingen, 1987), 53-122.  Lecture Notes in Math. {\bf 1361}, Springer-Verlag Berlin, 1988.

\bibitem[LF]{LF}  P. Lima-Filho, {\em On the equivariant homotopy of free abelian groups on $G$-spaces and
   $G$-spectra\/}, Math. Z. {\bf 224} (1997), no. 4, 567--601.

\bibitem[M]{M} J.~P. May, {\em Equivariant homotopy and cohomology theory\/}, with contributions by M. Cole, G. Comezana, S. Costenoble, A.~D. Elmendorff, J.~P.~G. Greenlees, L.~G. Lewis Jr., R.~J. Piacenza, G. Triantafillou, and S. Waner.  CBMS Regional Conf. Ser. in Math. {\bf 91}, Published for the Conference Board of Mathematical Sciences, Washington D.C., by the American Mathematical Society, Providence, RI, 1996.

\bibitem[Mo]{Mo} D. Moseley, {\em Equivariant cohomology and the {V}archenko-{G}elfand filtration\/},
   J. Algebra {\bf 472} (2017), 95--114.

\bibitem[MPY]{MPY} D. Moseley, N. Proudfoot, and B. Young, {\em
The {O}rlik-{T}erao algebra and the cohomology of
              configuration space},
              Exp. Math. {\bf 26} (2017), no. 3, 373--380.
              
\bibitem[P]{P} N. Proudfoot, {\em The equivariant Orlik-Solomon algebra\/}, J. Algebra {\bf 305} (2006), no. 2, 1186--1196.

\bibitem[dS]{dS} P.~F. dos Santos, {\em A note on the equivariant Dold-Thom theorem\/}, J. Pure Appl. Algebra {\bf 183} (2003), no. 1--3, 299--312.

\bibitem[S]{S} D.~P. Sinha, {\em The (non-equivariant) homology of the little disks operad\/},  OPERADS 2009, 253--279, S\'emin. Congr., {\bf 26}, Soci\'et\'e Math\'ematique de France, Paris, 2013.

\bibitem[VG]{VG} A.~N. Varchenko and I.~M. Gel'fand, {\em Heaviside functions of a configuration of hyperplanes\/}, Functional Analysis and its Applications {\bf 21} (1987), 255--270.

\bibitem[Y1]{Y1} G. Yan, {\em The ring structure of $RO(\Sigma_3)$-graded homotopy of $H\mZ$\/}, 2021 preprint, {\tt https://sites.google.com/view/guoqiyan?usp=sharing}

\bibitem[Y2]{Y2} G. Yan, {\em The $RO(C_{2^n})$-graded homotopy of $H\mZ$ through generalized Tate squares\/}, 2022 preprint, arXiv:2208.14368.

\bibitem[Z]{Z} M. Zeng, {\em Equivariant Eilenberg-MacLane spectra in cyclic $p$-groups\/}, preprint, 2017, arXiv:1710.01769.

\end{thebibliography}

\end{document}